\documentclass[a4paper,DIV=9]{scrartcl}
\pdfoutput=1

\usepackage[utf8]{inputenc}
\usepackage[T1]{fontenc}
\addtokomafont{disposition}{\rmfamily}              %
\usepackage[british]{babel}
\usepackage{csquotes}
\usepackage[UKenglish,nodayofweek]{datetime}
\date{}

\usepackage[pdfusetitle,bookmarksnumbered=true,unicode]{hyperref}
\usepackage{xcolor}
\hypersetup{
  colorlinks,
  linkcolor={red!45!black},
  citecolor={green!45!black},
  urlcolor={blue!45!black}
}

\usepackage{lmodern}                                %
\usepackage{amssymb,amsmath,mathtools}              %
\usepackage{stmaryrd}
\usepackage[osf,sc]{mathpazo}                       %
\usepackage{ellipsis}                               %
\usepackage[babel]{microtype}                       %
\usepackage{amsthm}                                 %
\recalctypearea

\usepackage[scr=boondoxo, scrscaled=.98]{mathalfa}  %
\usepackage{tikz-cd}
\usetikzlibrary{quotes,babel,angles,arrows}
\usepackage{tikz}
\usepgflibrary{arrows}
\usetikzlibrary{shapes.arrows,trees,backgrounds,decorations.markings,shadows}

\usepackage[backref,%
            style=alphabetic,%
            backend      = biber,%
            maxbibnames  = 9,%
            maxcitenames = 9,%
            bibencoding  = utf8]{biblatex}
\DeclareFieldFormat[article]{volume}{\textbf{#1}}       %
\DeclareFieldFormat[misc]{title}{‘#1’}                  %
\DeclareFieldFormat[book]{number}{\textbf{#1}}          %
\DeclareFieldFormat[inproceedings]{number}{\textbf{#1}} %
\AtBeginBibliography{\small}
\DefineBibliographyStrings{english}{
  backrefpage={$\uparrow$},
  backrefpages={$\uparrow$}
}

\DeclareSymbolFont{eulargesymbols}{U}{zeuex}{m}{n}
\DeclareMathSymbol{\intop}{\mathop}{eulargesymbols}{"52}

\usepackage{graphicx}
\usepackage{overpic}
\usepackage[margin=10pt,font=small,labelfont=bf,labelsep=space,format=plain]{caption}

\usepackage{relsize}
\newcommand{\acr}[1]{\textsmaller{#1}}

\deffootnote{1.2em}{1em}{%
\makebox[1.2em][r]{\textsuperscript{\thefootnotemark~}}%
}

\theoremstyle{plain}
\newtheorem{theo} {Theorem}[section]
\newtheorem{lem}  [theo]{Lemma}
\newtheorem{prop} [theo]{Proposition}

\newtheorem{conj} [theo]{Conjecture}

\theoremstyle{definition}
\newtheorem{defi}  [theo]{Definition}

\newtheorem{rem}   [theo]{Remark}
\newtheorem{expl}  [theo]{Example}
\newtheorem{constr}[theo]{Construction}
\newtheorem{nota}  [theo]{Notation}

\newtheorem{setting}[theo]{Setting}

\newenvironment{addr}{
~\\[-10px]
\begingroup
\scriptsize
\begin{flushright}
}{
\end{flushright}
\endgroup
}

\newcommand{\auth}[5]{
\begin{minipage}[t]{#1}
\scriptsize
{\footnotesize\textbf{#2}}\\[3px]
\scriptsize #4\\[3px]
\texttt{\href{mailto:#5}{#5}}
\end{minipage}
}

\newlength{\fixmidfigure}
\newenvironment{tfigure}{
  \setlength{\fixmidfigure}{\lastskip}\addvspace{-\lastskip}
  \begin{figure}[t]
  }{
  \end{figure}
  \addvspace{\fixmidfigure}
}

\let\setminus\smallsetminus
\let\backslash\smallsetminus
\let\le\leqslant
\let\ge\geqslant
\let\leq\leqslant
\let\geq\geqslant
\let\emptyset\varnothing
\let\phi\varphi
\let\epsilon\varepsilon
\let\theta\vartheta

\newcommand{\on}     [1]{\text{\textup{#1}}}
\newcommand{\ul}     [1]{\underline{#1}}

\newcommand{\quot}   [2]{\left.\raise0.5ex\hbox{$#1$} \right/\hspace*{-2px}\lower0.5ex\hbox{$#2$}}

\newcommand{\pa}[1]{\mathopen{}\left(#1\right)\mathclose{}}     %
\newcommand{\abs}[1]{\lvert #1\rvert}                           %
\newcommand{\set}[1]{\mathopen{}\left\{#1\right\}\mathclose{}}  %

\newcommand{\N}{\mathbb{N}}     %
\newcommand{\R}{\mathbb{R}}     %
\newcommand{\Z}{\mathbb{Z}}     %

\newcommand{\tZ}{\tilde{Z}}  %

\newcommand{\bbD}{\mathbb{D}}

\newcommand{\bbI}{\mathbb{I}}

\newcommand{\bbO}{\mathbb{O}}
\newcommand{\bbP}{\mathbb{P}}
\newcommand{\bbQ}{\mathbb{Q}}

\newcommand{\bbZ}{\mathbb{Z}}

\newcommand{\bZ}{\mathbb{Z}}

\newcommand{\caS}{\mathcal{S}}

\newcommand{\scA}{\mathscr{A}}
\newcommand{\scB}{\mathscr{B}}

\newcommand{\scD}{\mathscr{D}}

\newcommand{\scG}{\mathscr{G}}

\newcommand{\scM}{\mathscr{M}}

\newcommand{\scO}{\mathscr{O}}
\newcommand{\scP}{\mathscr{P}}
\newcommand{\scQ}{\mathscr{Q}}

\newcommand{\bfC}{\mathbf{C}}

\newcommand{\bfI}{\mathbf{I}}

\newcommand{\bfM}{\mathbf{M}}

\ifdefined\bfq\renewcommand{\bfq}{\mathbf{q}}\else \newcommand{\bfq}{\mathbf{q}}\fi

\newcommand{\frC}{\mathfrak{C}}

\newcommand{\frH}{\mathfrak{H}}
\newcommand{\frI}{\mathfrak{I}}

\newcommand{\frM}{\mathfrak{M}}

\newcommand{\frS}{\mathfrak{S}}

\newcommand{\frp}{\mathfrak{p}}

\newcommand{\frv}{\mathfrak{v}}

\newcommand{\ob}{\on{ob}}       %

\newcommand{\tref}[2]{#1~\ref{#2}}
\newcommand{\iref}[2]{#1~\ref{#2}}

\newcommand{\qq}{\sslash}

\newcommand{\fS}{\mathfrak{S}} %

\newcommand{\bm}[1]{\mathbold{#1}}

\newcommand{\fg}{\mathfrak{g}} %

\newcommand{\one}{\mathbb{1}} %

\newcommand{\del}{\partial}
\renewcommand{\phi}{\varphi}
\renewcommand{\epsilon}{\varepsilon}

\DeclareMathOperator{\UConf}{C}
\DeclareMathOperator{\Diff}{Diff}

\newcommand{\BrG}{\mathrm{Br}}
\newcommand{\op}{{\operatorname{op}}}
\newcommand{\Inj}{\mathbf{Inj}}
\newcommand{\bSig}{\mathbf{\Sigma}}

\newcommand{\Top}{\mathbf{Top}}
\newcommand{\bEn}{\mathbb{1}}

\newcommand{\Th}[1]{\textsuperscript{#1}}

\newcommand{\bmG}{{\bm{G}}}
\newcommand{\bmX}{{\bm{X}}}

\newcommand*\lon{%
     \nobreak
     \mskip6mu plus1mu
     \mathpunct{}%
     \nonscript
     \mkern-\thinmuskip
     {:}%
     \mskip2mu
     \relax
}

\newcommand{\Alg}[1]{#1\text{-}\mathbf{Alg}}
\let\binom\tbinom

\definecolor{dgrey}{rgb}{.4,.4.,.4}
\definecolor{bgreen}{rgb}{.7,.9.,.6}
\definecolor{dyellow}{rgb}{.85,.7,0}
\definecolor{byellow}{rgb}{.95,.93,.4}
\definecolor{dred}{rgb}{.9,.3,.3}
\definecolor{dblue}{rgb}{.3,.3,.8}
\definecolor{bred}{rgb}{.95,.6,.6}
\definecolor{bblue}{rgb}{.5,.5,.9}
\definecolor{dgrey}{rgb}{.6,.6,.6}
\definecolor{lila}{rgb}{.8,.6,.9}
\definecolor{dgreen}{rgb}{.25,.7,.3}
\definecolor{dred}  {rgb}{.7,.3,.25}
\definecolor{bbblue}{rgb}{.7,.75,.9}
\definecolor{bgrey}{rgb}{.7,.7,.7}
\definecolor{lila}{rgb}{.8,.5,.8}

\newcommand{\new}[1]{#1}

\newcommand{\LF}{\tilde F}
\newcommand{\tB}{\tilde B}

\newcommand{\tbO}{\tilde\bbO}
\newcommand{\tbD}{\tilde\bbD}
\newcommand{\tF}{\tilde F}

\newcommand{\swr}{\kern.4px\wr\kern.4px}
\newcommand{\fComm}{\includegraphics[height=6.4pt]{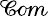}}

\bibliography{Bibliography.bib}

\title{Parametrised moduli spaces of\\surfaces as infinite loop spaces}
\author{Andrea Bianchi, Florian Kranhold,\\and Jens Reinhold}

\begin{document}

\maketitle

\begin{abstract}
  We study the $E_2$-algebra $\Lambda\frM_{*,1}\coloneqq\coprod_{g\ge0}\Lambda\frM_{g,1}$
  consisting of free loop spaces of moduli spaces of Riemann surfaces with one
  parametrised boundary component, and compute the homotopy type of the group
  completion $\Omega B\Lambda\frM_{*,1}$: it is the product of
  $\Omega^{\infty}\mathbf{MTSO}(2)$ with a certain free $\Omega^{\infty}$-space
  depending on the family of all boundary-irreducible mapping classes in all
  mapping class groups $\Gamma_{g,n}$ with $g\ge0$ and $n\ge1$.\par{}\vspace*{8px}
  \footnotesize\noindent%
  \textbf{Date.} 12\textsuperscript{th} May, 2021.
  \textbf{Last change.} \today.\\[2px]
  \textbf{Key words.}  mapping class group, moduli space, free loop space,
  infinite loop space,\\[-.5px]\hspace*{1em} mapping space, arc system,
  coloured operad, homological stability.\\[2px]
  \textbf{2020 \acr{MSC}.}
  \texttt{18M60}, %
  \texttt{55P48}, %
  \texttt{57K20}, %
  \texttt{55P47}, %
  \texttt{55P15}, %
  \texttt{55R35}, %
  \texttt{20E45}, %
  \texttt{20E22}, %
  \texttt{55P50}.\vspace*{3px}
\end{abstract}

\section{Introduction}
The Madsen–Weiss theorem \cite{MadsenWeiss} can be formulated as follows: let
$\frM_{g,1}$ denote the moduli space of Riemann surfaces of genus $g\ge 0$ with
one para\-metrised boundary curve. By \cite{miller} and \cite{cfb2}, the
collection
\[\frM_{*,1}\coloneqq\coprod_{g\ge0}\frM_{g,1}\]
admits the structure of an $E_2$-algebra, more precisely an algebra over the
little 2-discs operad $\scD_2$. Madsen and Weiss identify the group completion
$\Omega B\frM_{*,1}$ with the infinite loop space
$\Omega^{\infty}\mathbf{MTSO}(2)$, where $\mathbf{MTSO}(2)$ is the
two-dimensional oriented tangential Thom spectrum \cite{GMTW}.

One can consider the analogous problem with $\frM_{*,1}$ replaced by the mapping
space $\on{map}(X,\frM_{*,1})$. This space is again a $\scD_2$-algebra by
pointwise composition and it is our goal to understand its group completion
$\Omega B\,\on{map}(X,\frM_{*,1})$. Note that the $\scD_2$-algebra structure
extends to an algebra structure over Tillmann's surface operad built out of
$\frM_{*,1}$, so the main theorem of \cite{Tillmann00} implies the homotopy type
we wish to understand is an infinite loop space.

In this article we will focus on the simplest non-trivial case $X=S^1$, i.e.\ we
consider the free loop space $\Lambda\frM_{*,1}\coloneqq\on{map}(S^1,\frM_{*,1})$;
we will briefly discuss in the appendix the general case, which is very similar.

For any discrete group $\Gamma$, one can identify
$\Lambda B\Gamma \simeq \coprod_{[\gamma] \in \on{Conj}(\Gamma)}
BZ(\gamma,\Gamma)$, where $\on{Conj}(\Gamma)$ denotes the set of conjugacy
classes of $\Gamma$, and $Z(\gamma,\Gamma)$ is the centraliser of
$\gamma \in \Gamma$. Note also that the isomorphism type of the group
$Z(\gamma,\Gamma)$ only depends on the conjugacy class of $\gamma\in\Gamma$.
The problem we address in this paper is strongly related to analysing the
structure of centralisers of elements of mapping class groups: indeed, recall
that $\frM_{g,1}$ is a classifying space for the mapping class group
$\Gamma_{g,1}$ of a smooth oriented surface of genus $g$ with one parametrised
boundary curve; we then have a homotopy equivalence
\[\Lambda\frM_{*,1}\simeq \coprod_{g\ge 0}\Lambda B\Gamma_{g,1}
  \simeq \coprod_{g\ge 0}\coprod_{[\varphi]\in\on{Conj}(\Gamma_{g,1})}
  BZ(\varphi,\Gamma_{g,1}).\]

\paragraph{Results}
The free loop space of the moduli space of surfaces of genus $g$ with $n$
parametrised boundary circles, $\Lambda\frM_{g,n}$, admits an action by the
isometry group of the disjoint union of $n$ oriented circles, i.e.\ by
$T^n \rtimes \frS_n=(S^1)^n\rtimes \frS_n$.

We introduce an irreducibility criterion for mapping classes that is invariant
under conjugation. We then consider, for any $n\ge 1$ and $g \ge 0$, the
subspace $\mathfrak C_{g,n} \subseteq \Lambda\frM_{g,n}$ of free loops whose
corresponding conjugacy classes of elements in
$\pi_0 (\frM_{g,n}) \cong \Gamma_{g,n}$ are irreducible. The pointwise action of
$T^n \rtimes \frS_n$ on $\Lambda\frM_{g,n}$ restricts to an action on
$\mathfrak C_{g,n}$, and the main result of this work is the following
identification, where $\qq$ stands for homotopy quotient.

\begin{theo}\label{theo:main} There is a weak homotopy equivalence
  \[\Omega B\Lambda\frM_{*,1}\simeq \Omega^\infty\mathbf{MTSO}(2)
    \times \Omega^\infty\Sigma^\infty_+\coprod_{n\ge 1}
    \coprod_{g\ge 0} \mathfrak C_{g,n}  \qq (T^n \rtimes \frS_n)
    \label{eq:a}\]
\end{theo}  

The key to proving this theorem is a good understanding of mapping class groups
of surfaces (also with more than one boundary component) as well as an extension
of classical operadic techniques to a coloured setting. As a first step, we
prove a structure result for centralisers of mapping classes in $\Gamma_{g,n}$
which might be of independent interest: see
\tref{Proposition}{prop:Zdelphistructure}.

Second, we develop a machinery for $N$-coloured operads with homological
stability $\scO$ containing a suboperad $\scP$, such as a family of topological
groups: the group completion of the derived relatively free algebra
$\tF_{\scP}^{\scO}(\bm{X})$ over a $\scP$-algebra $\bm{X}$ is computed
colourwise as an infinite loop space, under suitable assumptions on $\scO$ and
$\scP$, see \tref{Theorem}{theo:splitOHS}. This part of the work is a
generalisation of \cite{Tillmann00} and \cite{BasterraEtAl} to the coloured and
relative setting.

The two ingredients are put together by proving that $\Lambda \frM_{*,1}$ is the
colour-$1$ part of a relatively free algebra over a coloured version $\scM$ of
Tillmann’s surface operad, relative to a suboperad built out of $T^n
\rtimes \frS_n$; the ‘relative generators’ are precisely the spaces
$\mathfrak C_{g,n}$ mentioned above: see \tref{Theorem}{theo:isfreealg}.

\paragraph{Related work}
One approach to studying classifying spaces of diffeomorphism groups pertains to
the notion of cobordism categories. It was pioneered with the breakthrough
theorem by Madsen and Weiss and refined by Galatius, Madsen, Tillmann, and Weiss
\cite{GMTW}.

Recall that, in the orientable setting, the cobordism category $\mathrm{Cob}_d$
is a topological category, with object space given by the union of all moduli
spaces of closed, oriented $(d-1)$-manifolds, and morphism space given by the
union of all moduli spaces of compact, oriented $d$-manifolds with incoming and
outgoing boundary.

It is natural to study two related generalisations of $\mathrm{Cob}_d$,
in an \emph{equivariant} and a \emph{parametrised} direction:
\begin{enumerate}
\item for a (topological) group $G$ we can consider the $G$-equivariant
  cobordism category $\mathrm{Cob}_d^G$: objects and morphisms are, respectively,
  $(d-1)$- and $d$-manifolds endowed with an (continuous) action of $G$ by
  orientation-preserving diffeomorphisms;
\item for a topological space $Y$ we can consider the $Y$-parametrised
  cobordism category $\mathrm{Cob}_d(Y)$: objects and morphisms are, respectively,
  orientable $(d-1)$- and $d$-manifold bundles over $Y$.
\end{enumerate}
In the case $G=\Z$ and $Y=S^1$ there is a continuous functor
$\mathrm{Cob}_d^{\Z}\to \mathrm{Cob}_d(S^1)$, given by taking mapping tori:
using that every smooth bundle over $S^1$ is induced from a diffeomorphism, this
functor is in fact a levelwise equivalence. For more general groups $G$ and
$Y = BG$, the analogous argument pertaining to $G$-actions and bundles over $BG$
can fail: see \cite{Jens} for a discussion of this phenomenon and
counterexamples in case $G = \text{SU}(2)$.

Our work can be seen as a contribution toward understanding the homotopy type of
$\mathrm{Cob}_d^{\Z}\simeq \mathrm{Cob}_d(S^1)$: gluing a pair of pants gives rise to a
map of monoids \mbox{$\Lambda\frM_{*,1}\to \mathrm{Cob}_2(S^1)|_{S^1}$}, where
$\mathrm{Cob}_2(S^1)|_{S^1}$ denotes the full subcategory of
$\mathrm{Cob}_2(S^1)$ on a single object represented by a trivial $S^1$-bundle
over $S^1$.

In the non-parametrised setting, the composition
$\frM_{*,1}\to \mathrm{Cob}_2|_{S^1}\to \mathrm{Cob}_2$ is known to induce an
equivalence after taking classifying spaces; we hope that in a similar way, the
understanding of $B\Lambda\frM_{*,1}$ can shed some light on the homotopy type
of $B\mathrm{Cob}_2^{\Z} \simeq B\mathrm{Cob}_2(S^1)$ in future work.

In the case of a finite group $G$, the homotopy type of $\mathrm{Cob}_d^{G}$ was
recently determined by Szűcs and Galatius \cite{Gergely}.  In work by Raptis and
Steimle \cite{RaptisSteimle}, parametrised cobordism categories
$\mathrm{Cob}_d(Y)$ featured as a tool to prove index theorems; however, it was
not necessary for the scopes of that work to describe the homotopy type of the
classifying spaces of these categories.  The much older work of Kreck on
bordisms of diffeomorphisms \cite{Kreck} can be seen as a description of
$\pi_0\pa{\mathrm{Cob}_d(S^1)}$.

\paragraph{Outline}
In \tref{Section}{sec:arcSystems}, we recall Alexander's method concerning arc
systems. We apply these concepts to associate with a mapping class
$\phi\in\Gamma_{g,n}$ a canonical decomposition of $\Sigma_{g,n}$ along a system
of simple closed curves, called the \emph{cut locus} of $\phi$.  The goal of
\tref{Section}{sec:centralisers} is a detailed understanding of centralisers of
mapping classes: this uses the canonical decomposition described in the previous
section in a crucial way.\looseness-1

In \tref{Section}{sec:colouredOHS}, we recall some basic definitions and
constructions related to coloured operads, and introduce the coloured surface
operad $\scM$. In \tref{Section}{sec:infiniteLSFromOHS}, we introduce the notion
of a coloured operad with homological stability and prove a levelwise splitting
result in the spirit of \cite{BasterraEtAl}, which applies in particular to
$\scM$.  Finally, in \tref{Section}{sec:LM1asFreeAlg}, we show that
$\Lambda\frM_{*,1}$ is the colour-1 part of a relatively free $\scM$-algebra,
which in combination with the splitting result concludes the proof of
\tref{Theorem}{theo:main}.

We briefly discuss in \tref{Appendix}{apx:generalX} the analogue of
\tref{Theorem}{theo:main} for a general parametrising space $X$ (see
\tref{Theorem}{theo:maingeneralised}) as well as a weak form of naturality in
$X$ of the equivalence (see \tref{Theorem}{theo:maingeneralisedfunctorial}); in
\tref{Appendix}{apx:braidSpaces}, we address two similar problems concerning
group completion of free loop spaces, related to braid groups and symmetric
groups, respectively.

\paragraph{Acknowledgements}
We are indebted to Søren Galatius for suggesting the problem treated in this
paper to the third named author, who would also like to thank Rachael Boyd and
Adva Wolf for helpful discussions that were part of a first attempt to solve
it. Moreover, we are grateful to Carl-Friedrich Bödigheimer, Søren Galatius, and
Oscar Randal-Williams for fruitful conversations.

We would also like to thank Manuel Krannich, Ulrike Tillmann, and Nathalie Wahl
for useful comments they made on a first draft. Finally, we are deeply grateful
to the anonymous referee for a detailed list of corrections on the first
version, and for suggesting a way to sensibly simplify and at the same time
generalise the discussion in \tref{Section}{sec:infiniteLSFromOHS}.

\paragraph{Funding}
Andrea Bianchi was partially supported by the \emph{Deutsche
  Forschungsgemeinschaft} (\acr{DFG}, German Research Foundation) under Germany’s
Excellence Strategy (\texttt{EXC-2047/1}, \texttt{390685813}), by the
\emph{European Research Council} under the European Union’s Seventh Framework
Programme (\texttt{ERC StG 716424 – CASe}, PI Karim Adiprasito), and by the
\emph{Danish National Research Foundation} through the \emph{Copenhagen Centre
  for Geometry and Topology} (\texttt{DNRF151}).

Florian Kranhold was supported by the \emph{Max Planck Institute for
  Mathematics} in Bonn, by the \emph{Deutsche Forschungsgemeinschaft} (\acr{DFG},
German Research Foundation) under Germany’s Excellence Strategy
(\texttt{EXC-2047/1}, \texttt{390685813}) and by the \emph{Promotionsförderung}
of the \emph{Stu\-dien\-stif\-tung des Deutschen Volkes}.

Jens Reinhold was supported by the \emph{Deutsche Forschungsgemeinschaft} (\acr{DFG},
German Research Foundation) -- \texttt{SFB 1442 427320536},
Geometry:~\emph{Deformations and Rigidity}, as well as under Germany's
Excellence Strategy \texttt{EXC-2044}, Mathematics Münster:
\emph{Dynamics – Geometry – Structure}.

\section{Arc systems and the cut locus}
\label{sec:arcSystems}
The aim of this and the next section is to study centralisers in mapping class
groups of surfaces. This interest is motivated by the following observation: for
$g\ge1$ the space $\Lambda \frM_{g,1}\simeq \Lambda B\Gamma_{g,1}$ has one
connected component for each conjugacy class $[\phi]\in\on{Conj}(\Gamma_{g,1})$;
this component is homotopy equivalent to $B Z(\phi,\Gamma_{g,1})$, where we
denote by $Z(\phi,\Gamma_{g,1})\subset\Gamma_{g,1}$ the centraliser of $\phi$ in
$\Gamma_{g,1}$, i.e.\ the subgroup of all mapping classes $\psi\in\Gamma_{g,1}$
commuting with $\phi$.

In this section, we will first introduce some notation for surfaces and mapping
class groups, and then define the cut locus of a mapping class.

\subsection{Surfaces and mapping class groups}
We work in the entire article with smooth, oriented surfaces and
orientation-preserving diffeomorphisms of surfaces.
\begin{nota}
  \label{nota:surface}
  We usually denote by $\caS$ a smooth, compact, oriented, possibly disconnected
  surface, such that each component of $\caS$ has non-empty boundary; we denote
  the boundary of $\caS$ by $\partial \caS\subset \caS$.
  
  The boundary $\del\caS$ is equipped with a decomposition
  $\del\caS=\del^{\text{in}}\caS\sqcup\del^{\text{out}}\caS$, into unions of
  connected components: the \emph{incoming} boundary $\del^{\text{in}}\caS$ is
  allowed to be empty, whereas each component of $\caS$ is required to intersect
  the \emph{outgoing} boundary in at least one curve; see
  \tref{Figure}{fig:surface} for an example.
  
  Both parts of the boundary are equipped with an \emph{ordering} and a
  \emph{parametrisation}: i.e.\ there are preferred diffeomorphisms
  $\vartheta^{\text{in}}\colon \{1,\dotsc,n\}\times S^1\to \partial^{\text{in}}
  \caS$ and
  $\vartheta^{\text{out}}\colon \{1,\dotsc,n'\}\times S^1\to
  \partial^{\text{out}} \caS$, where $n=\#\pi_0(\del^{\text{in}}\caS)$ and
  $n'=\#\pi_0(\del^{\text{out}}\caS)$.
  
  Note that each boundary component $c\subset\del\caS$ is endowed with two
  natural orientations: the first is induced from the orientation of $\caS$,
  i.e. it is the unique orientation of $c$ that, concatenated with a vector
  field along $c$ pointing out of $\caS$, returns the orientation of $\caS$; the
  second orientation comes from the parametrisation of $c$.  For an
  \emph{incoming} boundary component $c\subset\del^{\text{in}}\caS$ we require
  that these two orientations coincide, whereas for an \emph{outgoing} boundary
  component $c\subset\del^{\text{out}}\caS$ we require that these two
  orientations differ.
\end{nota}
We denote a surface by $(\caS,\theta)$, or shortly by $\caS$ when it is
not necessary to mention the parametrisation of the boundary; here $\theta$ is
the map \mbox{$\set{1,\dots,n+n'}\times S^1\to \del\caS$} obtained by concatenation of
$\theta^{\text{in}}$ and $\theta^{\text{out}}$.\vspace*{6px}

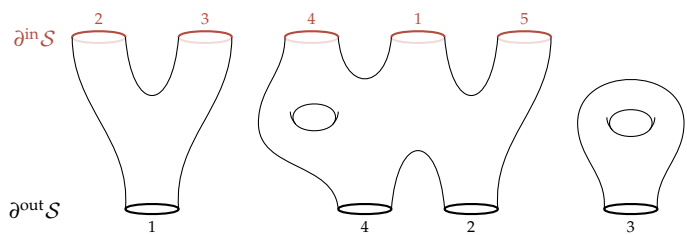
\begin{figure}[h]
  \centering
  \begin{tikzpicture}[xscale=.7,yscale=.76]
    \draw[looseness=.35,thick,dred] (0,3) to[out=-90,in=-90] ++(1,0) to[out=90,in=90] (0,3);
    \draw[looseness=.35,thick,dred] (2,3) to[out=-90,in=-90] ++(1,0) to[out=90,in=90] (2,3);
    \draw[looseness=.35,thick,dred] (4,3) to[out=-90,in=-90] ++(1,0) to[out=90,in=90] (4,3);
    \draw[looseness=.35,thick,dred] (6,3) to[out=-90,in=-90] ++(1,0) to[out=90,in=90] (6,3);
    \draw[looseness=.35,thick,dred] (8,3) to[out=-90,in=-90] ++(1,0) to[out=90,in=90] (8,3);
    \fill[white, opacity=.8] (0,2.5) rectangle (11,3);   
    \draw[looseness=.3,thick] (1,0) to[out=-90,in=-90] ++(1,0) to[out=90,in=90] (1,0);
    \draw[looseness=.3,thick] (5,0) to[out=-90,in=-90] ++(1,0) to[out=90,in=90] (5,0);
    \draw[looseness=.3,thick] (7,0) to[out=-90,in=-90] ++(1,0) to[out=90,in=90] (7,0);
    \draw[looseness=.3,thick] (10,0) to[out=-90,in=-90] ++(1,0) to[out=90,in=90] (10,0);
    \draw (1,0) to[out=85,in=-90]  (0,3);
    \draw (2,0) to[out=95,in=-90]  (3,3);
    \draw[looseness=3.5] (1,3) to[out=-85,in=-95] (2,3);
    \draw[looseness=2.5] (5,3) to[out=-85,in=-95] (6,3);
    \draw[looseness=3.5] (7,3) to[out=-85,in=-95] (8,3);
    \draw[looseness=3.5] (6,0) to[out=85,in=95]   (7,0);
    \draw (4,3) to [out=-90,in=90] ++(-.5,-1.5) to [out=-90,in=95] (5,0);
    \draw (8,0) to[out=95,in=-90]  (9,3);
    \draw (10,0) to[out=85,in=-90] (9.5,1.5);
    \draw (11,0) to[out=95,in=-90] (11.5,1.5);
    \draw[looseness=1.3] (9.5,1.5) to[out=90,in=90] (11.5,1.5);
    \draw[looseness=1] (10.1,1.5) to[out=-90,in=-90] ++(.8,0) to[out=90,in=90] (10.1,1.5);
    \draw[thin] (10.15,1.38) to[out=130,in=-90] (10.05,1.6);
    \draw[thin] (10.85,1.38) to[out=50 ,in=-90] (10.95,1.6);
    \draw[looseness=1] (4.15,1.6) to[out=-90,in=-90] ++(.8,0) to[out=90,in=90] (4.15,1.6);
    \draw[thin] (4.2,1.48) to[out=130,in=-90] (4.1,1.7);
    \draw[thin] (4.9,1.48) to[out=50 ,in=-90] (5,1.7);
    \node[anchor=base,dred] at (.5,3.23) {\tiny $2$};
    \node[anchor=base,dred] at (6.5,3.23) {\tiny $1$};
    \node[anchor=base,dred] at (2.5,3.23) {\tiny $3$};
    \node[anchor=base,dred] at (4.5,3.23) {\tiny $4$};
    \node[anchor=base,dred] at (8.5,3.23) {\tiny $5$};
    \node[dred] at (-.7,3) {\scriptsize $\partial^{\mathrm{in}}\caS$};
    \node at (-.7,0) {\scriptsize $\partial^{\mathrm{out}}\caS$};
    \node[anchor=base] at (1.5,-.4) {\tiny $1$};
    \node[anchor=base] at (5.5,-.4) {\tiny $4$};
    \node[anchor=base] at (7.5,-.4) {\tiny $2$};
    \node[anchor=base] at (10.5,-.4) {\tiny $3$};
  \end{tikzpicture}
  \caption{A surface $\caS$ with $5$ incoming and $4$ outgoing boundary curves.}
  \label{fig:surface}
\end{figure}

\begin{defi}
\label{defi:presbparam}
Let $\Phi\colon(\caS,\theta)\to(\caS',\theta')$ be an orientation-preserving
diffeomorphism of surfaces. We say that $\Phi$ \emph{preserves the boundary
  parametrisation} if the following conditions hold:
\begin{itemize}
\item $\Phi$ restricts to diffeomorphisms
  $\Phi\colon\del^{\text{in}}\caS\overset{\cong}{\to}\del^{\text{in}}\caS'$ 
  and $\Phi\colon\del^{\text{out}}\caS\overset{\cong}{\to}\del^{\text{out}}\caS'$.
\item If $n\coloneqq \#\pi_0(\del^{\text{in}}\caS)=\#\pi_0(\del^{\text{in}}\caS')$
  and $n'\coloneqq\#\pi_0(\del^{\text{out}}\caS)=\#\pi_0(\del^{\text{out}}\caS')$, then
  there exist permutations $\sigma^{\text{in}}\in \frS_n$ and
  $\sigma^{\text{out}}\in \frS_{n'}$ such that
  \begin{itemize}
  \item $(\Phi\circ\vartheta^{\text{in}})(j,\zeta) = (\vartheta')^{\text{in}}(\sigma^{\text{in}}(j),\zeta)$
    for each $1\le j\le n$ and $\zeta\in S^1$;
  \item $(\Phi\circ\vartheta^{\text{out}})(j,\zeta) = (\vartheta')^{\text{out}}(\sigma^{\text{out}}(j),\zeta)$
    for each $1\le j\le n'$ and $\zeta\in S^1$.
  \end{itemize}
\end{itemize}
\end{defi}
Note that in the previous definition we do not require that $\Phi$ also
preserves the orderings of the incoming and outgoing components of $\del\caS$
and $\del\caS'$, i.e.\ the permutations $\sigma^{\text{in}}\in \frS_n$ and
$\sigma^{\text{out}}\in \frS_{n'}$ may be non-trivial. To emphasise this we
distinguish between the words ‘ordering’ and ‘parametrisation’. In
\iref{Section}{sec:colouredOHS}, when introducing the coloured operad $\scM$, we
will also consider surfaces equipped with a parametrisation of \emph{collar
  neighbourhoods} of the incoming and the outgoing boundary.

\begin{nota}
  For all $g\ge0$ and $n\ge1$, we fix a model surface $\Sigma_{g,n}$: it is a
  connected surface of genus $g$ with $n$ outgoing and no incoming boundary
  components. We say that $\caS$ is of type $\Sigma_{g,n}$ if there exists a
  diffeomorphism $\caS\to\Sigma_{g,n}$ preserving the boundary parametrisation.
\end{nota}

\begin{defi}
  \label{defi:mcg}
  The \emph{mapping class group} $\Gamma(\caS,\del\caS)$ is the group of isotopy
  classes of diffeomorphisms $\Phi\colon \caS\to \caS$ that fix the boundary
  pointwise, i.e.\ $\Phi\circ\vartheta=\vartheta$. Such a $\Phi$ is called a
  \emph{diffeomorphism of $(\caS,\partial \caS)$}.  For $\caS=\Sigma_{g,n}$, we
  also write $\Gamma_{g,n}$ for $\Gamma(\caS,\del\caS)$. We usually denote
  isotopy classes by small Greek letters $\varphi$, and use capital Greek
  letters for diffeomorphisms.
\end{defi}

\begin{rem}
  \label{rem:Xiconj}
  Note that any diffeomorphism $\Xi\colon\caS\to\caS'$ induces an identification
  of the groups $\Gamma(\caS,\del\caS)\cong\Gamma(\caS',\del\caS')$ by
  conjugation with $\Xi$: the mapping class $\phi\in \Gamma(\caS,\del\caS)$,
  represented by the diffeomorphism $\Phi$, corresponds to the mapping class
  $\phi^\Xi\in \Gamma(\caS',\del\caS')$, represented by
  $\Xi\circ\Phi\circ \Xi^{-1}$.
\end{rem}

\begin{defi}
  \label{defi:extendedmcg}
  Let
  $\frH\subset\fS_{\pi_0(\del^{\text{out}}\caS)}\times\fS_{\pi_0(\del^{\text{in}}\caS)}$
  be a subgroup, where $\frS$ denotes the symmetric group on the finite set
  given as index. We define the \emph{extended mapping class group}
  $\Gamma^\frH(\caS)$ as the group of isotopy classes of diffeomorphisms
  $\Phi\colon\caS\to\caS$ that preserve the orientation of $\caS$ and the
  boundary parametrisation, and permute the components of
  $\del^{\text{out}}\caS$ and $\del^{\text{in}}\caS$ according to a pair of
  permutations in $\frH$.
  
  If we take
  $\frH=\fS_{\pi_0(\del^{\text{out}}\caS)}\times\fS_{\pi_0(\del^{\text{in}}\caS)}$,
  we also write $\Gamma(\caS)$ for the extended mapping class group.  If
  $\caS=\Sigma_{g,n}$ we also write $\Gamma_{g,n}^\frH=\Gamma^{\frH}(\caS)$ and
  $\Gamma_{g,(n)}=\Gamma(\caS)$ for the extended mapping class groups.
\end{defi}
Note that we have an extension
$1\to\Gamma(\caS,\del\caS)\to\Gamma^\frH(\caS)\to \frH\to 1$.

\begin{defi}
  \label{def:centraliser}
  If $G$ is a group, we denote by $\on{Conj}(G)$ the set of conjugacy classes of
  $G$. For a group element $\gamma\in G$, we denote by $Z(\gamma,G)\subseteq G$
  the \emph{centraliser} of $\gamma$, i.e.\ the subgroup of all elements
  $\gamma'\in G$ that commute with $\gamma$.
\end{defi}

\begin{nota}
  \label{nota:cG}
  We fix, once and for all, for all conjugacy classes in
  $\on{Conj}(\Gamma_{g,n})$, a representative of the class. We denote by
  $\fg\colon\Gamma_{g,n}\to\Gamma_{g,n}$ the function of sets assigning to each
  element of $\Gamma_{g,n}$ the representative of its class.
\end{nota}

\begin{defi}
 \label{defi:modulispace}
 Let $\caS$ be a surface. We denote by $\frM(\caS)$ the moduli space of Riemann
 structures on $\caS$; two Riemann structures on $\caS$ are considered
 equivalent if there is a diffeomorphism $\Psi\colon\caS\to\caS$ fixing
 $\del\caS$ pointwise and pulling back one Riemann structure to the other. If
 $\caS=\Sigma_{g,n}$ we also write $\frM_{g,n}$ for the moduli space
 $\frM(\Sigma_{g,n})$.\looseness-1
\end{defi}
The hypothesis that every connected component of $\caS$ has non-empty boundary
implies that $\frM(\caS)$ is a classifying space for the group
$\Gamma(\caS,\del\caS)$.

\subsection{Arcs and the Alexander method}

For the rest of this section we fix a connected surface $\caS$ of type
$\Sigma_{g,n}$, with $g\ge0$ and $n\ge1$, and focus on the mapping class group
$\Gamma(\caS,\del\caS)$.  Given a mapping class $\phi\in\Gamma(\caS,\del\caS)$,
we construct a system of simple closed curves on $\caS$ cutting $\caS$ into two
subsurfaces $W$ and $Y$: the subsurface $W\subset\caS$ is the \emph{white}
subsurface and is, up to isotopy, the maximal subsurface of $\caS$ satisfying
the following conditions:
\begin{itemize}
\item all connected components of $W$ touch $\del\caS$;\vspace*{-3px}
\item $\phi$ can be represented by a diffeomorphism of $\caS$ fixing $W$ pointwise.
\end{itemize}

We start by recalling some standard facts about embedded arcs in surfaces. The
material of this subsection is taken, up to minor changes, from \cite{FarbMargalit}.
For the following definition see \cite[§\,1.2.7]{FarbMargalit}.

\begin{defi}
  \label{defn:arc}
  An \emph{arc} in $\caS$ is a smooth embedding
  $\alpha\colon[0;1]\hookrightarrow \caS$ such that
  $\alpha^{-1}(\del \caS)=\{0,1\}$ and $\alpha$ is transverse to $\del \caS$.
  Two arcs are \emph{disjoint} if their images are disjoint (also at the
  endpoints).  An arc is \emph{essential} if it does not cut $\caS$ in two
  parts, one of which is a disc.
  
  Two arcs $\alpha$ and $\alpha'$ are \emph{directly isotopic} if
  $\alpha(0)=\alpha'(0)$, $\alpha(1)=\alpha'(1)$, and there is an isotopy of
  embeddings $[0;1]\to \caS$ that is stationary on $\set{0,1}$ and connects
  $\alpha$ to $\alpha'$.  Two arcs are \emph{inversely isotopic} if the previous
  holds after reparametrising one of the two arcs in the opposite direction.
  Two arcs are \emph{isotopic} if they are directly or inversely isotopic; we
  write $\alpha\sim\alpha'$ if $\alpha$ and $\alpha'$ are isotopic.
  
  Two arcs $\alpha$ and $\beta$ are in \emph{minimal position} if they are
  disjoint at their endpoints, they intersect transversely, and the number of
  intersection points in $\alpha\cup\beta$ is minimal among all choices of
  $\alpha'\sim\alpha$ and $\beta'\sim\beta$ with $\alpha'$ and $\beta'$
  transverse.
\end{defi}

Note that we only consider isotopy classes of arcs relative to their endpoints; two
arcs sharing one endpoint are never considered in minimal position (and, by convention,
cannot be isotoped to be in minimal position). In particular, unless $\caS$ is a disc,
there are more than countably many isotopy classes of essential arcs in $\caS$.

If $\chi(\caS)=2-2g-n\le0$, then according to \cite[§\,1.2.7]{FarbMargalit} the
following statement holds: given a collection of essential arcs
$\alpha_1,\dots,\alpha_k$ in $\caS$ which have all distinct endpoints and are
pairwise non-isotopic, one can replace each $\alpha_i$ with an arc
$\alpha'_i\sim\alpha_i$ so that $\alpha'_1,\dots,\alpha'_k$ are pairwise in
minimal position. In fact it suffices to choose a Riemannian metric of constant
curvature on $\caS$ such that $\del\caS$ is geodesic, and replace each
$\alpha_i$ with its geodesic representative relative to the endpoints: the
hypothesis on $\chi(\caS)$ ensures that we get a non-positively curved metric,
so that geodesic representatives are unique; moreover geodesic representatives
are automatically pairwise in minimal position.

Among all connected surfaces with non-empty boundary, the only one with positive
Euler characteristic is $\Sigma_{0,1}$, i.e.\ the disc: note that the statement
holds vacuously also for the disc, which contains no essential arc.

The following is a special case of the Alexander method
\cite[Prop.\,2.8]{FarbMargalit}.
\begin{prop}
  \label{prop:Alexandermethod}
  Let $\alpha_0,\dots,\alpha_k$ and $\beta$ be a collection of essential arcs in
  $\caS$, and assume the following:
  \begin{itemize}
  \item all arcs are pairwise in minimal position;\vspace*{-5px}
  \item the arcs $\alpha_0,\dots,\alpha_k$ are pairwise disjoint.
  \end{itemize}
  Let $\Phi$ be a diffeomorphism of $(\caS,\del\caS)$, and suppose that $\Phi$
  fixes each of $\alpha_0,\dots,\alpha_k$ and $\beta$ up to isotopy relative to
  the endpoints. Then $\Phi$ can be isotoped to a diffeomorphism $\Phi'$ of
  $\caS$ that fixes $\alpha_0\cup\dotsb\cup\alpha_k\cup\beta$ pointwise.
\end{prop}
In \iref{Proposition}{prop:Alexandermethod}, one can enhance the requirement on
$\Phi'$ to be the following: the map $\Phi'$ fixes pointwise a small
neighbourhood $U\subset \caS$ of the union
\mbox{$\alpha_0\cup\dotsb\cup\alpha_k\cup\beta\cup\del\caS$}. Here and in the
following, a \emph{small neighbourhood} is required to deformation retract onto
$\alpha_0\cup\dotsb\cup\alpha_k\cup\beta\cup\del\caS$ by restriction of an
ambient homotopy, defined on $\caS$ and stationary on
$\alpha_0\cup\dotsb\cup\alpha_k\cup\beta\cup\del\caS$.

\begin{rem}
  The Alexander method, as stated in \cite{FarbMargalit}, only applies under the
  additional hypothesis that the arcs $\alpha_0,\dots,\alpha_k$ and $\beta$ are
  pairwise non-isotopic. We remark, however, that this hypothesis is not
  essential.

  To see this, suppose that $\alpha_0,\dots,\alpha_k$ and $\beta$ is a
  collection of arcs as in \iref{Proposition}{prop:Alexandermethod}: up to
  reordering the arcs $\alpha_i$, we can assume that there is $0\le k'\le k$
  such that the arcs $\alpha_0,\dots,\alpha_{k'}$ and $\beta$ are pairwise
  non-isotopic, and, moreover, each $\alpha_i$ with $i\ge k'+1$ is isotopic to
  some $\alpha_j$ with $j\le k'$.

  We can then apply the Alexander method to the collection
  $\alpha_0,\dots,\alpha_{k'}$ and $\beta$, obtaining a diffeomorphism $\Phi'$
  that fixes a small neighbourhood $U$ of
  \mbox{$\alpha_0\cup\dotsb\cup\alpha_{k'}\cup\beta$} pointwise. We then argue as
  follows: for each index $i\ge k'+1$, there is an index $j\le k'$ such that
  $\alpha_i$ and $\alpha_j$ are isotopic and in minimal position: this implies
  that they cobound (together with two segments in $\partial \caS$) a rectangle
  in $\caS$, and, up to shrinking, we can assume that this rectangle lies
  already inside $U$, i.e.\ we can assume that $\alpha_i$ is fixed pointwise by
  $\Phi'$ as well.
\end{rem} 

\subsection{The cut locus of a mapping class}

In this subsection we fix a class $\phi\in\Gamma(\caS,\del\caS)$, represented by
a diffeo\-morphism $\Phi$, and study the isotopy classes of arcs and curves that
it fixes.

\begin{defi}
  \label{defi:parallelarcs}
  Two arcs $\alpha$ and $\alpha'$ in $\caS$ are \emph{directly parallel} if they
  are disjoint and there is an \emph{embedding}
  $[0;1]\times[0;1]\hookrightarrow\caS$ restricting to $\alpha$ on
  $[0;1]\times\set{0}$ and to $\alpha'$ on $[0;1]\times\set{1}$, and restricting
  to an embedding $\set{0,1}\times[0;1]\hookrightarrow\del\caS$.
  
  Two arcs $\alpha$ and $\alpha'$ are \emph{inversely parallel} if the previous
  holds after reparametrising one of the two arcs in the opposite direction. Two
  arcs $\alpha$ and $\alpha'$ are \emph{parallel} if they are directly or
  inversely parallel.
\end{defi}

\begin{figure}
  \centering
  \begin{tikzpicture}[xscale=1.45,yscale=1.45]
    \draw[looseness=1.5] (2,0) to[out=90,in=90] (4,0);
    \draw (1,0) to[out=90,in=-90] ++(-.6,1.6) to[out=90,in=180] ++(1.1,.8) to[out=0,in=160]
    (3,2) to[out=-20,in=90] (5,0);
    \draw[thick,dblue!20] (1.2,.08) to[out=90,in=-90] (.7,1.5) to[out=90,in=180]
    (1.5,2.1) to[out=0,in=90] (2.3,1.5) to[out=-90,in=90] (1.8,.08);
    \draw[thick,dblue] (1.8,.08) -- (1.8,-.08);
    \draw[thick,dblue] (1.2,.08) -- (1.2,-.08);
    \draw[thick,dblue,looseness=1.5] (1.9,.06) to[out=90,in=90] (4.1,.07);
    \draw[thick,dblue,looseness=1] (1.3,.09) to[out=90,in=-90] (.8,1.5) to[out=90,in=180]
    (1.5,2) to[out=0,in=90] (2.2,1.5) to[out=-90,in=90] (1.7,.09);
    \draw[thick,dblue,looseness=1] (1.4,.09) to[out=90,in=-90] (.95,1.5) to[out=90,in=180]
    (1.5,1.85) to[out=0,in=160] (1.8,1.67);
    \draw[thick,dblue!20] (1.8,1.67) to[out=-20,in=90] (2,1.4) to[out=-90,in=90] (1.65,.09);
    \draw[thick,dblue] (1.65,.09) -- (1.65,-.09);
    \draw[thick,bbblue,looseness=1] (1.5,1.26) to[out=50,in=90] (1.56,1.05) to[out=-90,in=90] (1.55,.1);
    \draw[thick,dblue] (1.55,.1) -- (1.55,-.09);
    \draw[looseness=1] (1.1,1.5) to[out=-90,in=-90] ++(.8,0) to[out=90,in=90] (1.1,1.5);
    \draw[thick,dblue,looseness=.7] (1.5,.1) to[out=90,in=-90] (1.47,1.1) to[out=90,in=230] (1.5,1.27);
    \draw[thin] (1.15,1.38) to[out=130,in=-90] (1.05,1.6);
    \draw[thin] (1.85,1.38) to[out=50 ,in=-90] (1.95,1.6);
    \draw[thick,dblue,looseness=1] (1.1,.06) to[out=90,in=-90] (.58,1.55) to[out=90,in=180]
    (1.5,2.2) to[out=0,in=150] (2.4,1.88) to[out=-30,in=90] (4.9,.07);
    \draw[looseness=.3,thick] (1,0) to[out=-90,in=-90] ++(1,0) to[out=90,in=90] (1,0);
    \draw[looseness=.3,thick] (4,0) to[out=-90,in=-90] ++(1,0) to[out=90,in=90] (4,0);
    \draw[dblue,very thick,fill=dblue] (1.55,2) -- (1.45,1.97) -- (1.45,2.03) -- (1.55,2);
    \draw[dblue!20,very thick,fill=dblue!20] (1.55,2.1) -- (1.45,2.07) -- (1.45,2.13) -- (1.55,2.1);
    \draw[dblue,very thick,fill=dblue] (1.35,1.84) -- (1.26,1.78) -- (1.24,1.84) -- (1.35,1.84);
    \draw[dblue,very thick,fill=dblue] (1.55,2.2) -- (1.45,2.17) -- (1.45,2.23) -- (1.55,2.2);
    \draw[dblue,very thick,fill=dblue] (3.05,1.035) -- (2.95,1.005) -- (2.95,1.065) -- (3.05,1.035);
    \draw[dblue,very thick,fill=dblue] (1.475,.8) -- (1.51,.7) -- (1.45,.7) -- (1.475,.8);
    \node[dblue] at (3.2,1.11) {\tiny $\alpha_0$};
    \node[dblue] at (3.4,1.55) {\tiny $\alpha_1$};
    \node[dblue] at (1.98,1.75) {\tiny $\alpha_2$};
    \node[dblue] at (1.37,1) {\tiny $\alpha_3$};
    \node[dblue] at (1.1,1.2) {\tiny $\alpha_4$};
    \node[dblue!70] at (2.41,1.5) {\tiny $\alpha_5$};
    \node at (1.5,-.3) {\tiny $1$};
    \node at (4.5,-.3) {\tiny $2$};
  \end{tikzpicture}
  \caption{A maximal collection of six pairwise non-parallel arcs
    $\alpha_0,\dotsc,\alpha_5$ on a surface of type $\Sigma_{1,2}$}
  \label{fig:nonparallelarcs}
\end{figure}
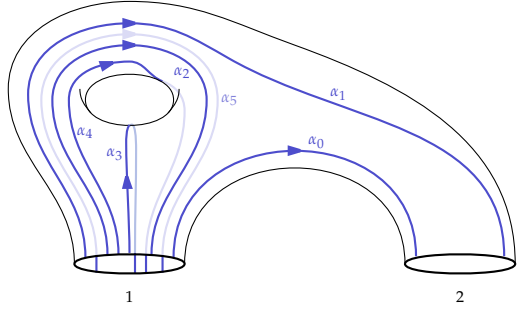
Note that in the previous definition, we do not insist that the embedding
\mbox{$[0;1]\times[0;1]\hookrightarrow\caS$} is orientation-preserving; see
\iref{Figure}{fig:nonparallelarcs}.
\begin{defi}
  \label{defi:fixedarccomplex}
  The \emph{fixed-arc complex of $\phi$} is an abstract simplicial complex whose
  vertices are all isotopy classes of essential arcs $\alpha$ in $\caS$ which
  are fixed by $\phi$.  A collection of isotopy classes of arcs
  $\alpha_0,\dotsc,\alpha_k$ span a $k$-simplex if the arcs
  $\alpha_0,\dotsc,\alpha_k$ can be isotoped to disjoint, pairwise non-parallel
  arcs $\alpha'_0,\dotsc,\alpha'_k$.
  The mapping class $\phi$ is called \emph{$\del$-irreducible} if its fixed-arc
  complex is empty and if $\caS$ is not of type $\Sigma_{0,1}$.
\end{defi}

\begin{expl}
  Every isotopy class of essential arcs in $\caS$ is fixed (up to isotopy) by
  the identity $\one\in\Gamma(\caS,\del\caS)$.  Therefore $\one$ is not
  $\del$-irreducible, provided that $\caS$ admits some essential arc; if $\caS$
  does not admit essential arcs, then $\caS$ is a disc $\Sigma_{0,1}$ and we
  have prescribed, also in this case, that $\one\in\Gamma_{0,1}$ is \emph{not}
  $\del$-irreducible.  The reason to regard $\one\in\Gamma_{0,1}$ as not being
  $\del$-irreducible will become clear later; for the moment we make a simple
  comparison and say that, in a similar way, $1\in\Z$ is not considered a prime
  number.
\end{expl}

\begin{expl}
  The fixed-arc complex of $\phi\in\Gamma_{0,2}\cong\Z$ is empty if
  $\phi\neq\one$ and consists of uncountably many vertices, joined by no higher
  simplex if $\phi=\one$.  Therefore every non-trivial element in $\Gamma_{0,2}$
  is $\del$-irreducible.
\end{expl}

\begin{expl}
  \label{expl:boundarydehn}
  For $g\ge1$ the boundary Dehn twist $T_{\del}\in\Gamma_{g,1}$ is
  $\del$-irreducible, though there are plenty of isotopy classes of \emph{simple
    closed curves} in $\Sigma_{g,1}$ that are fixed by $T_{\del}$: in fact,
  \emph{all} simple closed curves are fixed, up to isotopy, by
  $T_{\del}$. Nevertheless, \emph{no isotopy class} of essential arcs is fixed by
  $\phi$; here it is crucial to consider isotopy classes of arcs relative to the
  endpoints.
\end{expl}

Note that a simplex in the fixed-arc complex of a class
$\phi\in\Gamma(\caS,\del\caS)$ has dimension at most $-3\chi(\caS)-1$ if
$\chi(\caS)<0$, and at most $0$ if $\caS$ is of type $\Sigma_{0,2}$.  In the
second case, just note that each two disjoint essential arcs in $\Sigma_{0,2}$
are parallel. In the first case, let $\alpha_0,\dots,\alpha_k$ be disjoint,
pairwise non-parallel and essential arcs in $\caS$; then cutting $\caS$ along
the arcs $\alpha_i$ yields a surface whose connected components are either
hexagons or \new{connected} surfaces of negative Euler characteristic. Up to
adjoining more arcs (and thus increase $k$), we can assume to have only
hexagons. Each hexagon has 3 sides coming from $\del\caS$ and 3 sides coming
from the cuts. If $\ell\ge1$ denotes the number of hexagons, we have
$3\ell=2(k+1)$, as each arc contributes to 2 hexagons; moreover
$\chi(\caS)=\ell-(k+1)$, implying $3\chi(\caS)=3\ell-3(k+1)=-k-1$.

\begin{constr}
  \label{constr:cutlocus}
  Let $\caS$ be not of type $\Sigma_{0,1}$, let $\phi\in\Gamma(\caS,\del\caS)$,
  and let $\alpha_0,\dots,\alpha_k$ be disjoint, essential, pairwise
  non-parallel arcs in $\caS$, representing a \emph{maximal} simplex in the
  fixed-arc complex of $\phi$. Let $U$ be a closed, small neighbourhood of the
  union $\alpha_0\cup\dotsb\cup\alpha_k\cup\del \caS$. The complement
  $\caS\setminus U$ consists of many regions, some of which may be discs: let
  $W\subset\caS$ denote the union of $U$ and all discs in $\caS\setminus U$.
  Then $W$ is a closed, possibly disconnected subsurface of $\caS$; we denote by
  $Y$ the closure of $\caS\setminus W$.  If $\del W$ denotes the union of all
  boundary components of $W$, and $\del Y$ denotes the union of all boundary
  components of $Y$, then $\del W$ takes the form
  $\del \caS\cup c_1\cup\dotsb\cup c_h$, for some $h\geq0$ and some curves
  $c_1,\dots,c_h\subset\caS$; similarly $\del Y=c_1\cup\dotsb \cup c_h$.  The
  curves $c_1,\dots,c_h$ inherit a canonical boundary orientation from $Y$,
  which is oriented as subsurface of $\caS$.
\end{constr}

\begin{defi}
  \label{defi:cutlocus}
  For $\phi$ and $\alpha_0,\dots,\alpha_k$ as above, we define the \emph{cut
    locus} of $\phi$, re\-la\-tive to the simplex $\alpha_0,\dots,\alpha_k$, as
  the isotopy class of the multicurve $c_1,\dots,c_h$, de\-noted
  $[c_1,\dots,c_h]$. Here and in the following, a \emph{multicurve} is an
  unordered collection of disjoint and \emph{oriented} simple closed curves, and
  two multicurves are considered isotopic if there is an ambient isotopy
  bringing the first to the second.
  
  The two regions $W$ and $Y$ are called the associated \emph{white} and
  \emph{yellow} regions or subsurfaces of $\caS$, and they depend, as subsets of
  $\caS$, on a choice of a multicurve representing the cut locus.  If
  $\phi=\one\in\Gamma_{0,1}$, we declare the cut locus to be empty and $W$ to be
  the entire surface $\Sigma_{0,1}$.
\end{defi}

\begin{figure}
  \centering
  \begin{tikzpicture}[scale=.9]
    \draw[dgreen,thick,looseness=.2] (3.5,3.4) to[out=0,in=0] (3.5,4.3)
    to[out=180,in=180] (3.5,3.4);
    \fill[white,opacity=.8] (3.5,5) rectangle (3.8,3);
    \draw[dgreen,semithick,looseness=.25] (.5,4.24) to[out=0,in=0] (.5,4.69)
    to[out=180,in=180] (.5,4.24);
    \fill[white,opacity=.8] (.5,4.8) rectangle (.7,4.2);
    \draw[dgreen,semithick,looseness=.25] (6.5,4.24) to[out=0,in=0] (6.5,4.69)
    to[out=180,in=180] (6.5,4.24);
    \fill[white,opacity=.8] (6.5,4.8) rectangle (6.7,4.2);
    \draw[dgreen,semithick,looseness=.25] (8.5,4.24) to[out=0,in=0] (8.5,4.69)
    to[out=180,in=180] (8.5,4.24);
    \fill[white,opacity=.8] (8.5,4.8) rectangle (8.7,4.2);
    \node[dgreen] at (.5,4.9) {\tiny $d_1$};
    \node[dgreen] at (3.5,4.5) {\tiny $d_3$};
    \node[dgreen] at (6.5,4.9) {\tiny $d_5$};
    \node[dgreen] at (8.5,4.9) {\tiny $d_7$};
    \node[dgreen] at (-.5,4) {\tiny $d_2$};
    \node[dgreen] at (5.5,4) {\tiny $d_4$};
    \node[dgreen] at (9.5,4) {\tiny $d_6$};
    \node at (.2,3.25) {\tiny $c_1$};
    \node at (2.2,3.25) {\tiny $c_2$};
    \node at (4.2,3.25) {\tiny $c_3$};
    \node at (6.2,3.25) {\tiny $c_4$};
    \node at (8.2,3.25) {\tiny $c_5$};
    \draw[dgreen,semithick,looseness=1.36] (-.1,4) to[out=90,in=90] ++(1.2,0)
    to[out=-90,in=-90] ++(-1.2,0);
    \draw[dgreen,semithick,looseness=1.36] (5.9,4) to[out=90,in=90] ++(1.2,0)
    to[out=-90,in=-90] ++(-1.2,0);
    \draw[dgreen,semithick,looseness=1.36] (7.9,4) to[out=90,in=90] ++(1.2,0)
    to[out=-90,in=-90] ++(-1.2,0);
    \draw[looseness=.35,thick] (0,3) to[out=-90,in=-90] ++(1,0) to[out=90,in=90] (0,3);
    \draw[looseness=.35,thick] (2,3) to[out=-90,in=-90] ++(1,0) to[out=90,in=90] (2,3);
    \draw[looseness=.35,thick] (4,3) to[out=-90,in=-90] ++(1,0) to[out=90,in=90] (4,3);
    \draw[looseness=.35,thick] (6,3) to[out=-90,in=-90] ++(1,0) to[out=90,in=90] (6,3);
    \draw[looseness=.35,thick] (8,3) to[out=-90,in=-90] ++(1,0) to[out=90,in=90] (8,3);
    \draw[dyellow,looseness=1.3] (3,3) to[out=90,in=90] (4,3);
    \draw[dyellow,looseness=1.5] (2,3) to[out=90,in=90] (5,3);
    \draw[dyellow] (0,3) to[out=90,in=-90] (-.3,4) to[out=90,in=180] (.5,4.7)
    to[out=0,in=90] (1.3,4) to[out=-90,in=90] (1,3);
    \draw[dyellow] (6,3) to[out=90,in=-90] (5.7,4) to[out=90,in=180] (6.5,4.7)
    to[out=0,in=90] (7.3,4) to[out=-90,in=90] (7,3);
    \draw[dyellow] (8,3) to[out=90,in=-90] (7.7,4) to[out=90,in=180] (8.5,4.7)
    to[out=0,in=90] (9.3,4) to[out=-90,in=90] (9,3);
    \fill[white, opacity=.8] (0,2.5) rectangle (11,3);   
    \draw[looseness=.3,thick] (1,0) to[out=-90,in=-90] ++(1,0) to[out=90,in=90] (1,0);
    \draw[looseness=.3,thick] (5,0) to[out=-90,in=-90] ++(1,0) to[out=90,in=90] (5,0);
    \draw[looseness=.3,thick] (7,0) to[out=-90,in=-90] ++(1,0) to[out=90,in=90] (7,0);
    \draw (1,0) to[out=85,in=-90]  (0,3);
    \draw (2,0) to[out=95,in=-90]  (3,3);
    \draw[looseness=3.5] (1,3) to[out=-85,in=-95] (2,3);
    \draw[looseness=2.5] (5,3) to[out=-85,in=-95] (6,3);
    \draw[looseness=3.5] (7,3) to[out=-85,in=-95] (8,3);
    \draw[looseness=3.5] (6,0) to[out=85,in=95]   (7,0);
    \draw (4,3) to [out=-90,in=90] ++(-.5,-1.5) to [out=-90,in=95] (5,0);
    \draw (8,0) to[out=95,in=-90]  (9,3);
    \draw[looseness=1] (4.15,1.6) to[out=-90,in=-90] ++(.8,0) to[out=90,in=90] (4.15,1.6);
    \draw[thin] (4.2,1.48) to[out=130,in=-90] (4.1,1.7);
    \draw[thin] (4.9,1.48) to[out=50 ,in=-90] (5,1.7);
    \node[anchor=base] at (1.5,-.4) {\tiny $3$};
    \node[anchor=base] at (5.5,-.4) {\tiny $1$};
    \node[anchor=base] at (7.5,-.4) {\tiny $2$};
    \draw[dyellow,looseness=1] (.1,4) to[out=-90,in=-90] ++(.8,0) to[out=90,in=90] (.1,4);
    \draw[dyellow,thin] (.15,3.88) to[out=130,in=-90] (.05,4.1);
    \draw[dyellow,thin] (.85,3.88) to[out=50 ,in=-90] (.95,4.1);
    \draw[dyellow,looseness=1] (6.1,4) to[out=-90,in=-90] ++(.8,0) to[out=90,in=90] (6.1,4);
    \draw[dyellow,thin] (6.15,3.88) to[out=130,in=-90] (6.05,4.1);
    \draw[dyellow,thin] (6.85,3.88) to[out=50 ,in=-90] (6.95,4.1);
    \draw[dyellow,looseness=1] (8.1,4) to[out=-90,in=-90] ++(.8,0) to[out=90,in=90] (8.1,4);
    \draw[dyellow,thin] (8.15,3.88) to[out=130,in=-90] (8.05,4.1);
    \draw[dyellow,thin] (8.85,3.88) to[out=50 ,in=-90] (8.95,4.1);
    \draw[bgrey,thick] (10,-.5) -- (10,5);
    \draw[dgreen,thick,looseness=.35] (11,.9) to[out=-90,in=-90] ++(1,0)
    to[out=90,in=90] (11,.9);
    \fill[white,opacity=.8] (10.9,.9) rectangle (12.1,.7);
    \node[dgreen] at (12.2,.9) {\tiny $d$};
    \draw[dyellow] (11,.5) to[out=90,in=-90] (11,1) to[out=90,in=-90] (10.7,2)
    to[out=90,in=-90] (11,2.75)
    to[out=90,in=-90] (10.7,3.5) to[out=90,in=180] (11.5,4.2) to[out=0,in=90] (12.3,3.5)
    to[out=-90,in=90]
    (12,2.75) to[out=-90,in=90] (12.3,2) to[out=-90,in=90] (12,1) -- (12,.5);
    \draw[dyellow,looseness=1] (11.1,2) to[out=-90,in=-90] ++(.8,0) to[out=90,in=90] (11.1,2);
    \draw[dyellow,thin] (11.15,1.88) to[out=130,in=-90] (11.05,2.1);
    \draw[dyellow,thin] (11.85,1.88) to[out=50 ,in=-90] (11.95,2.1);
    \draw[dyellow,looseness=1] (11.1,3.5) to[out=-90,in=-90] ++(.8,0)
    to[out=90,in=90] (11.1,3.5);
    \draw[dyellow,thin] (11.15,3.38) to[out=130,in=-90] (11.05,3.6);
    \draw[dyellow,thin] (11.85,3.38) to[out=50 ,in=-90] (11.95,3.6);
    \draw[looseness=.35,thick] (11,0) to[out=-90,in=-90] ++(1,0) to[out=90,in=90] (11,0);
    \draw[looseness=.35,thick] (11,.5) to[out=-90,in=-90] ++(1,0) to[out=90,in=90] (11,.5);
    \fill[white,opacity=.8] (10.9,.5) rectangle (12.1,.3);
    \draw (11,0) -- (11,.5);
    \draw (12,0) -- (12,.5);
  \end{tikzpicture}
  \caption{Two examples of a decomposition into the ‘yellow’ and the ‘white’
    region, according to a fixed mapping class $\varphi$. In the first case,
    $\varphi$ is given by the product of the Dehn twists along the curves
    $d_1,\dotsc,d_7$, and in the second case, it is just the Dehn twist along
    the single green curve $d$. In the second case, the mapping class $\varphi$
    is $\partial$-irreducible, the cut locus consists of the only isotopy class
    of $d$, oriented as boundary of the yellow region, and the white region is
    just a collar neighbourhood of $\partial \caS$.}
  \label{fig:cutlocus}
\end{figure}
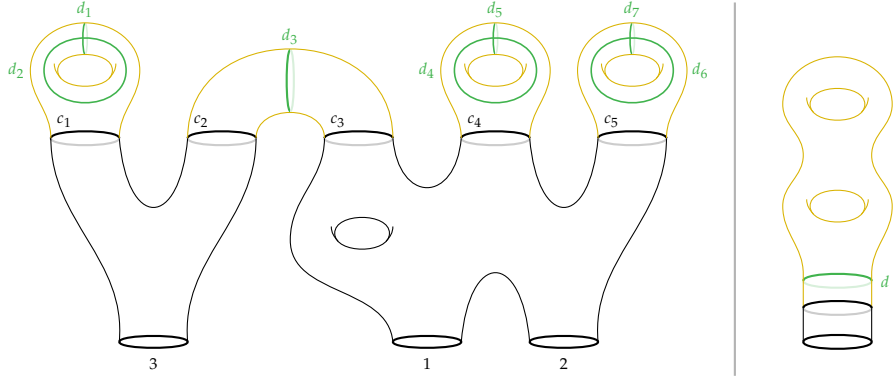
See \iref{Figure}{fig:cutlocus} for an example of the cut locus of a mapping
class obtained as a simple product of Dehn twists.  Note that it is possible
that two curves $c_i$ and $c_j$ cobound a cylinder in $Y$: in this case the two
curves are isotopic as non-oriented simple closed curves, but the isotopy
bringing $c_i$ to $c_j$, spanned by the cylinder of $Y$, ends with an
orientation-reversing diffeomorphism $c_i\cong c_j$, as the two curves inherit
their orientation from $Y$ while being on \new{\emph{opposite}} sides of the
cylinder contained in $Y$. Hence the isotopy classes of the curves
$c_1,\dots,c_h$, considered as oriented curves, are all distinct.

Note that the cut locus of the identity $\one\in\Gamma(\caS,\del\caS)$ is
\emph{empty}, and the white and yellow decomposition consists of a white region
$W=\caS$ and an empty yellow region. Vice versa, the cut locus of a non-trivial
mapping class $\phi\in\Gamma(\caS,\del\caS)$ is always non-empty.

\iref{Definition}{defi:cutlocus} depends a priori on a choice of a maximal simplex in
the fixed-arc complex of $\phi$; there is, moreover, a subtle detail that we
should check to guarantee that \iref{Definition}{defi:cutlocus} is well-posed:
suppose that the disjoint arcs $\alpha'_0,\dotsc,\alpha'_k$ are isotopic to the
disjoint arcs $\alpha_0,\dotsc,\alpha_k$ (that is, $\alpha'_i\sim\alpha_i$ for all
$0\leq i\leq k$), so that the two collections of arcs represent the same maximal
simplex in the fixed-arc complex of $\phi$; then we need to check that the two
collections of arcs give rise to the same collection of isotopy classes of
oriented, disjoint simple closed curves $c_1,\dots, c_h$.

We will prove directly that the cut locus only depends on $\phi$, and not on
the chosen \emph{maximal} simplex (and its representative) in the fixed-arc
complex of $\phi$.

\begin{lem}
  \label{lem:cutlocus}
  Let $\phi$ and $\alpha_0,\dots,\alpha_k$ be as in
  \iref{Definition}{defi:cutlocus}, and let $\beta$ be an arc whose endpoints
  are disjoint from the endpoints of $\alpha_0,\dots,\alpha_k$. Suppose that the
  isotopy class of $\beta$ is fixed by $\phi$; then $\beta$ can be isotoped,
  relative to its endpoints, to an arc $\hat\beta$ lying in a small
  neighbourhood $U$ of $\alpha_0\cup\dotsb\cup\alpha_k\cup\del \caS$.
\end{lem}
\begin{proof}
  Up to isotoping $\beta$ to another arc, we can assume that $\beta$ is in minimal
  position with respect to the arcs $\alpha_0,\dots,\alpha_k$.
  By \iref{Proposition}{prop:Alexandermethod} we can represent $\phi$ by a
  diffeomorphism $\Phi$ that fixes a closed neighbourhood $U'$ of the union
  $\alpha_0\cup\dotsb\cup\alpha_k\cup\beta\cup\del \caS$. Let $U\subset U'$ be a small
  neighbourhood of $\alpha_0\cup\dotsb\cup\alpha_k\cup\del \caS$.
  
  If $\beta$ is disjoint from the arcs $\alpha_i$, by maximality of the simplex
  $\alpha_0,\dots,\alpha_k$ in the fixed-arc complex of $\phi$, we obtain that
  either $\beta$ is not essential (and can then be isotoped inside a small
  neighbourhood of $\del \caS$, hence inside $U$), or $\beta$ is parallel to one
  of the arcs $\alpha_i$ (and can then be isotoped to a small neighbourhood of
  $\del\caS\cup \alpha_i$, hence inside $U$).

  If $\beta$ is not disjoint from the arcs $\alpha_i$, let $\ell\geq 1$ be the
  number of transverse intersections of $\beta$ with
  $\alpha_0\cup\dotsb\cup\alpha_k$: by induction, let us suppose that the
  statement of the lemma holds whenever $\beta$ is replaced by an arc that can
  be isotoped so as to have at most $\ell-1$ intersection points with the arcs
  $\alpha_i$. Suppose that $p$ is one intersection point of $\beta$ with one arc
  $\alpha_i$: suppose further that the segment $[\alpha_i(0);p]\subset\alpha_i$
  contains no other point of $\alpha_i\cap \beta$ in its interior (this means
  that $p$ is an \emph{outermost} point of $\alpha_i\cap\beta$ along
  $\alpha_i$).

  We can operate a surgery on $\beta$ and produce two arcs $\beta'$ and
  $\beta''$ also contained in $U'$ and transverse to
  $\alpha_0\cup\dotsb\cup\alpha_k$; see \iref{Figure}{fig:surgeryarc}: the arc
  $\beta'$ is obtained by smoothing the concatenation of the segments
  $[\beta(0);p]\subset\beta$ and $[p;\alpha_i(0)]\subset\alpha_i$, whereas the
  arc $\beta''$ is obtained by smoothing the concatenation of the segments
  $[\alpha_i(0);p]\subset\alpha_i$ and $[p;\beta(1)]\subset\beta$. We assume
  that $\beta'$ and $\beta''$ are disjoint, and have an endpoint on a small
  interval of $\del\caS$ centred at $\alpha_i(0)$, on opposite sides with
  respect to $\alpha_i(0)$.
  
  Both $\beta'$ and $\beta''$ are contained in $U'$, hence they are fixed
  pointwise by $\Phi$: in particular the isotopy classes of $\beta'$ and
  $\beta''$ are fixed by $\Phi$. Moreover each of $\beta'$ and $\beta''$ has
  strictly less than $\ell$ intersections with $\alpha_0\cup\dotsb\cup\alpha_k$,
  and hence, by inductive hypothesis, each of $\beta'$ and $\beta''$ can be
  isotoped to lie in $U$, relative to its endpoints.

  Let $\hat\beta'$ and $\hat\beta''$ be the two arcs obtained in this way, and
  assume that $\hat\beta'$ and $\hat\beta''$ are transverse. If $\hat\beta'$ and
  $\hat\beta''$ are not in minimal position, they must form some bigon in
  $\caS$; the possibility of half-bigons in the sense of
  \cite[§\,1.2.7]{FarbMargalit} is irrelevant, since we consider arcs up to
  isotopy relative to the endpoints.  Clearly we can simplify all bigons formed
  by $\hat\beta'$ and $\hat\beta''$ without losing that these two arcs are
  contained in $U$.
  
  Suppose therefore that $\hat\beta'$ and $\hat\beta''$ are in minimal position:
  then they are disjoint because $\beta'$ and $\beta''$ were disjoint (in
  particular, it is automatic that $\hat\beta'$ and $\hat\beta''$ do not form
  half-bigons). The arc $\beta$ is homotopic, relative to its endpoints, to the
  concatenation of the arcs $\hat\beta'$ and $\hat\beta''$, which can be
  connected using a small segment near $\alpha_i(0)$. Note that this
  concatenation gives an embedded arc $\hat\beta$ with the same endpoints as
  $\beta$; since homotopic arcs relative to their endpoints are also isotopic
  relative to their endpoints, we have that $\beta$ is isotopic to $\hat\beta$;
  moreover, $\hat\beta$ lies in $U$.
\end{proof}

\begin{tfigure}
  \centering
  \begin{tikzpicture}[scale=1.2]
    \draw[thick,dblue!20,looseness=.8] (0,2.48) to[out=-10,in=130] (.6,2)
    to[out=-50,in=90] (1.4,.6) -- (1.4,.16);
    \draw[thick,dblue!20,looseness=.4] (-1.2,3.81) to[out=60,in=180] (0,3.525);
    \draw[thick,dblue] (1.4,.16) -- (1.4,-.16);
    \draw[dgreen,thick,looseness=.3] (3,2.5) to[out=90,in=90] (4,2.5)
    to[out=-90,in=-90] (3,2.5);
    \fill[white,opacity=.8] (3,2.5) rectangle (4,2.3);
    \draw[thick,dred!20] (0.6,3.41) to[out=-20,in=110] (1.1,2.6) to[out=-70,in=90] (1.6,.16);
    \draw[thick,dyellow!20] (0.6,3.41) to[out=-20,in=110] (1.05,2.6)
    to[out=-70,in=90] (1.55,.16);
    \draw[thick,red] (1.6,.16) -- (1.6,-.16);
    \draw[thick,dyellow] (1.55,.16) -- (1.55,-.16);
    \node[dgreen] at (4.15,2.5) {\tiny $d$};
    \draw[looseness=1] (3.1,3.5) to[out=-90,in=-90] ++(.8,0) to[out=90,in=90] ++(-.8,0);
    \draw[thin] (3.15,3.38) to[out=130,in=-90] (3.05,3.6);
    \draw[thin] (3.85,3.38) to[out=50 ,in=-90] (3.95,3.6);
    \draw[looseness=1] (-.9,3) to[out=-90,in=-90] ++(1.8,0) to[out=90,in=90] ++(-1.8,0);
    \draw[thin] (-.857,2.82) to[out=125,in=-90] (-.95,3.13);
    \draw[thin] (.857,2.82) to[out=65 ,in=-90] (.95,3.13);
    \draw[thick,dblue,looseness=.3] (2.5,2) to[out=60,in=90] (2.7,-.17);
    \draw[very thick,white,opacity=.8,looseness=.3] (2.515,1.97) to[out=60,in=90] (2.69,.17);
    \draw[thick,dblue,looseness=.3] (2.5,.17) to[out=90,in=-120] (2.5,2);
    \draw[thick,dblue] (.6,.12) -- (.6,.6) to[out=90,in=-90] (-1.23,3)
    to[out=90,in=180] (0,4) to[out=0,in=110]
    (1.6,2.9) to[out=-70,in=90] (2.2,.18);
    \draw[thick,dblue] (.3,.11) -- (.3,.5) to[out=90,in=-90] (-1.4,3)
    to[out=90,in=-120] (-1.2,3.81);
    \draw[thick,dblue] (0,3.525) to[out=0,in=160] (0.6,3.65)
    to[out=-20,in=110] (1.3,2.9) to[out=-70,in=90] (1.9,.18);
    \draw[thick,dblue] (1.2,.15) -- (1.2,.5) to[out=90,in=-50] (.3,2);
    \draw[thick,dblue,looseness=.4] (.3,2) to [out=130,in=215] (0,2.48);
    \draw[thick,dred] (0.9,.15) -- (.9,.65) to[out=90,in=-90] (-1.05,3)
    to[out=90,in=180] (0,3.7)
    to[out=0,in=160] (0.6,3.41);
    \draw[thick,lila] (0.85,.15) -- (.85,.61) to[out=90,in=-90] (-1.1,3)
    to[out=90,in=180] (0,3.75)
    to[out=0,in=160] (.3,3.67) to[out=20,in=160] (0.6,3.7) to[out=-20,in=110] (1.35,2.9)
    to[out=-70,in=90] (1.95,.18);
    \draw[very thick,dyellow] (1.835,.67) -- (1.81,.57) -- (1.87,.57) -- (1.835,.67);
    \draw[very thick,lila] (1.935,.67) -- (1.91,.57) -- (1.97,.57) -- (1.935,.67);
    \draw[very thick,dblue] (.6,.63) -- (.57,.53) -- (.63,.53) -- (.6,.62);
    \draw[very thick,lila] (.85,.55) -- (.82,.65) -- (.88,.65) -- (.85,.55);
    \draw[very thick,dred] (.9,.55) -- (.87,.65) -- (.93,.65) -- (.9,.55);
    \draw[very thick,dblue] (2.478,.66) -- (2.46,.56) -- (2.52,.56) -- (2.478,.66);
    \node[dblue] at (-1.35,3.88) {\tiny $\alpha_2$};
    \node[dblue] at (2.5,2.13) {\tiny $\alpha_0$};
    \node[dred] at (1.6,-.3) {\tiny $\beta$};
    \node[dblue] at (1.4,-.3) {\tiny $\alpha_1$};
    \node[dblue] at (1.55,3.83) {\tiny $\alpha_3$};
    \node[dyellow] at (.5,3.26) {\tiny $\beta'$};
    \draw[very thick,white] (1.41,3.81) to[out=200,in=70] (1.2,3.6);
    \draw[thin,dblue!50] (1.41,3.81) to[out=200,in=70] (1.2,3.6);
    \draw[thick,dyellow] (.7,3.35) to[out=135,in=-25] (.58,3.47)
    to[out=155,in=-40] (.368,3.596) to[out=35,in=160] (0.6,3.6)
    to[out=-20,in=110] (1.25,2.9) to[out=-70,in=90] (1.85,.18);
    \draw[very thick,dblue] (1.885,.67) -- (1.86,.57) -- (1.92,.57) -- (1.885,.67);
    \node[lila] at (1.95,0) {\tiny $\beta''$};
    \draw (0,0) -- (0,.5) to[out=90,in=-90] (-1.6,3) to[out=90,in=180] (0,4.2)
    to[out=0,in=130] (1.5,3.5)
    to[out=-50,in=180] (2.5,2) to[out=0,in=-90] (3,2.5) to [out=90,in=-90] (2.7,3.5)
    to[out=90,in=180]
    (3.5,4.2) to[out=0,in=90] (4.3,3.5) to[out=-90,in=90] (4,2.5) -- (4,0);
    \draw[looseness=.15,thick] (0,0) to[out=-90,in=-90] ++(4,0) to[out=90,in=90] (0,0);
  \end{tikzpicture}
  \caption{If $\varphi$ is the Dehn twist along the curve $d$, then the blue
    arcs $\alpha_0,\dotsc,\alpha_3$ constitute a maximal simplex in the
    fixed-arc complex of $\varphi$; the subset $U$ is a small neighbourhood of
    the union of the blue arcs and the black boundary curve.  The red arc
    $\beta$ intersects $\alpha_2$ transversally and the surgery produces the
    yellow and violet arcs $\beta'$ and $\beta''$.}
    \label{fig:surgeryarc}
\end{tfigure}

\begin{prop}
  \label{prop:cutlocus}
  Let $\phi\in\Gamma(\caS,\del\caS)$ and let $\alpha_0,\dots,\alpha_k$ and
  $\beta_0,\dots,\beta_{k'}$ be two sequences of arcs representing two different
  maximal simplices in the fixed-arc complex of $\phi$; we assume that all arcs
  $\alpha_0,\dots,\alpha_k,\beta_0,\dots,\beta_{k'}$ are in minimal
  position. Then the associated cut loci constitute the same isotopy class of an
  oriented multicurve in $\caS$.
\end{prop}
\begin{proof}
  Let $U_{\alpha}$ be a closed small neighbourhood of
  $\alpha_0\cup\dotsb\cup\alpha_k\cup\del \caS$ and $U_{\beta}$ be a closed
  small neighbourhood of $\beta_0\cup\dotsb\cup\beta_{k'}\cup\del \caS$; let
  $W_{\alpha}$ and $W_{\beta}$ be obtained from $U_{\alpha}$ and $U_{\beta}$ by
  adjoining the disc components of $\caS\setminus U_{\alpha}$ and
  $\caS\setminus U_{\beta}$, respectively; see \iref{Definition}{defi:cutlocus}.
  
  By \iref{Lemma}{lem:cutlocus} we can find an isotopy of the identity of
  $(\caS,\del\caS)$ bringing $U_{\alpha}$ in the interior of $U_{\beta}$, and
  hence in the interior of $W_{\beta}$: without loss of generality, in the
  following assume $U_{\alpha}\subseteq U_{\beta}\subseteq W_{\beta}$.
  If $D$ is a disc component of $\caS\setminus U_{\alpha}$, then $D\setminus
  U_{\beta}$ is a union of discs contained in $\caS\setminus U_{\beta}$, and
  therefore $D\subset W_{\beta}$. It follows that $W_{\alpha}\subseteq W_{\beta}$.
  Since every component of $W_{\beta}$ touches $\del \caS$, the map
  $\pi_0(\del\caS)\to\pi_0(W_\beta)$ is surjective. This map factors through the
  canonical map $\pi_0(W_{\alpha})\to\pi_0(W_{\beta})$, as
  $\del\caS\subset W_\alpha$. We conclude that
  $\pi_0(W_{\alpha})\to\pi_0(W_{\beta})$ is surjective.
  
  By the same argument we can find an isotopy of the identity of $\caS$ bringing
  $U_{\beta}$ in the interior of $U_{\alpha}$, and hence $W_{\beta}$ in the
  interior of $W_{\alpha}$: as a consequence we can obtain a surjection
  $\pi_0(W_{\beta})\to\pi_0(W_{\alpha})$, showing that
  $\#\pi_0(W_{\beta})=\#\pi_0(W_{\alpha})$ and that the map
  $\pi_0(W_{\alpha})\to\pi_0(W_{\beta})$ induced by the inclusion is in fact a
  bijection.\enlargethispage{\baselineskip}
  
  We next prove that each component $V$ of $W_{\beta}\setminus W_{\alpha}$ is a
  cylinder with one boundary curve equal to some $c_i$ and one boundary curve
  equal to some $c'_j$. Fix a component $\bar W_{\beta}\subset W_{\beta}$ and
  let $\bar W_{\alpha}\subset W_{\alpha}$ be the unique component of
  $W_{\alpha}$ contained in $\bar W_{\beta}$. We know that, conversely,
  $\bar W_{\beta}$ can be embedded in $\bar W_{\alpha}$: since the genus is
  weakly increasing along embeddings of orientable surfaces with boundary, the
  surfaces $\bar W_{\alpha}$ and $\bar W_{\beta}$ must have the same genus; this
  implies in particular every component $V$ of
  $\bar W_{\beta}\setminus\bar W_{\alpha}$ has genus $0$ and touches at most one
  curve $c_i\in\del\bar W_{\alpha}$.
  
  Since $\bar W_{\alpha}$ and $\bar W_{\beta}$ are connected, we have that a
  component $V$ of $\bar W_{\beta}\setminus\bar W_{\alpha}$ touches exactly one
  curve $c_i\subset\del\bar W_{\alpha}$, and since $V$ cannot be a disc, we
  obtain that $V$ is a surface of genus $0$ with at least two boundary
  components; more precisely, one boundary component of $V$ is a curve
  $c_i\subset\del \bar W_{\alpha}$, and all other boundary components are curves
  $c'_j\subset\del\bar W_{\beta}$.
  
  This proves in particular that the number of boundary components of
  $\bar W_{\beta}$ is greater or equal to the number of boundary components of
  $\bar W_{\alpha}$; we can again reverse the r\^oles of $\alpha$ and $\beta$
  and conclude that $\bar W_{\alpha}$ and $\bar W_{\beta}$ have the same number
  of boundary component; this in turn implies that every connected component $V$
  of $\bar W_{\beta}\setminus\bar W_{\alpha}$ is a cylinder with one boundary
  curve equal to some $c_i$ and one boundary curve equal to some $c'_j$.
  
  \new{Thus} we obtain a bijection between the curves $c_1,\dots,c_h$ and the curves
  $c'_1,\dots,c'_{h'}$, showing in particular that $h=h'$; note also that the
  bijection associates with each curve $c_i$ a curve $c'_j$ in the same isotopy
  class of oriented curves.
\end{proof}

When referring to the \emph{cut locus} of a mapping class
$\phi\in\Gamma(\caS,\del\caS)$ we will henceforth mean the cut locus of $\phi$
with respect to any maximal simplex in the fixed-arc complex of $\phi$.
The cut locus of a mapping class behaves well under conjugation of the mapping
class, as explained in the following Lemma.

\begin{lem}
  \label{lem:cutlocusequiv}
  Let $\phi,\psi\in\Gamma(\caS,\del\caS)$ be mapping classes, let $\Psi$ be a
  diffeomorphism representing $\psi$, and let $[c_1,\dots,c_h]$ be the cut locus
  of $\phi$; then $[\Psi(c_1),\dots,\Psi(c_h)]$ is the cut locus of
  $\psi\phi\psi^{-1}$. In particular, if $\phi$ and $\psi$ commute, then $\psi$
  preserves the cut locus of $\phi$ as an unordered collection of isotopy
  classes of oriented simple closed curves.
\end{lem}
\begin{proof}
  Let $\alpha_0,\dots,\alpha_k$ be disjoint arcs representing a maximal simplex
  in the fixed-arc complex of $\phi$, and represent $\phi$ by a diffeomorphism
  $\Phi$ fixing pointwise a small neighbourhood $U$ of
  $\alpha_0\cup\dots\cup\alpha_k\cup\del\caS$; finally, represent the cut locus
  of $\phi$ by curves $c_1,\dots,c_h$ contained in $U$. Then $\Psi$ induces an
  isomorphism from the fixed-arc complex of $\phi$ to the fixed arc complex of
  $\psi\phi\psi^{-1}$; in particular $\Psi(\alpha_0),\dots,\Psi(\alpha_k)$ are
  disjoint arcs representing a maximal simplex in the fixed-arc complex of
  $\psi\phi\psi^{-1}$. Moreover $\Psi(U)$ is a small neighbourhood of
  $\Psi(\alpha_0)\cup\dots\cup\Psi(\alpha_k)\cup\del\caS$ which is fixed
  pointwise by the representative $\Psi\Phi\Psi^{-1}$ of $\psi\phi\psi^{-1}$.
  We can represent the cut locus of $\phi$ by curves
  $c_1,\dots,c_h\subset\del U$; then
  \[\Psi(c_1),\dots,\Psi(c_h)\subset\del\Psi(U)\]
  are automatically curves representing the cut locus of $\psi\phi\psi^{-1}$,
  according to \iref{Construction}{constr:cutlocus}.
\end{proof}

\section{Centralisers of mapping classes}
\label{sec:centralisers}
We fix a surface $\caS\cong\Sigma_{g,n}$ and a mapping class
$\phi\in\Gamma(\caS,\del\caS)\cong\Gamma_{g,n}$ as in the previous section. In
this section, we prove a structural result for the centraliser
$Z(\phi,\Gamma(\caS,\del\caS))$ of $\phi$ in $\Gamma(\caS,\del\caS)$; see
\iref{Proposition}{prop:Zdelphistructure}.

\subsection{Yellow components, similarity and irreducibility}
\label{subsec:similar}

We fix oriented simple closed curves $c_1,\dots,c_h\subset \caS$ representing
the cut locus of $\phi$, and we let $W\cup Y$ be the associated decomposition of
$\caS$ into its white and yellow regions.
Recall from \iref{Construction}{constr:cutlocus} that the curves $c_1,\dots,c_h$
inherit an orientation from $Y$: we fix an orientation-compatible
parametrisation of each curve $c_i$, i.e.\ an identification with $S^1$. Note
that in this way $c_1,\dots,c_h$ are \emph{incoming} curves for $W$ and
\emph{outgoing} for $Y$. In fact we have
$\del Y=\del^{\text{out}}Y=c_1\cup\dotsb\cup c_h =\del^{\text{in}}W$, whereas
$\del^{\text{out}}W=\del^{\text{out}}\caS=\del\caS$.

We fix a representative $\Phi\colon\caS\to\caS$ of $\phi$ which fixes the white
region $W$ pointwise: to see that this is possible, choose arcs
$\alpha_0,\dots,\alpha_k$ in a maximal simplex of the fixed-arc complex of
$\phi$, choose a first representative $\Phi'$ of $\phi$ fixing pointwise a small
neighbourhood $U$ of $\alpha_0\cup\dots\cup\alpha_k\cup\del\caS$, construct $W$
and $Y$ starting from $U$, and use that the group $\Diff(D^2,\del D^2)$ of
diffeomorphisms of a disc relative to its boundary is contractible, and in
particular connected, in order to isotope $\Phi'$ relative to $U$ to a
diffeomorphism $\Phi$ fixing $W$ pointwise.

We note that this representative $\Phi$ is unique up to an isotopy that is
stationary on $W$: to see this, first note that there is a fibration
\[\Diff(Y,\del Y)\hookrightarrow \Diff(\caS,\del\caS)\overset{p}{\to}
  \mathrm{Emb}_{\del^{\text{out}}}(W,\caS),\]
where $\mathrm{Emb}_{\del^\text{out}}(W,\caS)$ denotes the space of embeddings
of $W$ into $\caS$ restricting to the identity on the boundary
$\del^{\text{out}}W=\del\caS$.  A result of Earle–Schatz \cite{EarleSchatz}
ensures that $\Diff(\caS,\del \caS)$ has contractible components for every
compact orientable surface $\caS$ such that every connected component of $\caS$
is connected; in particular, $\Diff(Y,\del Y)$ also has contractible
components. A result of Gramain \cite[Thm.\,5]{Gramain} ensures that for
disjoint, properly embedded arcs $\alpha_0,\dots,\alpha_k\subset\caS$, the space
$\mathrm{Emb}_{\coprod\del\alpha_i}(\coprod\alpha_i,\caS)$ has also contractible
components; this, together with contractibility of $\Diff(D^2,\del D^2)$,
implies that also $\mathrm{Emb}_{\del^{\text{out}}}(W,\caS)$ has contractible
connected components. Thus, all spaces involved in the above fibration have
contractible connected components; in particular the component
$\Diff(\caS,\del\caS)_{\phi}$ intersects the fibre
$p^{-1}(W\subset\caS)\cong \Diff(Y,\del Y)$ in a connected or empty
subspace, and the representative $\Phi$ of $\phi$ witnesses that this
intersection is non-empty, hence contractible, in particular connected.

For each path component $P\subseteq Y$, the diffeomorphism $\Phi$ restricts to
$\Phi|_P\colon P\to P$, giving an element $\phi_P\in\Gamma(P,\del P)$.

\begin{defi}
  \label{defi:similar}
  Two path components $P$ and $P'$ of $Y$ are \emph{similar} if there is a
  diffeomorphism $\Xi\colon P\to P'$ preserving the boundary parametrisation and
  such that $\smash{\phi_P=(\phi_{P'})^\Xi}$. Note that the path components of
  $\del P$ are not equipped with a preferred order, as well as the path
  components of $\del P'$; yet \iref{Definition}{defi:presbparam} is meaningful
  here; see also \iref{Remark}{rem:Xiconj}.
\end{defi}

\begin{nota}
  \label{nota:Yi}
  We write $Y=\coprod_{i=1}^r\coprod_{j=1}^{s_i}Y_{i,j}$ where
  $Y_{1,1},\dotsc,Y_{r,s_r}\subseteq Y$ are the connected components of $Y$ and
  $Y_{i,j}$ is similar to $Y_{i',j'}$ if and only if $i=i'$. We also let
  $Y_i\coloneqq \coprod_{j=1}^{s_i}Y_{i,j}$.  Here $r\ge0$ is the number of
  similarity classes of components of $Y$ (it can be $0$ if $Y$ is empty, i.e.\
  if $\phi$ is the identity mapping class), whereas $s_i\ge1$ is the number of
  components of $Y$ belonging to the $i$\textsuperscript{th} similarity class.

  For each $1\le i\le r$, there are unique $g_i\ge0$ and $n_i\ge1$ such that
  $Y_{i,j}$ is of type $\Sigma_{g_i,n_i}$.  We denote by
  $\phi_{i,j}\in\Gamma(Y_{i,j},\del Y_{i,j})$ the class represented by the
  restriction $\Phi|_{Y_{i,j}}$, Moreover, we fix diffeomorphisms
  $\Xi_{i,j}\colon Y_{i,j}\to \Sigma_{g_i,n_i}$ and assume that $\Xi_{i,j}$
  preserves the boundary parametrisation.

  The conjugation by $\Xi_{i,j}$ induces an identification
  $\Gamma(Y_{i,j},\del Y_{i,j})\to\Gamma_{g_i,n_i}$, under which $\phi_{i,j}$
  corresponds to some element
  $\smash{\bar\phi_{i,j}\coloneqq\smash{(\phi_{i,j})}^{\smash{\Xi_{i,j}}}}\in\Gamma_{g_i,n_i}$.
  Up to replacing $\Xi_{i,j}$ by another diffeomorphism
  $Y_{i,j}\to \Sigma_{g_i,n_i}$, we can assume that
  $\bar\phi_{i,j}\in\Gamma_{g,n}$ \emph{coincides} with $\fg(\bar\phi_{i,j})$,
  i.e.\ it is the representative of its own conjugacy class (see
  \iref{Notation}{nota:cG}). Note that the diffeomorphism replacing $\Xi_{i,j}$
  is not required to induce the same bijection of sets of boundary components as
  $\Xi_{i,j}$. It can be helpful to remark that $\pi_0(\del Y_{i,j})$ is not
  equipped a priori with a preferred order, and only after choosing $\Xi_{i,j}$
  we obtain an order on $\pi_0(\del Y_{i,j})$ by pulling back the canonical
  order on $\pi_0(\del\Sigma_{g_i,n_i})$.  Under the assumption that the
  diffeomorphisms $\Xi_{i,j}$ are well-chosen we also have
  $\bar\phi_{i,j}=\bar\phi_{i,j'}$ for all $1\le i\le r$ and $1\le j,j'\le s_i$.
  We therefore write $\bar\phi_i\coloneqq\bar\phi_{i,j}$.
\end{nota}
Note that, counting the components of $\del Y$, we obtain $h=\sum_{i=1}^r n_i\cdot s_i$.

\begin{lem}
 \label{lem:phiiirreducible}
 In the situation above, for each $1\leq i\leq r$ and $1\le j\le s_i$ we have
 that $\phi_{i,j} \in \Gamma(Y_{i,j},\del Y_{i,j})$ is $\del$-irreducible.
\end{lem}
\begin{proof}
  Suppose that there is an essential arc $\beta\subset Y_{i,j}$ (i.e.\ the
  endpoints of $\beta$ are on $\del Y_{i,j}$) that is fixed up to isotopy by
  $\phi_{i,j}$. Then we can isotope $\Phi$ relative to
  $\caS\setminus \smash{\mathring{Y}_{i,j}}$ so that $\Phi$ fixes $\beta$
  pointwise; without loss of generality, we assume that $\Phi$ already fixes
  $\beta$ pointwise.

  We can extend $\beta$ to an arc $\alpha\subset \caS$ with endpoints on
  $\del \caS$ by joining the endpoints of $\beta$ inside $W$ with $\del\caS$.
  Then $\Phi$ fixes $\alpha$ pointwise. By \iref{Lemma}{lem:cutlocus}, we can
  isotope $\alpha$ into the region $W$. This implies that $\alpha$ is not in
  minimal position with respect to $\del Y_{i,j}$, and therefore $\alpha$ must
  form a bigon with the multicurve $\del Y_{i,j}$. Since $\alpha$ intersects
  $\del Y_{i,j}$ in exactly two points, namely the endpoints of $\beta$, there
  must be a bigon with one side equal to $\beta$ and the other contained in
  $\del Y_{i,j}$.  This bigon is contained in $Y_{i,j}$, contradicting the
  assumption that $\beta$ is essential in $Y_{i,j}$.
\end{proof}

\subsection{The group \texorpdfstring{$\tZ(\phi)$}{Ẑ(φ)} and its
  relation to \texorpdfstring{$Z(\phi,\Gamma(\caS,\del\caS))$}{Z(φ,Γ(S,∂S))}}

In this subsection, we introduce a certain group $\tZ(\phi)$, built out of small
mapping class groups and symmetric groups. The peculiarity of $\tZ(\phi)$ is
that it admits a natural map
$\epsilon\colon \tZ(\phi)\to
Z(\phi,\Gamma(\caS,\del\caS))\subset\Gamma(\caS,\del\caS)$. In the next
subsection, we will identify the kernel of $\epsilon$, and we will prove in the
final subsection that $\epsilon$ is surjective.

Recall that $\phi_Y\in\Gamma(Y,\del Y)$ denotes the mapping class represented by
$\Phi|_Y$.  We consider the centraliser $Z(\phi_Y)\subset\Gamma(Y)$ of $\phi_Y$
in the extended mapping class group $\Gamma(Y)$ using the natural inclusion
$\Gamma(Y,\del Y)\subset\Gamma(Y)$.  Note that $\Gamma(Y)$ admits a natural map
to $\frS_h\cong\frS_{\smash{\pi_0(\del Y)}}$ given by the action of mapping
classes on boundary components. This map restricts to a map
$Z(\phi_Y)\to\frS_h$.

Similarly, we can consider the extended mapping class group $\Gamma^{\frS_h}(W)$
that contains mapping classes of $W$ that fix $\del^{\text{out}}W=\del\caS$
pointwise but may permute the $h$ incoming boundary components of $W$: here we
identify
\[\frS_h\cong\frS_{\pi_0(\del^{\text{in}} W)}\subset\frS_{\pi_0(\del^{\text{out}} W)}\times \frS_{\pi_0(\del^{\text{in}} W)}.\]

\begin{defi}
  \label{defi:tZdelphi}
  We define $\tZ(\phi)$ as the fibre product
  \[\tZ(\phi)\coloneqq  \Gamma^{\frS_h}(W) \times^{\frS_h} Z(\phi_Y).\]
  Gluing $Y$ and $W$ along $\del Y=\del^{\text{in}} W$ yields a map of groups
  \[\hat\epsilon\colon \Gamma^{\frS_h}(W) \times^{\frS_h} \Gamma(Y)
    \to\Gamma(\caS,\del\caS).\]
  Explicitly, for a couple of mapping classes $(\psi_W,\psi_Y)$, we choose
  representatives $\Psi_W\colon W\to W$ and $\Psi_Y\colon Y\to Y$. The fact that
  $\psi_W$ and $\psi_Y$ project to the same permutation of
  $\pi_0(\del^{\text{in}}W)=\pi_0(\del Y)\cong\set{1,\dots,h}$, together with
  the fact that both $\Psi_W$ and $\Psi_Y$ preserve the boundary
  parametrisation, implies that $\Psi_Y|_{\del Y}=\Psi_W|_{\del^{\text{in}}W}$,
  and hence we can glue the two diffeomorphisms to a diffeomorphism of $\caS$
  (we skip all details about smoothing the output homeomorphism near the gluing
  curves).
\end{defi}

\begin{lem}
  \label{lem:epsilon}
  The restriction $\epsilon=\hat\epsilon|_{\tZ(\phi)}$ has image inside
  $Z(\phi,\Gamma(\caS,\del\caS))\subset\Gamma(\caS,\del\caS)$.
\end{lem}
\begin{proof}
  Note that $(\one\hspace*{-.8px}_W,\phi_Y)$ is a central element of $\tZ(\phi)$:
  indeed, given a pair $(\psi_W,\psi_Y)\in \tZ(\phi)$, we have that
  $\psi_Y\in Z(\phi_Y)$, so $\psi_Y$ commutes with $\phi_Y$, and clearly
  $\one_W$ commutes with $\psi_W$ in  $\Gamma^{\frS_h}(W)$.
  Applying $\epsilon$, we obtain that $\epsilon(\psi_W,\psi_Y)$ commutes with
  $\epsilon(\one_W,\phi_Y)=\phi$.
\end{proof}

In the following lemma, we decompose $Z(\phi_Y)$, which is the second factor
appearing in the formula for $\tZ(\phi)$.

\begin{lem}
  \label{lem:ZphiYdecomposition} 
  There is an isomorphism of groups
  \[Z(\phi_Y)\cong\prod_{i=1}^r Z(\bar\phi_i)\wr \fS_{s_i},\]
  where $Z(\bar\phi_i)\subset\Gamma_{g_i,(n_i)}$ is the centraliser in the
  extended mapping class group, and where $\smash{Z(\bar\phi_i)\wr
    \fS_{s_i}=\big(Z(\bar\phi_i)\big)^{s_i}\rtimes\fS_{s_i}}$ denotes the wreath
  product.
\end{lem}
\begin{proof}
  Let $\psi_Y\in Z(\phi_Y)$ be a centralising mapping class, and represent
  $\psi_Y$ by a diffeomorphism $\Psi_Y$ preserving the boundary
  parametrisation. Furthermore, let $P,P'\subset Y$ be two connected components
  with $\Psi_Y(P)=P'$; then restricting the commutativity of $\psi_Y$ and
  $\phi_Y$ to these two components, we obtain the equality
  $\smash{\smash{\phi_Y|}_{P}^{\smash{\Psi_Y}}=\phi_Y|_{P'}}$ in
  $\Gamma(P',\del P')$. This implies that $P$ and $P'$ are similar, and using
  \iref{Notation}{nota:Yi} we have that each $Y_i$ is $\Psi_Y$-invariant and
  therefore
  \[Z(\phi_Y)=\prod_{i=1}^r Z(\phi|_{Y_i}),\]
  where $\phi|_{Y_i}$ is defined using that $Y_i$ is a $\phi$-invariant union of
  connected components of $Y$.  Now fix $1\le i\le r$; using the diffeomorphisms
  $\Xi_{i,j}$ for varying $1\le j\le s_i$ we can identify the surface $Y_i$ with
  $\coprod_{1\le j\le s_i} \Sigma_{g_i,n_i}$ and thus identify $\Gamma(Y_i)$
  with $\Gamma_{g_i,(n_i)}\wr \frS_{s_i}$. Thus, $\phi|_{Y_i}$ corresponds to
  the element
  \[(\bar\phi_i,\dots,\bar\phi_i)\in (\Gamma_{g_i,(n_i)})^{s_i}
    \subset\Gamma_{g_i,(n_i)}\wr \frS_{s_i}.\]
  It follows that $Z(\phi|_{Y_i})$ is isomorphic to $Z(\bar\phi_i)\wr \fS_{s_i}$.
\end{proof}

We conclude the subsection by analysing the actual subgroup of $\frS_h$ over
which the fibre product $\tZ(\phi)$ lives. Now we will focus on the case in
which $W$ is connected, because the exposition is a bit easier: indeed, the
natural map $\Gamma^{\frS_h}(W)\to\frS_h$ is surjective under this hypothesis on
$W$, so we just have to describe the image of the map $Z(\phi_Y)\to\frS_h$.

\begin{nota}
  \label{nota:Hi}
  We denote by $\frH_i\subset \frS_{n_i}$ the image of $Z(\bar\phi_i)$ under the
  natural map $\Gamma_{g_i,(n_i)}\to\frS_{n_i}$.
\end{nota}
The proof of \iref{Lemma}{lem:ZphiYdecomposition} shows that the image of
$Z(\phi_Y)\to\frS_h\cong\frS_{\pi_0(\del Y)}$ is the subgroup
$\prod_i\frH_i\wr \frS_{s_i}$ consisting of those permutations of the set
$\pi_0(\del Y)$ that preserve each subset $\pi_0(\del Y_i)$ for each
$1\le i\le r$, and send each $\pi_0(\del Y_{i,j})$ to some subset
$\pi_0(\del Y_{i,j'})$ in a way that, under the identifications
\mbox{$\pi_0(\del Y_{i,j})\cong\{1,\dots,n_i\}\cong \pi_0(\del Y_{i,j'})$}, gives a
permutation in $\frS_{n_i}$ which can also be attained by projecting an element
$\bar\psi_i\in Z(\bar\phi_i)$.

\subsection{The kernel of \texorpdfstring{$\epsilon$}{ε}}
\label{subsec:kerneleta}
Recall the gluing homomorphism
$\hat\epsilon\colon
\Gamma^{\frS_h}(W)\times^{\frS_h}\Gamma(Y)\to\Gamma(\caS,\del\caS)$ from the
previous subsection. We proceed by identifying the kernel of
$\hat\epsilon$. Note that $\hat\epsilon$ has its image in the subgroup
$\smash{\Gamma(\caS,\del\caS)_{[c_1,\dots,c_h]}}$ of $\Gamma(\caS,\del\caS)$
containing all mapping classes $\psi$ that preserve the cut locus
$[c_1,\dots,c_h]$ of $\phi$, i.e.\ send each oriented homotopy class of a curve
$c_i$ to the oriented homotopy class of some (possibly different) curve $c_j$.
If $(\psi_W,\psi_Y)\in\Gamma^{\frS_h}(W)\times^{\frS_h}\Gamma(Y)$ belongs to the
kernel of $\hat\epsilon$, then in particular $\hat\epsilon(\psi_W,\psi_Y)$ acts
trivially on the components of the cut locus. It follows that both $\psi_W$ and
$\psi_Y$ project to the identity element in $\frS_h$, i.e.\ $(\psi_W,\psi_Y)$ is
in fact contained in the subgroup $\Gamma(W,\del W)\times \Gamma(Y,\del Y)$ of
$\Gamma^{\frS_h}(W)\times^{\frS_h}\Gamma(Y)$. Hence the kernel of $\hat\epsilon$
coincides with the kernel of the restriction of $\hat\epsilon$ to
$\Gamma(W,\del W)\times\Gamma(Y,\del Y)$.

We can now use \cite[Thm.\,3.18]{FarbMargalit}, in the version for disconnected
surfaces: since no component of $W$ or $Y$ is a disc, the kernel of
\[\hat\epsilon\colon\Gamma(W,\del W)\times\Gamma(Y,\del Y)\to\Gamma(\caS,\del\caS)\]
is generated by the couples $(D_{c_i},D_{c_i}^{-1})$, where $D_{c_i}$ denotes
the Dehn twist about the curve $c_i$.
Since each component of $W$ has at least one outgoing boundary, whereas the
curves $c_i$ are incoming for $W$, we can apply \cite[Lem.\,3.17]{FarbMargalit}
to the first coordinates of the elements $(D_{c_i},D_{c_i}^{-1})$ and conclude
that they generate a subgroup of $\Gamma(W,\del W)\times\Gamma(Y,\del Y)$
isomorphic to $\Z^h$.

Finally, we note that the elements $(D_{c_i},D_{c_i}^{-1})$ belong to the
subgroup $\tZ(\phi)$, as $D_{c_i}^{-1}\in\Gamma(Y,\del Y)$, being the inverse of
a boundary twist, is a central element and in particular it commutes with
$\phi_Y$. It follows that the kernel of $\epsilon$ is the free abelian group of
rank $h$ generated by the elements $(D_{c_i},D_{c_i}^{-1})$.

\subsection{Surjectivity of \texorpdfstring{$\epsilon$}{ε}}
We now prove that the map $\epsilon\colon\tZ(\phi)\to Z(\phi,\Gamma(\caS,\del\caS))$
is surjective. In order to do so, let $\psi\in Z(\phi,\Gamma(\caS,\del\caS))
\subset\Gamma(\caS,\del\caS)$ be a centralising mapping class;
see \iref{Definition}{def:centraliser}. Then, by \iref{Lemma}{lem:cutlocusequiv},
$\psi$ preserves the cut locus of $\phi$.

We can fix a representative $\Psi\colon\caS\to\caS$ of $\psi$ that permutes the
curves $c_1,\dots,c_h$ preserving their parametrisation.  In particular $\Psi$
restricts to a diffeomorphism of $W$ and of $Y$, respectively.  Moreover $\Psi$
fixes pointwise $\del\caS=\del^{\text{out}}W$, and both $\Psi|_W$ and $\Psi|_Y$
are diffeomorphisms preserving the boundary parametrisation of $W$ and $Y$,
respectively.  Consider now the mapping class $\phi_Y\in\Gamma(Y,\del Y)$
represented by $\Phi|_Y$, and note that also
$(\Psi|_Y)\circ (\Phi|_Y)\circ (\Psi^{-1}|_Y)$ represents a mapping class in
$\Gamma(Y,\del Y)$, which we denote by $\smash{\phi_Y^{\smash{\Psi|_Y}}}$.

We claim that $\phi_Y=\phi_Y^{\smash{\Psi|_Y}}$ holds in $\Gamma(Y,\del Y)$. To
see this, note that gluing with the identity in $\Gamma(W,\del W)$ gives a map
$\lambda_Y^{\caS}\colon\Gamma(Y,\del Y)\to\Gamma(\caS,\del\caS)$, which is
injective by \cite[Thm.\,3.18]{FarbMargalit}; the claim follows from the
observation that $\lambda_Y^{\caS}(\phi_Y^{\smash{\Psi|_Y}})=\phi^\Psi$, which
by the hypothesis on $\Psi$ is equal to $\phi=\lambda_Y^{\caS}(\phi_Y)$.

It follows that $\Psi|_Y$ represents a mapping class $\psi_Y\in\Gamma(Y)$ that
belongs to $Z(\phi_Y)$. Similarly, $\Psi|_W$ represents a class in
$\Gamma^{\frS_h}(W)$, and it is clear that $\psi_W$ and $\psi_Y$ project to the
same element of $\frS_h$, i.e. the couple $(\psi_W,\psi_Y)$ gives an element of
$\tZ(\phi)$. It is also evident that $\epsilon(\psi_W,\psi_Y)=\psi$. This
implies that $\epsilon$ is surjective.  Putting together the discussion of this
and the previous subsections, we conclude the following.

\begin{prop}
  \label{prop:Zdelphistructure}
  Let $g\ge0, n\ge1$, let $\caS$ be a surface of type $\Sigma_{g,n}$, let
  $\phi\in\Gamma(\caS,\del\caS)$, let $c_1,\dots,c_h$ be a system of oriented
  curves representing the cut locus of $\phi$, and let $W$ and $Y$ be the
  corresponding white and yellow regions of $\caS$, using \iref{Notation}{nota:Yi}.
  Then there is an isomorphism of groups induced by $\epsilon$
  \[\tZ(\phi)/\Z^h=\quot{\pa{\Gamma^{\frS_h}(W)\times^{\frS_h}
        \prod_{i} Z(\bar\phi_i)\wr \fS_{s_i}}}{\Z^h}
    ~\overset{\cong}{\longrightarrow}~ Z(\phi,\Gamma(\caS,\del\caS)),\]
  where $\Z^h$ is the free abelian group generated by the elements
  $(D_{c_i},D_{c_i}^{-1})$.  If, moreover, $W$ is connected, we can use
  \iref{Notation}{nota:Hi} and rewrite the isomorphism as
  \[\quot{\pa{\Gamma^{\prod_{i}\frH_i\wr \fS_{s_i}}(W)\times^{\prod_{i}\frH_i\wr \fS_{s_i}}
        \prod_{i} Z(\bar\phi_i)\wr \fS_{s_i}}}{\Z^h}
    ~\overset{\cong}{\longrightarrow}~ Z(\phi,\Gamma(\caS,\del\caS)).\]
\end{prop}

\section{Generalities on coloured operads}
\label{sec:colouredOHS}
In this section, we establish the operadic framework that we will use in the
rest of the article. The reader who is well-acquainted with the language of
coloured operads may skip this interlude and go directly to
\iref{Section}{sec:infiniteLSFromOHS}.

\subsection{Notation and diagram categories}

By ‘space’, we mean a topological space that is compactly generated and has the
weak Hausdorff property. Let $\Top$ be the category of spaces; it is Cartesian
closed, complete, and cocomplete.

For a topologically enriched category $\bfI$ and objects $k,n\in\ob(\bfI)$,
we denote by $\bfI\binom{k}{n}$ the space of morphisms from $k$ to $n$, and we
denote the identity of $n$ by $\bEn_n\in \bfI\binom{n}{n}$.

\begin{nota}
  Let $\Inj$ be the small category with objects $\ul{r}=\set{1,\dots,r}$ for all
  non-negative integers $r\in \{0,1,2,\dotsc\}$, and with morphisms
  $\ul{r}\to \ul{r}'$ being all injective maps of finite sets. Moreover, let
  $\bSig\subseteq\Inj$ be the subcategory of all bijective maps.
  
  The category $\Inj$ is spanned by two sorts of maps: on the one hand,
  \emph{permutations} $\tau\in\frS_r$, which constitute the category $\bSig$,
  and on the other hand, the \emph{top cofaces}
  $d^r\colon \ul{\smash{r-1}}\to \ul{r}$, where for each $1\le i\le r$, we
  denote by $d^i$ the unique strictly monotone function whose image does not
  contain the element $i\in\ul{r}$. Whenever we apply a contravariant functor to
  $\Inj$, we write $d_i\coloneqq (d^i)^*$.
\end{nota}

\begin{nota}\label{nota:tuple}
  Let $N$ be a fixed set, $r\ge 0$, and let $K=(k_1,\dotsc,k_r)$ be a tuple of
  elements of $N$.  We write $\# K\coloneqq r$ for the \emph{length} of $K$.  If
  $u\colon \ul{s}\hookrightarrow \ul{r}$ is a map in $\mathbf{Inj}$, we write
  $u^*K \coloneqq \big(k_{u(1)},\dotsc,k_{u(s)}\big)$.  If
  $\bm{X}\coloneqq (X_n)_{n\in N}$ is a family of objects in some category with
  finite products, we write
  $\bm{X}(K) \coloneqq X_{k_1} \times \dotsb\times X_{k_r}$;

  For two tuples $K=(k_1,\dotsc,k_r)$ and $K'=(k'_1,\dotsc,k'_{r'})$ of elements
  of $N$, we denote by
  $\smash{\Inj \binom{K\,}{K'}\subset\Inj\binom{\ul{r}\,}{\ul{r}'}}$ the subset of those
  morphisms $u\colon \ul{r}\hookrightarrow \ul{r}'$ satisfying $u^*K'=K$, and we
  denote by $\bSig\binom{K\,}{K'}$ the intersection
  $\Inj\binom{K\,}{K'}\cap\bSig\binom{\ul{r}\,}{\ul{r}'}$.
\end{nota}

\begin{defi}\label{defi:equigroth}
  Let $N$ be a set and $\bm{G}\coloneqq (G_n)_{n\in N}$ be a family
  of topological groups. We define the \emph{wreath product} $\bm{G}\wr \Inj$ as
  the following topologically enriched category:\looseness-1
  \begin{enumerate}
  \item the objects of $\bm{G}\wr\Inj$ are all tuples $K=(k_1,\dotsc,k_r)$ with
    $r\ge 0$ and $k_i\in N$;
  \item for two tuples $K$ and $K'$, we define
    $(\bm{G}\wr\Inj)\smash{\binom{K}{K'}} = \bm{G}(K)\times
    \Inj\smash{\binom{K\,}{K'}}$, i.e.\ a morphism $K\to K'$ in $\bm{G}\wr\Inj$
    is a pair $(\bm{\gamma},u)$ consisting of a tuple $\bm{\gamma}\in\bmG(K)$,
    and of an injective map $u\colon \ul{r}\hookrightarrow \ul{r}'$ satisfying
    $K=u^*K'$;\vspace*{-2px}
  \item we let $(\bm{\gamma}',u')\circ (\bm{\gamma},u) \coloneqq (u^*\bm{\gamma}'\cdot \bm{\gamma},u'\circ u)$,
    where $u^*\bm{\gamma}' =(\gamma'_{u(1)},\dotsc,\gamma'_{u(r)})$, and
    ‘$\cdot$’ denotes component-wise multiplication.
  \end{enumerate}
  For each tuple $K$, we define $\bm{G}[K]\subseteq \bm{G}\wr \Inj$ as the full
  subcategory spanned by objects of the form $\tau^*K$ for $\tau\in\frS_r$.
  Moreover we let $\bm{G}\wr\bSig$ be the subgroupoid with morphism spaces given
  by $\bmG(K)\times \bSig\binom{K\,}{K'}$.

  If $(G_n)_{n\in N}$ is the trivial sequence $G_n=1$ of groups, then we write
  $N\wr \Inj$ for the wreath product, and we also write $N\wr\bSig$ and $N[K]$
  for the corresponding subgroupoids, respectively.
\end{defi}

\begin{expl}\label{ex:easywreath}
  If $\bm{X}=(X_n)_{n\in N}$ is a sequence of spaces, then we obtain a functor
  \[\bm{X}\colon N\wr\bSig\to \mathbf{Top},
    \quad K\mapsto \bm{X}(K)=X_{k_1}\times\dotsb\times X_{k_r}.\]
\end{expl}

\begin{constr}
  Let $\bm{X}\coloneqq (X_n)_{n\in N}$ be a family of based spaces, together
  with based left actions of $G_n$ on $X_n$ for each $n\in N$. Then the functor
  from \iref{Example}{ex:easywreath} can be extended to a functor
  $\bm{G}\wr\Inj\to \Top$ as follows. For each injective map
  $u\colon \ul{r}\hookrightarrow \ul{r}'$, each fibre has at most one element;
  therefore we obtain an extension $N\wr\Inj\to \Top$ by
  \[u_*(x_1,\dotsc,x_r)\coloneqq (x_{\smash{u^{-1}(1)}},\dotsc,x_{\smash{u^{-1}(r')}}),\]
  where we define $x_\varnothing$ to be the basepoint.
  Moreover $\bm{G}(K)$ acts on $\bm{X}(K)$ component-wise, so for a morphism
  $(\bm{\gamma},u)$ in $\bm{G}\wr\Inj$ we can define
  $(\bm{\gamma},u)_*(x) \coloneqq u_*(\bm{\gamma}\cdot x)$.
\end{constr}

\begin{defi}
  Let $\bfI$ be a small and topologically enriched category and, moreover, let
  \mbox{$H\colon \bfI^\op\times\bfI\to {\large\Top}$} be a functor. Then we define the
  \emph{coend} to be the coequaliser
  \[\hspace*{-2px}\int^{k\in\bfI}\hspace*{-3px}H(k,k) \coloneqq \on{coeq}
    \bigg(\hspace*{-3px}
    \begin{tikzcd}[column sep=7.4em]
      \displaystyle\coprod_{k,n}\bfI\binom{k}{n}\times H(n,k)
      \ar[shift left=1]{r}{(f,x)\mapsto H(f,\bEn_{k})(x)}
      \ar[shift right=1]{r}[swap]{(f,x)\mapsto H(\bEn_{n},f)(x)} &
      \displaystyle\coprod_kH(k,k)
    \end{tikzcd}\hspace*{-6px}\bigg)\hspace*{-1px}.\]
\end{defi}

\subsection{Coloured operads }

We assume that the reader is familiar with the classical notion of an operad, as
it is for example presented in \cite{MSS}, in particular, the visualisation of
operations by trees is taken for granted.

We will give a brief introduction to the notion of a \emph{coloured} operad
mostly for the purpose of fixing the notation we will use later. For a detailed
introduction to coloured operads, we refer the reader to \cite{Yau}.

\begin{defi}\label{defi:Opd}
  Let $N$ be a fixed set. An \emph{$N$-coloured operad} is a family of functors
  $\scO\binom{-}{n}\colon (N\wr\mathbf{\Sigma})^\op\to \mathbf{Top}$ for each
  $n\in N$, together with:
  \begin{enumerate}
  \item choices of identities $\bEn_n\in \scO\binom{n}{n}$;\vspace*{-3px}
  \item composition maps for each $n,k_i,l_{ij}\in N$, which are of the form
    \[\scO\binom{k_1,\dotsc,k_r}{n}\times\prod_{i=1}^r
      \scO\binom{l_{i1},\dotsc,l_{is_i}}{k_i}\to
      \scO\binom{l_{11},\dotsc,l_{rs_r}}{n},\quad\!
      (\mu;\mu'_1,\dotsc,\mu'_r)\mapsto \mu\circ (\mu'_1,\dotsc,\mu'_r);\]
  \end{enumerate}
  such that the usual coherence axioms \cite[§\,11.2]{Yau} are satisfied.
  For \mbox{$\mu\in \scO\binom{k_1,\dotsc,k_r}{n}$}, we call $n$ the
  \emph{output}, $(k_1,\dotsc,k_r)$ the \emph{input profile}, and
  $\# \mu \coloneqq r$ the \emph{arity} of $\mu$. For the first few values of
  $r$, we call $\mu$ \emph{nullary}, \emph{unary}, or \emph{binary} if $\mu$ has
  arity $0$, $1$, or $2$, respectively. For the empty tuple, we write
  $\scO\binom{}{n}$ for the space of nullaries.

  We say that $\scO$ is \emph{$\frS$-free} if for each $K=(k_1,\dotsc,k_r)$, the
  subgroup $\frS_K\subseteq\frS_r$ which fixes the tuple $K$ acts freely on
  $\scO\binom{K}{n}$.

  We call $\scO$ \emph{monochromatic} if $N=*$ is just a singleton. In this
  case, we also write $\scO(r)\coloneqq \scO\binom{*,\dotsc,*}{*}$ for the space
  of $r$-ary operations. For an $N$-coloured operad $\scO$ and a fixed colour
  $n\in N$, we also consider the monochromatic operad $\scO|_n$ with operation
  spaces $(\scO|_n)(r) \coloneqq \scO\binom{n,\dotsc,n}{n}$.

  Additionally, we use the short notation for ‘partial’ composition: for
  \mbox{$\mu\in \scO\binom{k_1,\dotsc,k_r}{n}$} and
  $\mu'\in\scO\binom{l_1,\dotsc,l_s}{k_i}$, we write
  \[\mu\circ_i \mu' \coloneqq \mu\circ
    (\bEn_{k_1},\dotsc,\bEn_{k_{i-1}},\mu',\bEn_{k_{i+1}},\dotsc,\bEn_{k_r})\in
    \scO\binom{k_1,\dotsc,k_{i-1},l_1,\dotsc,l_s,k_{i+1},\dotsc,k_r}{n}.\]
\end{defi}

\begin{expl}
  \begin{enumerate}
  \item The little $d$-discs operads $\scD_d$ for $1\le d\le\infty$ are examples
    of monochromatic operads. In our setting, $\scD_d(0)=\{\frv\}$, the
    single nullary operation being the empty configuration of discs (thus,
    $\frv$ stands for ‘void’).\looseness-1
  \item Each small topologically enriched category $\bfI$ is a coloured operad
    with colour set $\ob(\bfI)$ and only unaries.
  \end{enumerate}
\end{expl}

\begin{defi}
  Let $\scO$ be an $N$-coloured operad. An \emph{$\scO$-algebra} is an $N$-indexed family
  $\bm{X}\coloneqq (X_n)_{n\in N}$ of spaces, together with maps
  \[\scO\binom{K}{n}\times \bm{X}(K)\to X_n,\quad
    (\mu;x_1,\dotsc,x_r)\mapsto \mu(x_1,\dotsc,x_r)\]
  such that the usual coherence axioms from \cite[§\,13]{Yau} are satisfied.

  A morphism $f\colon \bm{X}\to \bm{X}'$ of $\scO$-algebras is a family
  $(f_n\colon X_n\to X'_n)_{n\in N}$ of maps satisfying
  $f_n(\mu(x_1,\dotsc,x_r)) = \mu(f_{k_1}(x_1),\dotsc,f_{k_r}(x_r))$ for all
  operations $\mu$ and all elements $x_i$. This gives rise to the category
  $\Alg{\scO}$.
\end{defi}

\begin{expl}\label{ex:initialAlg}
  For each $N$-coloured operad $\scO$, we have an $\scO$-algebra by the family
  $(\scO\binom{}{n})_{n\in N}$. This algebra is initial in $\Alg{\scO}$ by
  construction, so we call it the \emph{initial $\scO$-algebra}.  For
  $1\le d\le\infty$, the initial $\scD_d$-algebra is just a single point.
\end{expl}

\begin{defi}\label{defi:baseChange}
  For a fixed colour set $N$, a morphism $\rho\colon \scP\to \scO$ of
  $N$-coloured operads is a family
  $\rho^-_n\colon \scP\binom{-}{n}\Rightarrow \scO\binom{-}{n}$ of
  transformations such that we have, abbreviating $\rho\coloneqq\rho^K_n$ for
  all $K$ and $n$,
  \begin{enumerate}
  \item $\rho(\bEn^{\scP}_n)=\bEn^{\scO}_n$ for each $n\in N$;\vspace*{-3px}
  \item $\rho(\mu)\circ (\rho(\mu'_1),\dotsc,\rho(\mu'_r)) =
    \rho(\mu\circ (\mu'_1,\dotsc,\mu'_r))$.
  \end{enumerate}
  Each operad morphism $\rho\colon \scP\to \scO$ gives rise to a
  \emph{base-change adjunction}
  \[\rho_!\colon \Alg{\scO}\rightleftarrows \Alg{\scP}\lon \rho^*\]
  as follows: each $\scO$-algebra is a $\scP$-algebra by restriction. For the
  converse, we consider the \emph{absolute} adjunction
  \mbox{$F^\scO\colon \Top^N\rightleftarrows \Alg{\scO}\lon U^\scO$,} where
  $U^\scO$ forgets the action, and where for each $N$-family $\bmX$
  of spaces, we put \mbox{$F^\scO(\bmX)_n \coloneqq \smash{\int^{K\in N\wr\mathbf{\Sigma}}
    \scO\binom{K}{n}\times \bm{X}(K)}$}. Then each
  $\scP$-algebra $\bmX$ is the reflexive coequaliser of
  $F^\scP U^\scP F^\scP U^\scP \bmX \rightrightarrows F^\scP U^\scP\bmX$, whence
  the $\scO$-algebra $\rho_!\bmX$ is the reflexive coequaliser of
  \mbox{$F^\scO U^\scP F^\scP U^\scP \bmX \rightrightarrows F^\scO U^\scP\bmX$},
  compare \cite[§\,4]{BergerMoerdijk2}. Intuitively $\rho_!\bmX$ is a quotient
  of the free $\scO$-algebra over $\bmX$ by the existing $\scP$-action on $\bmX$.
    
  This adjunction clearly respects compositions: if $\rho\colon \scQ\to \scP$
  and $\rho'\colon \scP\to \scO$ are two morphisms of $N$-coloured operads, then we
  clearly have $(\rho'\circ\rho)^* = \rho^*\circ \rho'{}^*$, so by the
  uniqueness of left adjoints, we also have $(\rho'\circ\rho)_! \cong
  \rho'_!\circ \rho_!$.
  
  When $\rho$ is clear from the context, we also write
  $F^{\scO}_{\scP}\colon \Alg{\scP}\rightleftarrows \Alg{\scO}\lon U^{\scO}_{\scP}$ for
  the base-change adjunction.
\end{defi}

\subsection{The coloured surface operad}
\label{subsec:scM}
We define an $\N_{\ge 1}$-coloured operad $\scM$ which is a ‘coloured’ version
of Tillmann’s surface operad \cite{Tillmann00}, see
\iref{Figure}{fig:colouredM}.

We first recall Segal’s cobordism category $\bfM$ \cite{cft}, which is a
topologically enriched category: objects of $\bfM$ are non-negative integers
$n\ge 0$; a morphism from $n$ to $n'$ is represented by a (possibly
disconnected) Riemann surface $W$ with $n$ incoming and $n'$ outgoing boundary
components; the surface $W$ is equipped with a choice of collar neighbourhoods
$U^{\text{in}}_{\del W}$ and $U^{\text{in}}_{\del W}$ of $\del^{\text{in}}W$ and
$\del^{\text{out}}W$ respectively; these neighbourhoods are equipped with:
\begin{enumerate}
\item a \emph{holomorphic} parametrisation
  $\tilde\theta^{\text{in}}\colon \set{1,\dots,n}\times S^1\times[0;1)\to
  U^{\text{in}}_{\del W}$; and\vspace*{-3px}
\item an \emph{antiholomorphic} parametrisation
  $\tilde\theta^{\text{out}}\colon \set{1,\dots,n'}\times S^1\times[0;1)\to
  U^{\text{out}}_{\del W}$.
\end{enumerate}
Note that restricting $\tilde\theta^{\text{in}}$ and $\tilde\theta^{\text{out}}$
to $\set{1,\dots,n}\times S^1\times0$ and $\set{1,\dots,n'}\times S^1\times0$,
we obtain parametrisations of $\del^{\text{in}}W$ and $\del^{\text{out}}W$,
respectively, as in \iref{Notation}{nota:surface}.

The space of morphisms $\bfM\binom{n\,}{n'}$ is the moduli space of conformal
classes of such Riemann surfaces $W$, considered up to biholomorphism compatible
with the choice of parametrised collar neighbourhoods of the incoming and
outgoing boundary. We usually denote by $(W,\tilde\theta)$ a morphism, or
shortly by $W$ when it is not necessary to mention the parametrisation of the
collar neighbourhood of the boundary; here
$\smash{\tilde\theta\colon\set{1,\dots,n+n'} \times S^1\times [0;1)\to W}$ is
obtained by concatenation of $\tilde\theta^{\text{in}}$ and
$\tilde\theta^{\text{out}}$.

\begin{figure}
  \centering
  \begin{tikzpicture}[xscale=.7,yscale=.76]
    \draw[looseness=.35,thick,dred] (0,3) to[out=-90,in=-90] ++(1,0) to[out=90,in=90] (0,3);
    \draw[looseness=.35,thick,dgreen] (2,3) to[out=-90,in=-90] ++(1,0) to[out=90,in=90] (2,3);
    \draw[looseness=.35,thick,dgreen] (4,3) to[out=-90,in=-90] ++(1,0) to[out=90,in=90] (4,3);
    \draw[looseness=.35,thick,dred] (6,3) to[out=-90,in=-90] ++(1,0) to[out=90,in=90] (6,3);
    \draw[looseness=.35,thick,dyellow] (8,3) to[out=-90,in=-90] ++(1,0) to[out=90,in=90] (8,3);
    \fill[white, opacity=.8] (0,2.5) rectangle (11,3);   
    \draw[looseness=.3,thick] (1,0) to[out=-90,in=-90] ++(1,0) to[out=90,in=90] (1,0);
    \draw[looseness=.3,thick] (5,0) to[out=-90,in=-90] ++(1,0) to[out=90,in=90] (5,0);
    \draw[looseness=.3,thick] (7,0) to[out=-90,in=-90] ++(1,0) to[out=90,in=90] (7,0);
    \draw[looseness=.3,thick] (10,0) to[out=-90,in=-90] ++(1,0) to[out=90,in=90] (10,0);
    \draw (1,0) to[out=85,in=-90]  (0,3);
    \draw (2,0) to[out=95,in=-90]  (3,3);
    \draw[looseness=3.5] (1,3) to[out=-85,in=-95] (2,3);
    \draw[looseness=2.5] (5,3) to[out=-85,in=-95] (6,3);
    \draw[looseness=3.5] (7,3) to[out=-85,in=-95] (8,3);
    \draw[looseness=3.5] (6,0) to[out=85,in=95]   (7,0);
    \draw (4,3) to [out=-90,in=90] ++(-.5,-1.5) to [out=-90,in=95] (5,0);
    \draw (8,0) to[out=95,in=-90]  (9,3);
    \draw (10,0) to[out=85,in=-90] (9.5,1.5);
    \draw (11,0) to[out=95,in=-90] (11.5,1.5);
    \draw[looseness=1.3] (9.5,1.5) to[out=90,in=90] (11.5,1.5);
    \draw[looseness=1] (10.1,1.5) to[out=-90,in=-90] ++(.8,0) to[out=90,in=90] (10.1,1.5);
    \draw[thin] (10.15,1.38) to[out=130,in=-90] (10.05,1.6);
    \draw[thin] (10.85,1.38) to[out=50 ,in=-90] (10.95,1.6);
    \draw[looseness=1] (4.15,1.6) to[out=-90,in=-90] ++(.8,0) to[out=90,in=90] (4.15,1.6);
    \draw[thin] (4.2,1.48) to[out=130,in=-90] (4.1,1.7);
    \draw[thin] (4.9,1.48) to[out=50 ,in=-90] (5,1.7);
    \node[anchor=base,dred] at (.5,3.23) {\tiny $2$};
    \node[anchor=base,dred] at (6.5,3.23) {\tiny $1$};
    \node[anchor=base,dgreen] at (2.5,3.23) {\tiny $1$};
    \node[anchor=base,dgreen] at (4.5,3.23) {\tiny $2$};
    \node[anchor=base,dyellow] at (8.5,3.23) {\tiny $1$};
    \node[anchor=base] at (1.5,-.4) {\tiny $1$};
    \node[anchor=base] at (5.5,-.4) {\tiny $4$};
    \node[anchor=base] at (7.5,-.4) {\tiny $2$};
    \node[anchor=base] at (10.5,-.4) {\tiny $3$};
  \end{tikzpicture}
  \caption{An element in
    $\scM\binom{\textcolor{dgreen}{2},\textcolor{dyellow}{1},\textcolor{dred}{2}}{4}$.
    Note that the colours green, yellow, and red only indicate which inputs
    belong together, while the actual ‘colours’ of the inputs are $2$, $1$, and $2$
    respectively.}\label{fig:colouredM}
\end{figure}

The composition of two morphisms $(W,\tilde\theta)\colon n\to n'$ and
$(W',\tilde\theta')\colon n'\to n''$ is given by gluing the Riemann surfaces
$W\setminus\del^{\text{out}}W$ and $W'\setminus\del^{\text{in}}W'$, using the
identification
$U^{\text{out}}_{\del W}\setminus\del^{\text{out}}W \cong U^{\text{in}}_{\del
  W'}\setminus\del^{\text{in}}W'$ given by
\[\tilde\theta^{\text{out}}(j,\zeta,t)\equiv(\tilde\theta')^{\text{in}}(j,\zeta,1-t),\]
for all $1\le j\le n'$, $\zeta\in S^1$ and $0<t<1$. The resulting surface $W''$
is also endowed with collar neighbourhoods of the incoming and outgoing
boundaries, whose parametrisations are given by $\tilde\theta^{\text{in}}$ and
$(\tilde\theta')^{\text{out}}$ respectively.  The identity of $n\in \bfM$ is
described in \iref{Construction}{constr:Rnembedding}.

\begin{defi}\label{def:twisted_torus}
  For each $n\ge1$, the Lie group $T^n\rtimes \frS_n=(S^1)^n\rtimes \frS_n$ will
  be denoted by $R_n$. We regard $R_n$ as a \emph{twisted torus}: it
  is the isometry group of $\coprod_n S^1$.
\end{defi}

\begin{constr}\label{constr:Rnembedding}
  We can embed $R_n$ in the endomorphism space $\bfM\binom{n}{n}$; see
  \iref{Figure}{fig:tinm}: given $(z_1,\dotsc,z_n,\sigma)\in R_n$, we consider
  the morphism \mbox{$(W,\tilde\theta)\colon n\to n$}, which is given by the
  following:\looseness-1
  \begin{enumerate}
  \item we let $W\coloneqq\set{1,\dots,n}\times S^1\times [0;1]$, with
    the standard complex structure;\vspace*{-3px}
  \item we let $U_{\del W}^{\text{in}}=W\setminus\del^{\text{out}}W$ and
    $\tilde\theta^{\text{in}}\colon \set{1,\dots,n}\times
    S^1\times[0;1)\hookrightarrow W$.\vspace*{-3px}
  \item we let $U_{\del W}^{\text{out}}=W\setminus\del^{\text{in}}W$ and define
    $\tilde\theta^{\text{out}}\colon\set{1,\dots,n}\times S^1\times[0;1)\hookrightarrow W$ by
    \[\tilde\theta^{\text{out}}(j,\zeta,t)=\mathopen{}\big(\sigma^{-1}(j),z_j\cdot
        \zeta,1-t\big)\mathclose{}.\]
  \end{enumerate}
  This assignment embeds in fact $R_n$ as a group into the automorphisms of
  $n\in\bfM$; in particular the identity of $n$ can be described as the image of
  the unit of $R_n$ along the embedding; see \iref{Figure}{fig:tinm}.
\end{constr}

\begin{figure}[h]
  \centering
  \begin{tikzpicture}[xscale=1,yscale=.7]
    \node at (-.2,0) {$(z_1,\dotsc,z_n,\sigma)~ = $};
    \draw[looseness=.45,dred,thick] (2,1) to[out=90,in=90] ++(1,0) to[out=-90,in=-90] (2,1);
    \draw[looseness=.45,dred,thick] (5,1) to[out=90,in=90] ++(1,0) to[out=-90,in=-90] (5,1);
    \fill[white,opacity=.85] (1.5,1) rectangle (7.5,.5);
    \draw (2,-1) -- (2,1);
    \draw (3,-1) -- (3,1);
    \draw (5,-1) -- (5,1);
    \draw (6,-1) -- (6,1);
    \draw[looseness=.4,thick] (2,-1) to[out=90,in=90] ++(1,0) to[out=-90,in=-90] (2,-1);
    \draw[looseness=.4,thick] (5,-1) to[out=90,in=90] ++(1,0) to[out=-90,in=-90] (5,-1);
    \node[anchor=base] at (2.3,-1.45) {\tiny $\sigma(1)$};
    \node[anchor=base] at (5.3,-1.45) {\tiny $\sigma(n)$};
    \node[dred] at (2.3,1.38) {\tiny $1$};
    \node[dred] at (5.3,1.35) {\tiny $n$};
    \draw[bblue,thick] (2.8,1.1) -- (2.8,-.89);
    \node[bblue] at (2.85,1.3) {\tiny $1$};
    \node[bblue] at (2.9,-1.3) {\tiny $z_1$};
    \draw[bblue,thick] (5.8,1.1) -- (5.8,-.89);
    \node[bblue] at (5.85,1.3) {\tiny $1$};
    \node[bblue] at (5.9,-1.3) {\tiny $z_n$};
    \node at (4,0) {$\dotsb$};
  \end{tikzpicture}
  \caption{An instance of $R_n\hookrightarrow \bfM\binom{n}{n}$.}\label{fig:tinm}
\end{figure}
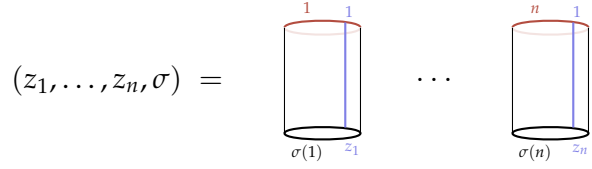

Consider now the subcategory $\bfM_\partial\subseteq\bfM$ containing all objects
and those cobordisms $W$ whose components have non-empty outgoing boundary.  In
detail, for fixed $k,n\ge 0$, the morphism space $\bfM_\partial(k,n)$ looks as
follows: $\pi_0\pa{\bfM_\partial(k,n)}$ is indexed by the number $1\le l\le n$
of path components, an unordered partition
$\{1,\dotsc,n\}=\bm{n}_1\sqcup\dotsb\sqcup \bm{n}_l$ into non-empty subsets with
$\min(\bm{n}_j)<\min(\bm{n}_{j+1})$, an ordered partition
$\{1,\dotsc,k\} = \bm{k}_1\sqcup\dotsb\sqcup \bm{k}_l$ into possibly empty sets,
and genera $g_1,\dotsc,g_l\ge 0$. If we let $n_j\coloneqq\#\bm{n}_j$ and
$k_j\coloneqq \# \bm{k}_j$, then the corresponding path component is homotopy
equivalent to $\frM_{g_1,k_1+n_1}\times\dotsb\times\frM_{g_l,k_l+n_l}$, by
restricting collar parametrisations to boundary parametrisations.

Let $n,n'\in\bfM_\del$, and note that the embedding $R_n\subset\bfM\binom{n}{n}$
has image inside $\bfM_\del\binom{n}{n}$. In the following lemma we consider the
right action of $R_n$ on $\bfM_\del\binom{n}{n'}$ by precomposition.

\begin{lem}
  \label{lem:freeactionRn}
  The group $R_n$ acts freely on the space $\bfM_\del\binom{n\,}{n'}$.
\end{lem}
\begin{proof}
  Let $(W,\tilde\theta)\colon n\to n'$ be a morphism in $\bfM\binom{n\,}{n'}$,
  let $(z_1,\dotsc,z_n,\sigma)\in R_n$, and suppose that
  $(W,\tilde\theta)=(W,\tilde\theta) \circ (z_1,\dotsc,z_n,\sigma)$. Note that the
  composed morphism \mbox{$(W,\tilde\theta) \cdot (z_1,\dotsc,z_n,\sigma)$} is represented by the
  pair $(W',\tilde\theta')$, where:
  \begin{enumerate}
  \item $W'=W$ and $(\tilde\theta')^{\text{out}}=\tilde\theta^{\text{out}}$;\vspace*{-2px}
  \item $(\tilde\theta')^{\text{in}}$ is the postcomposition of
    $\tilde\theta^{\text{in}}\colon \set{1,\dots,n}\times S^1\times[0;1)\to W$
    with the automorphism of $\set{1,\dots,n}\times S^1\times[0;1)$ given by
    \[(j,\zeta,t)\mapsto\mathopen{}\big(\sigma^{-1}(j),z_j\cdot \zeta,t\big)\mathclose{}.\]
  \end{enumerate}
  If $\psi\colon W\to W'$ is a diffeomorphism exhibiting the equivalence of
  $(W,\tilde\theta)$ and $(W',\tilde\theta')$ in $\bfM_\del\binom{n\,}{n'}$, then
  the first two conditions imply that $\psi$ restricts to the identity of $U_{\del
    W}^{\text{out}}$, i.e.\ on the image of $\tilde\theta^{\text{out}}$. Since
  $\psi$ is a holomorphic map, it must be the identity on a closed and open subset
  of $W$; since each connected component of $W$ has non-empty outgoing boundary,
  and thus intersects $U_{\del W}^{\text{out}}$, we conclude that $\psi$ must be
  the identity of $W$.
  Thus, the automorphism of $\set{1,\dots,n}\times S^1\times[0;1)$
  given by \mbox{$(j,\zeta,t)\mapsto(\sigma^{-1}(j),z_j\cdot \zeta,t)$} is in fact the
  identity of $\set{1,\dots,n}\times S^1\times[0;1)$, and this implies that
  $(z_1,\dots,z_n,\sigma)$ is the identity of $R_n$, as desired.
\end{proof}

\begin{defi}\label{defi:scM}
  We consider $\bfM_\partial$ as a symmetric monoidal category with monoidal sum
  being the disjoint union; since the monoidal sum behaves as the usual sum of
  natural numbers on objects, we have an associated coloured operad $\scM$ with
  colours $\N_{\ge 1}=\{1,2,\dotsc\}$ and
  \[\scM\binom{k_1,\dotsc,k_r}{n} \coloneqq \bfM_\partial\binom{k_1+\dotsb+k_r}{n}.\]
  The restriction $\scM|_1$ to the colour 1 is exactly Tillmann’s
  surface operad \cite{Tillmann00}. For each $\scM$-algebra
  $\bm{X}=(X_n)_{n\ge 1}$, the space $X_1$ is an algebra over $\scM|_1$.%
\end{defi}
 
\begin{expl}\label{ex:initialscM}
  In contrast to the little discs operads, the initial $\scM$-algebra is
  non-trivial: for instance, its colour-1 part $\scM\binom{}{1}$ homotopy
  equivalent to the familiar collection of moduli spaces
  \[\scM\binom{}{1} = \bfM_\partial\binom{0}{1} \simeq \coprod_{g\ge 0}\frM_{g,1}.\]
\end{expl}

\subsection{Tensor products and based operads}
\enlargethispage{\baselineskip}

\begin{defi}
  In \cite[§\,\textsc{ii}.3]{BoardmanVogtTensor}, Boardman and Vogt constructed
  a tensor product for operads. We are only interested in the following special
  case: let $\scA$ be a monochromatic operad and $\bfI$ be a small,
  topologically enriched category with object set $N$. Then $\scA\otimes \bfI$
  is an $N$-coloured operad with operation spaces
  \[(\scA\otimes\bfI)\binom{k_1,\dotsc,k_r}{n}
    = \scA(r)\times \prod_{i=1}^r \bfI\binom{k_i}{n},\]
  together with the following structure, where we denote operations in
  $\scA\otimes\bfI$ by $\mu\otimes (\nu_1,\dotsc,\nu_r)$ with $\mu\in \scA(r)$
  and $\nu_i\in\bfI\binom{k_i}{n}$:
  \begin{enumerate}
  \item symmetric actions
    $\tau^*(\mu\otimes (\nu_1,\dotsc,\nu_r)) = (\tau^*\mu)\otimes
    (\nu_{\tau(1)},\dotsc,\nu_{\tau(r)})$;\vspace*{-2px}
  \item identities $(\bEn^{\scA},\bEn^{\bfI}_n)$;\vspace*{-2px}
  \item compositions
    {
      \small
      \begin{align*}
        & \pa{\mu\otimes (\nu_1,\dotsc,\nu_r)}\circ \pa{\mu_1\otimes
          (\nu_{1,1},\dotsc,\nu_{1,s_1}),\dotsc,\mu_r\otimes (\nu_{r,1},\dotsc,\nu_{r,s_r})}
        \\ &\hspace*{0.4cm}\coloneqq \pa{\mu\circ (\mu_1,\dotsc,\mu_r)}\otimes
             \pa{\nu_1\circ\nu_{1,1},\dotsc,\nu_1\circ\nu_{1,s_1},\dotsc,
             \nu_r\circ\nu_{r,1},\dotsc,\nu_r\circ\nu_{r,s_r}}.
      \end{align*}
    }
  \end{enumerate}
  For $n\in N$, we also abbreviate
  $\mu\otimes n\coloneqq \mu\otimes (\bEn_{n},\dotsc,\bEn_{n})\in
  (\scA\otimes\bfI)\binom{n,\dotsc,n}{n}$. Note that
  $(\scA\otimes\bfI)$-algebras are the same as enriched functors
  $\bfI\to \Alg{\scA}$.

  This construction is bifunctorial: if $\rho_1\colon \scA\to \scA'$ is a
  morphism of operads and $\rho_2\colon\bfI\to \bfI'$ is a functor that is the
  identity on objects, then we get a morphism
  $\rho_1\otimes \rho_2\colon \scA\otimes\bfI\to {\scA'}\otimes{\bfI'}$.
\end{defi}

\begin{expl}\label{ex:CtensorN}
  Regard $N$ as the discrete category with objects $N$. Then we get
  \[(\scA\otimes N)\binom{k_1,\dotsc,k_r}{n} = \begin{cases}\scA(r) & \text{for
        $k_1=\dotsb=k_r=n$}\\\emptyset & \text{else,}\end{cases}\]
  and $(\scA\otimes N)$-algebras are just $N$-indexed families of
  $\scA$-algebras. One example that will be of particular importance for us
  later is the operad $\scD_1\otimes N$, which has a copy of the little 1-discs
  operad $\scD_1$ in each colour $n\in N$.
\end{expl}

\begin{defi}
  Consider the monochromatic operad $\scB$ with only two operations, namely the
  identity $\scB(1)=\{\bEn\}$ and a single nullary $\scB(0)=\{\frv\}$. Then
  $(\scB\otimes N)$-algebras are the same as families $\bm{X}=(X_n)_{n\in N}$ of
  based spaces.
  
  A \emph{based $N$-coloured operad} is an $N$-coloured operad $\scO$, together
  with an operad morphism $\scB\otimes N\to \scO$. A morphism of based
  $N$-coloured operads is an operad map $\rho\colon\scO\to \scP$ commuting
  with the two maps from $\scB\otimes N$.
\end{defi}

\begin{rem}
  \begin{enumerate}
  \item A based $N$-coloured operad is the same as an $N$-coloured operad
    $\scO$, together with a choice of nullary operation
    $\frv_n\in\scO\binom{}{n}$ for each colour $n\in N$, and a morphism
    $\rho\colon \scO\to \scP$ of based operads has to additionally satisfy
    $\rho(\frv^\scO_n)=\rho(\frv^\scP_n)$.
  \item The nullaries $\frv_n$ of a based operad can be used to ‘block’ inputs
    by precomposition with them. More precisely, for each input profile
    $K=(k_1,\dotsc,k_r)$ and $1\le i\le r$, we have a map
    \[d_i\colon \scO\binom{K}{n}\to \scO\binom{d_iK}{n},\quad
      \mu\mapsto \mu\circ_i \frv_{k_i}.\]
    In this way the functors  $\scO\binom{-}{n}\colon
    (N\wr\mathbf{\Sigma})^\op\to \mathbf{Top}$ can  be extended to functors
    $\scO\binom{-}{n}\colon (N\wr\Inj)^\op\to \mathbf{Top}$. Using these functors,
      one can give a concise description of the free $\scO$-algebra over a family
      $\bmX\coloneqq (X_n)_{n\in N}$ of based spaces: for each $n\in N$, we have
      \[F^\scO_{\scB\otimes N}(\bmX)_n \cong \int^{K\in
          N\wr\Inj}\scO\binom{K}{n}\times \bmX(K).\]
  \end{enumerate}
\end{rem}

\begin{expl}\label{ex:basedoperads}
  \begin{enumerate}
  \item The little discs operads $\scD_d$ have exactly one nullary operation
    and are thus canonically based. The same applies to $\scD_d\otimes N$ for
    each $N$.
  \item For each small and topologically enriched category $\bfI$, the tensor
    product $\scB\otimes \bfI$ differs from $\bfI$ only by the single nullary
    operation $\frv_n\in (\scB\otimes\bfI)\binom{}{n}$ for each colour $n$.
    Note that $(\scB\otimes\bfI)$-algebras are precisely functors
    $\bfI\to \mathbf{Top}_*$ to the category of based topological spaces.
  \item As a particular case of the previous example, let
    $\bm{G}=(G_n)_{n\in N}$ be a sequence of groups, and consider $\bm{G}$ as a
    (disconnected) groupoid. Then a $(\scB\otimes\bm{G})$-algebra is a sequence
    of based spaces $(X_n)_{n\in N}$ with a basepoint-preserving left action of
    $G_n$ on $X_n$ for all $n\in N$.
  \end{enumerate}
\end{expl}

\section{Infinite loop spaces from coloured operads with homological stability}
\label{sec:infiniteLSFromOHS}
\enlargethispage{-\baselineskip}
In this section we address the following problem: if $\scO$ is an $N$-coloured
operad with homological stability (which will be made precise soon), $\bfI$ is a
topological category together with a map $\scB\otimes \bfI\to \scO$, and if
$\bm{X}=(X_n)_{n\in N}$ an $(\scB\otimes\bfI)$-algebra, what can we say about
the homotopy type of $F^\scO_\scG(\bm{X})$?  By answering this question, we
extend the methods from \cite[§\,5]{BasterraEtAl}, where the monochromatic and
non-relative case was treated, i.e.\ $\bfI=N=*$, so $\scB\otimes\bfI=\scB$.

We briefly summarise the strategy of \cite[§\,5]{BasterraEtAl}: in a first step,
a notion of (monochromatic) ‘operad with homological stability’ is introduced:
such an operad $\scO$ comes in particular with a morphism of operads
$\imath\colon \scD_1\to\scO$, satisfying the weak homotopy commutativity
condition, which demands that $\imath(\new{\scD_1(2)})\subseteq\scO(2)$ lies in
a single path component;\footnote{In principle, any $A_\infty$-operad would
  suffice; we restrict to $\scD_1$ for simplicity.} hence it makes sense to
consider group completions of $\scO$-algebras.

In a second step, the authors of \cite{BasterraEtAl} focus on operads with
homological stability $\scO$ which come with a map
$\pi\colon \scO\to \scD_\infty$ of operads \new{under} $\scD_1$. Thus, we have
for each based space $X$ two interesting maps of $\scO$-algebras:
\begin{enumerate}
\item $F^\scO_\scB(X) \to F^\scO_\scB(*)=\scO(0)$ induced by $X\to *$;\vspace*{-2px}
\item $F^\scO_\scB(X)\to \pi^*F^{\scD_\infty}_\scB(X)$, the unit of the base-change
  adjunction.
\end{enumerate}
Intuitively, the first map forgets the space $X$, while the second map forgets
the operad $\scO$. In \cite[Thm.\,5.4]{BasterraEtAl}, it is shown that the
product map induces a weak equivalence
$\Omega BF^\scO_\scB(X)\to \Omega B\scO(0)\times \Omega^\infty \Sigma^\infty X$
on group completions, after identifying the group completion of
$F^{\scD^\infty}_{\scB}(X)$ with $\Omega^\infty\Sigma^\infty X$.

Finally, an operad with homological stability $\scO$ admits a replacement by
another operad with homological stability
$\scO'\coloneqq \scO\times \scD_\infty$, which has a comparison map
$\pi\colon \scO'\to \scD_\infty$, and under mild extra assumptions, the free
algebras $F^\scO_\scB(X)$ and $F^{\scO'}_\scB(X)$ are equivalent as
$A_\infty$-algebras and thus have equivalent group completions.

\subsection{Coloured operads with homological stability}

\begin{defi}
  An \emph{operad under $\scD_1$} is an $N$-coloured operad $\scO$
  together with an operad morphism $\imath\colon \scD_1\otimes N\to \scO$
  satisfying the weak homotopy commutativity condition levelwise, meaning that
  $\imath((\scD_1\otimes N)\binom{n,n}{n})\subseteq \scO\binom{n,n}{n}$ is
  contained in a single path component.
  \begin{enumerate}
  \item If $\scO$ is an operad under $\scD_1$, then $\scO$ is based by
    $\frv_n \coloneqq \imath(\frv\otimes n)\in \scO\binom{}{n}$. This gives rise
    to the input blocking maps
    \[\beta\colon \scO\binom{k_1,\dotsc,k_r}{n}\to \scO\binom{}{n},\quad
      \mu\mapsto \mu(\frv_{k_1},\dotsc,\frv_{k_r}).\]
  \item If $\scO$ is an operad under $\scD_1$, then each $\scO$-algebra
    $\bm{M}$ is levelwise an H-commutative $\scD_1$-algebra, so for each
    $n\in N$, the set $\pi_0(M_n)$ is an abelian monoid, whose (unique) binary
    operation we denote by ‘$\Ydown$’; there is a bar construction of $M_n$ and
    by the group completion theorem \cite{SegalMcDuff}, we have an isomorphism
    $H_\bullet(\Omega BM_n)\cong H_\bullet(M_n)[\pi_0(M_n)^{-1}]$.
  \end{enumerate}
  A morphism $\rho\colon \scO\to \scO'$ of operads
  $\imath\colon \scD_1\otimes N\to \scO$ and
  $\imath'\colon \scD_1\otimes N\to \scO'$ under $\scD_1$ is a morphism of
  operads such that $\rho\circ\imath=\imath'$ holds.  In that case, for each
  $\scO'$-algebra $\bm{X}$, the levelwise group completions of the
  $\scO'$-algebra $\bm{X}$ and of the $\scO$-algebra $\rho^*\bm{X}$ coincide, as
  they only depend on the $(\scD_1\otimes N)$-structure.
\end{defi}

\begin{nota}
  Let $\imath\colon \scD_1\otimes N\to \scO$ be an operad under
  $\scD_1$. We fix an operation $\frp\in\scD_1(2)$ and we abbreviate
  $\mu\Ydown \mu'\coloneqq\imath(\frp\otimes n)\circ (\mu,\mu')$ for operations
  $\mu\in\smash{\scO\binom{K}{n}}$ and $\mu'\in\smash{\scO\binom{K'}{n}}$. If
  $(M_n)_{n\in N}$ is an $\scO$-algebra and $x,x'\in M_n$, then we also write
  $x\Ydown x'\coloneqq \imath(\frp\otimes n)(x,x')$.
\end{nota}

The notation ‘$\Ydown$’ is pictorially inspired by the following example:

\begin{expl}\label{ex:genm}
  Recall \iref{Definition}{defi:scM} and \iref{Example}{ex:CtensorN}.  We have a
  morphism of operads $\imath_\scM\colon \scD_2\otimes\N_{\ge1}\to \scM$ given
  by applying the classical inclusion of $\scD_2$ into Tillmann’s surface operad
  level-wise; see \iref{Figure}{fig:d2inm}. This morphism restricts to a map
  $\scD_1\otimes\N_{\ge 1}\to \scM$ of operads satisfying the weak homotopy
  commutativity condition and thus turns $\scM$ into an operad under
  $\scD_1$.

  \begin{figure}[h]
    \centering
    \begin{tikzpicture}
      \draw (0,0) circle (1);
      \draw[thick,dred] (-.3,.3) circle (.3);
      \draw[thick,dgreen] (.3,-.3) circle (.4);
      \draw[thick,dyellow] (-.4,-.4) circle (.2);
      \node[dred] at (-.3,.3) {\tiny $3$};
      \node[dgreen] at (.3,-.3) {\tiny $1$};
      \node[dyellow] at (-.4,-.4) {\tiny $2$};
      \node at (1.44,0) {$\otimes\, n$};
      \node at (2.35,-.02) {$\mapsto$};
      \draw[looseness=.85,thick] (3,0) to[out=90,in=90] ++(2,0) to[out=-90,in=-90] (3,0);
      \draw[looseness=.6,thin,dgrey] (3.4,.48) to [out=-90,in=40] (3.35,.18);
      \draw[looseness=.6,thin,dgrey] (4,.48) to [out=-90,in=140] (4.05,.18);
      \draw[looseness=.6,thin,dgrey] (3.9,.15) to [out=-90,in=40] (3.85,-.15);
      \draw[looseness=.6,thin,dgrey] (4.7,.15) to [out=-90,in=140] (4.75,-.15);
      \draw[looseness=.6,thin,dgrey] (3.4,.1) to [out=-90,in=40] (3.35,-.2);
      \draw[looseness=.6,thin,dgrey] (3.8,.1) to [out=-90,in=140] (3.85,-.2);
      \draw[looseness=.85,thick,dred,fill=white,fill opacity=.9] (3.4,.45)
      to[out=90,in=90] ++(.6,0) to[out=-90,in=-90] (3.4,.45);
      \draw[looseness=.85,thick,dgreen,fill=white,fill opacity=.9] (3.9,.15)
      to[out=90,in=90] ++(.8,0) to[out=-90,in=-90] (3.9,.15);
      \draw[looseness=.85,thick,dyellow] (3.4,.1) to[out=90,in=90] ++(.4,0)
      to[out=-90,in=-90] (3.4,.1);
      \node[dred] at (3.7,.45) {\tiny $1$};
      \node[dgreen] at (4.3,.15) {\tiny $1$};
      \node[dyellow] at (3.6,.1) {\tiny $1$};
      \node at (4.9,-.45) {\tiny $1$};
      \node at (8.4,-.45) {\tiny $n$};
      \node at (5.75,0) {$\dotsb$};
      \draw[looseness=.85,thick] (6.5,0) to[out=90,in=90] ++(2,0) to[out=-90,in=-90] (6.5,0);
      \draw[looseness=.6,thin,dgrey] (6.9,.48) to [out=-90,in=40] (6.85,.18);
      \draw[looseness=.6,thin,dgrey] (7.5,.48) to [out=-90,in=140] (7.55,.18);
      \draw[looseness=.6,thin,dgrey] (7.4,.15) to [out=-90,in=40] (7.35,-.15);
      \draw[looseness=.6,thin,dgrey] (8.2,.15) to [out=-90,in=140] (8.25,-.15);
      \draw[looseness=.6,thin,dgrey] (6.9,.1) to [out=-90,in=40] (6.85,-.2);
      \draw[looseness=.6,thin,dgrey] (7.3,.1) to [out=-90,in=140] (7.35,-.2);
      \draw[looseness=.85,thick,dred,fill=white,fill opacity=.9] (6.9,.45)
      to[out=90,in=90] ++(.6,0) to[out=-90,in=-90] (6.9,.45);
      \draw[looseness=.85,thick,dgreen,fill=white,fill opacity=.9] (7.4,.15)
      to[out=90,in=90] ++(.8,0) to[out=-90,in=-90] (7.4,.15);
      \draw[looseness=.85,thick,dyellow] (6.9,.1) to[out=90,in=90] ++(.4,0)
      to[out=-90,in=-90] (6.9,.1);
      \node[dred] at (7.2,.45) {\tiny $n$};
      \node[dgreen] at (7.8,.15) {\tiny $n$};
      \node[dyellow] at (7.1,.1) {\tiny $n$};
    \end{tikzpicture}
    \caption{An instance of
      $\imath_\scM\colon (\scD_2\otimes \N_{\ge
        1})\binom{\textcolor{dgreen}{n},\textcolor{dyellow}{n},\textcolor{dred}{n}}{n}\to
      \scM\binom{\textcolor{dgreen}{n},\textcolor{dyellow}{n},\textcolor{dred}{n}}{n}$}
    \label{fig:d2inm}
  \end{figure}

  The input blocking maps $\beta\colon\scM\binom{K}{n}\to\scM\binom{}{n}$ are
  induced by capping each ingoing boundary curve with a disc, and the abelian
  monoid $\pi_0(\scM\binom{}{n})$ contains all isomorphism types of surfaces
  $\caS$ with $n$ ordered outgoing boundary curves and no incoming boundary
  curve with $\caS$ possibly disconnected, such that each path component of
  $\caS$ has non-empty boundary. The addition on \new{$\pi_0(\scM\binom{}{n})$}
  is given by gluing $n$ pairs of pants; the neutral element is given by an
  ordered collection of $n$ discs. The abelian monoid
  \new{$\pi_0(\scM\binom{}{n})$} is finitely generated: for instance it can be
  generated by the following elements $e_n^i$ and $\smash{e_n^{\smash{i,j}}}$;
  see \iref{Figure}{fig:gentn}:
  \begin{enumerate}
  \item For $1\le i\le n$, we let $e_n^i$ be the isomorphism type of surfaces
    with $n$ path components, such that the component carrying the $i$\Th{th}
    boundary curve has genus 1, whereas all others components are discs.
  \item For each $1\le i<j\le n$, we let $e_n^{\smash{i,j}}$ be the isomorphism
    type of surfaces with $n-1$ path components, all of genus 0, such that one
    component is a cylinder carrying the $i$\Th{th} and the $j$\Th{th} boundary
    curve.
  \end{enumerate}
  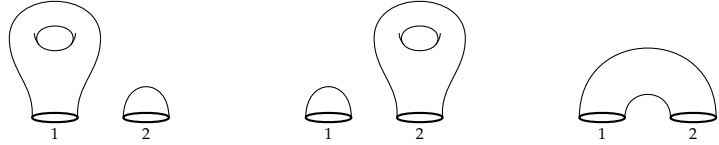
\begin{figure}[h]
    \centering
    \begin{tikzpicture}[xscale=.6,yscale=.7]
      \draw[looseness=.3,thick] (0,0) to[out=-90,in=-90] ++(1,0) to[out=90,in=90] (0,0);
      \draw[looseness=.3,thick] (2,0) to[out=-90,in=-90] ++(1,0) to[out=90,in=90] (2,0);
      \draw[looseness=.3,thick] (6,0) to[out=-90,in=-90] ++(1,0) to[out=90,in=90] (6,0);
      \draw[looseness=.3,thick] (8,0) to[out=-90,in=-90] ++(1,0) to[out=90,in=90] (8,0);
      \draw[looseness=.3,thick] (12,0) to[out=-90,in=-90] ++(1,0) to[out=90,in=90] (12,0);
      \draw[looseness=.3,thick] (14,0) to[out=-90,in=-90] ++(1,0) to[out=90,in=90] (14,0);
      \draw[looseness=2] (2,0) to [out=90,in=90] (3,0);
      \draw[looseness=2] (6,0) to [out=90,in=90] (7,0);
      \draw[looseness=1.5] (13,0) to[out=90,in=90] (14,0);
      \draw[looseness=1.5] (12,0) to[out=90,in=90] (15,0);
      \draw (0,0) to[out=85,in=-90] (-.5,1.5);
      \draw (1,0) to[out=95,in=-90] (1.5,1.5);
      \draw[looseness=1.2] (-.5,1.5) to[out=90,in=90] (1.5,1.5);
      \draw[looseness=1] (.1,1.5) to[out=-90,in=-90] ++(.8,0) to[out=90,in=90] (.1,1.5);
      \draw[thin] (.15,1.38) to[out=130,in=-90] (.05,1.6);
      \draw[thin] (.85,1.38) to[out=50 ,in=-90] (.95,1.6);
      \draw (8,0) to[out=85,in=-90] (7.5,1.5);
      \draw (9,0) to[out=95,in=-90] (9.5,1.5);
      \draw[looseness=1.2] (7.5,1.5) to[out=90,in=90] (9.5,1.5);
      \draw[looseness=1] (8.1,1.5) to[out=-90,in=-90] ++(.8,0) to[out=90,in=90] (8.1,1.5);
      \draw[thin] (8.15,1.38) to[out=130,in=-90] (8.05,1.6);
      \draw[thin] (8.85,1.38) to[out=50 ,in=-90] (8.95,1.6);
      \node [anchor=base] at (0.5,-.4) {\tiny $1$};
      \node [anchor=base] at (2.5,-.4) {\tiny $2$};
      \node [anchor=base] at (6.5,-.4) {\tiny $1$};
      \node [anchor=base] at (8.5,-.4) {\tiny $2$};
      \node [anchor=base] at (12.5,-.4) {\tiny $1$};
      \node [anchor=base] at (14.5,-.4) {\tiny $2$};
    \end{tikzpicture}
    \caption{The three generators $e_2^1$, $e_2^2$ and $e_2^{1,2}$ of
      $\pi_0(\scM\binom{}{2})$.}\label{fig:gentn}
  \end{figure}
\end{expl}

\begin{constr}[Stable operation space]\label{constr:stabopspace}
  We call an $N$-coloured operad $\scO$ under $\scD_1$ \emph{admissibly
    graded} if for each $n\in N$, the abelian monoid $\pi_0(\scO\binom{}{n})$ is
  finitely generated and, for each source profile $(k_1,\dotsc,k_r)$, the
  \emph{degree map}
  \[\abs{\cdot}\colon\scO\binom{k_1,\dotsc,k_r}{n}\stackrel{\beta}{\longrightarrow}
    \scO\binom{}{n}\to \pi_0(\scO\binom{}{n}), \quad \mu\mapsto \abs{\mu}=\pi_0(\beta(\mu))\]
  is surjective.\footnote{In the monochromatic setting, the degree map is
    automatically surjective: by the weak ho\-mo\-to\-py commutativity
    condition, $\pi_0(\scO)$ contains the commutative operad $\fComm$, and if
    we write $\fComm(r)=\set{\frp_r}$, then
    $\beta(\frp_{r+1}\circ_1 \delta)=\frp_1\circ\delta = \bEn\circ \delta =
    \delta$ for each $\delta\in \pi_0(\scO(0))$. For the coloured case,
    though, it seems necessary to additionally assume this property.} In this
  case we write for each $\delta\in \pi_0(\scO\binom{}{n})$
  \[\scO\binom{k_1,\dotsc,k_r}{n}^\delta \coloneqq\set{\mu\in\scO
      \binom{k_1,\dotsc,k_r}{n};\,\abs{\mu}=\delta}\ne\emptyset.\]
  We choose for each $n\in N$ a finite generating set
  $E_n\subseteq \pi_0(\scO\binom{}{n})$ and let $e_n$ be the sum of all elements
  from $E_n$. If we fix a nullary operation $\tilde{e}_n\in \scO\binom{}{n}$
  with $\abs{\tilde{e}_n}=e_n$, called \emph{propagator}, then we obtain for
  each component $\delta\in \pi_0(\scO\binom{}{n})$ and each input profile $ K$
  a \emph{stabilising map}
  \[\on{stab}\colon \scO\binom{K}{n}^\delta
    \to \scO^{\delta\Ydown e_n}\binom{K}{n},\quad
    \mu\mapsto \mu\Ydown \tilde{e}_n.\]
  From this, we can form the \emph{space of stable operations} from $K$ to $n$,
  \[\scO\binom{K}{n}^\infty\coloneqq
    \on{hocolim}\mathopen{}\left(\hspace*{-4px}
      \begin{tikzcd}[column sep=2.2em]
        \scO\binom{K}{n}^{0}\ar[r,"\on{stab}"]
        & \scO\binom{K}{n}^{e_0}\ar[r,"\on{stab}"]
        & \scO\binom{K}{n}^{2e_0}\ar[r,"\on{stab}"]
        & \dotsb
      \end{tikzcd}\!\!\right)\mathclose{}.\]
\end{constr}

\begin{defi}
  Let $\scO$ be an $N$-coloured operad under $\scD_1$ which is admissibly
  graded.  By the associativity of the operadic composition, input blocking and
  stabilisation commute, i.e.\ for each $\delta$, the square
  \[\begin{tikzcd}
      \scO\binom{k_1,\dotsc,k_r}{n}^\delta\ar[r,"\on{stab}"]\ar[d,swap,"\beta"] &
      \scO\binom{k_1,\dotsc,k_r}{n}^{\delta\Ydown e_n}\ar[d,"\beta"]\\
      \scO\binom{}{n}^\delta\ar[r,swap,"\on{stab}"] & \scO\binom{}{n}^{\delta\Ydown e_n}.
    \end{tikzcd}\]
  commutes. Hence we obtain a stable input blocking
  $\beta^K_n\colon \scO\binom{K}{n}^\infty\to \scO\binom{}{n}^\infty$ for each
  input profile $K$, which depends, up to homotopy, only on the path component
  from which the propagator is chosen, i.e.\ on the choice of generating set
  $E_n$.

  We call $\scO$ \new{an} \emph{operad with homological stability} if there is a
  choice of generating sets such that all stable input blockings $\beta^K_n$
  induce isomorphisms in integral homology.
\end{defi}

\begin{expl}
  The coloured surface operad $\scM$ is admissibly graded, and we may use the
  generating sets from \iref{Example}{ex:genm}.

  It is even an operad with homological stability: here we use that multiplying
  with the propagator auto\-matically yields a connected cobordism and increases
  the genus by at least one, see \iref{Figure}{fig:stabm}, so the stable input
  blocking is the cap\-ping map
  $\frM_{\infty,n+k_1+\dotsb+k_r}\to \frM_{\infty,n}$ between stable moduli
  spaces of Riemann surfaces: this is a homology equivalence by Harer’s
  stability theorem \cite{Harer}.\vspace*{-10px}
\end{expl}

\begin{figure}
  \centering
  \begin{tikzpicture}[xscale=.7,yscale=.78]
    \draw[looseness=.35,dgreen,thick] (0,6) to[out=90,in=90] ++(1,0) to[out=-90,in=-90] (0,6);
    \draw[looseness=.35,dgreen,thick] (4,6) to[out=90,in=90] ++(1,0) to[out=-90,in=-90] (4,6);
    \draw[looseness=.35,dred,thick]   (2,6) to[out=90,in=90] ++(1,0) to[out=-90,in=-90] (2,6);
    \fill[white,opacity=.85] (0,6) rectangle (7,5);
    \draw (0,6) to[out=-90,in=90] ++(-.5,-1.5) to[out=-90,in=90] (0,3);
    \draw (1,6) to[out=-90,in=90] ++(.5,-1.5)  to[out=-90,in=90] (1,3);
    \draw[looseness=.3,thick] (0,3) to[out=90,in=90] ++(1,0) to[out=-90,in=-90] (0,3);
    \draw[looseness=.3,thick] (3,3) to[out=90,in=90] ++(1,0) to[out=-90,in=-90] (3,3);
    \draw[looseness=.3,thick,dyellow] (7,3) to[out=90,in=90] ++(1,0) to[out=-90,in=-90] (7,3);
    \draw[looseness=.3,thick,dyellow] (9,3) to[out=90,in=90] ++(1,0) to[out=-90,in=-90] (9,3);
    \draw[looseness=1.5,dyellow] (8,3) to[out=90,in=90] (9,3);
    \draw[dyellow] (7,3) to[out=90,in=-90] (6.5,4.5);
    \draw[dyellow] (10,3) to[out=90,in=-90] (10.5,4.5);
    \draw[dyellow] (6.5,4.5) to[out=90,in=180] ++(.8,.7) to[out=0,in=180]
    ++(1.2,-.3) to[out=0,in=180] ++(1.2,.3) to[out=0,in=90] (10.5,4.5);
    \draw (3,3) to[out=90,in=-90] (2,6);
    \draw (4,3) to[out=90,in=-90] (5,6);
    \draw[looseness=3] (3,6) to[out=-90,in=-90] (4,6);
    \fill[white,opacity=.85] (3,3) rectangle (4,2);
    \fill[white,opacity=.85] (9,3) rectangle (10,2);
    \draw[looseness=.8,dblue] (4,3) to[out=-90,in=-90] (9,3);
    \draw[dblue] (7,0) to[out=90,in=-90] (10,3);
    \draw[dblue] (6,0) to[out=90,in=-90] (3,3);
    \draw[fill=white,white,opacity=.85] (7,3) -- (8,3) to[out=-90,in=90]
    (4.5,0) -- (3.5,0) to[out=90,in=-90] (0,3) -- (1,3) to[out=-38,in=-125] (7,3);
    \draw[dblue] (0,3) to[out=-90,in=90] (3.5,0);
    \draw[looseness=.7,dblue] (1,3) to[out=-90,in=-90] (7,3);
    \draw[dblue] (8,3) to[out=-90,in=90] (4.5,0);
    \draw[looseness=.3,dblue,thick] (6,0) to[out=90,in=90] ++(1,0) to[out=-90,in=-90] (6,0);
    \draw[looseness=.3,dblue,thick] (3.5,0) to[out=90,in=90] ++(1,0) to[out=-90,in=-90] (3.5,0);
    \draw[looseness=1] (.1,4.5) to[out=-90,in=-90] ++(.8,0) to[out=90,in=90] (.1,4.5);
    \draw[thin] (.15,4.38) to[out=130,in=-90] (.05,4.6);
    \draw[thin] (.85,4.38) to[out=50 ,in=-90] (.95,4.6);
    \draw[looseness=1,dyellow] (7.1,4.4) to[out=-90,in=-90] ++(.8,0) to[out=90,in=90] (7.1,4.4);
    \draw[thin,dyellow] (7.15,4.28) to[out=130,in=-90] (7.05,4.5);
    \draw[thin,dyellow] (7.85,4.28) to[out=50 ,in=-90] (7.95,4.5);
    \draw[looseness=1,dyellow] (9.1,4.4) to[out=-90,in=-90] ++(.8,0) to[out=90,in=90] (9.1,4.4);
    \draw[thin,dyellow] (9.15,4.28) to[out=130,in=-90] (9.05,4.5);
    \draw[thin,dyellow] (9.85,4.28) to[out=50 ,in=-90] (9.95,4.5);
    \node[dyellow] at (10.72,5) {\tiny $\tilde{e}_2$};
    \node at (-.45,5.3) {\tiny $\mu$};
    \node[dblue] at (7.55,.6) {\tiny $\Ydown\hspace*{-1px}_2$};
    \node[dblue,anchor=base] at (4,-.45)   {\tiny $1$};
    \node[dblue,anchor=base] at (6.5,-.45) {\tiny $2$};
    \node[dgreen,anchor=base] at (.5,6.25) {\tiny $2$};
    \node[dgreen,anchor=base] at (4.5,6.25) {\tiny $1$};
    \node[dred,anchor=base]   at (2.5,6.25) {\tiny $1$};
  \end{tikzpicture}
  \caption{A single stabilisation step on
    $\scM\binom{\textcolor{dgreen}{2}, \textcolor{dred}{1}}{2}$. Note that
    $\abs{\tilde{e}_2} = e_2^1\Ydown e_2^2\Ydown e_2^{1,2}$ is the isomorphism
    class of surfaces of type $\Sigma_{2,2}$.}
  \label{fig:stabm}
\end{figure}

\subsection{Derived base-change and a splitting result}

Recall that we want to establish an analogue of \cite[Thm.\,5.4]{BasterraEtAl}
for the coloured case \emph{and relative} case, i.e.\ we want to consider
relatively free algebras, relative to a map $\scP\to \scO$ of based operads,
where we will soon restrict to the case $\scP=\scB\otimes\bfI$ for an enriched
category $\bfI$.

The most convenient setting for such a discussion does not use the strict
functor $F^\scO_\scP$, but a homotopically better behaved one, which we denote
by $\LF^\scO_\scP$. This simplifies many point-set issues, and with regard to
our original problem, it will turn out to be equivalent to the space we want to
understand.

The functor $\LF^\scO_\scP$ can be constructed by considering the model
structure on the categories of $\scO$- and $\scP$-algebras as in
\cite{BergerMoerdijk2}, but we decided to give an explicit description. Here we
assume that the reader is familiar with monads and their two-sided bar
constructions, as introduced in \cite{May}.

\begin{constr}
  Let $\scP\to \scO$ be a map of based $N$-coloured operads. We obtain
  monads $\bbO\coloneqq U^\scO_{\scB\otimes N} F^\scO_{\scB\otimes N}$ and
  $\bbP\coloneqq U^\scP_{\scB\otimes N} F^\scP_{\scB\otimes N}$ on $\Top_*^N$,
  and $\bbO$ is a left $\bbP$-functor by the transformation
  $\bbO\bbP\Rightarrow \bbO^2\Rightarrow \bbO$.  For each $\scP$-algebra $\bmX$,
  we consider the two-sided bar construction $B_\bullet(\bbO,\bbP,\bmX)$ with
  $p$-simplices \mbox{$B_p(\bbO,\bbP,\bmX)=\bbO\bbP^p U^\scP_{\scB\otimes N}\bmX$},
  which is an $N$-coloured simplicial space, and we define the \emph{derived free
    algebra} \mbox{$\LF^\scO_\scP(\bmX)\coloneqq |B_\bullet(\bbO,\bbP,\bmX)|$} to be
  its levelwise geometric realisation. Then $\LF^\scO_\scP(\bmX)$ is itself an
  $\scO$-algebra with multiplication
  \[\bbO|B_\bullet(\bbO,\bbP,\bmX)|
    \cong |B_\bullet(\bbO^2,\bbP,\bmX)|
    \to |B_\bullet(\bbO,\bbP,\bmX)|,\]
  where the first identification is due to \cite[Lem.\,9.7]{May} and the last
  map is given by $|B_\bullet(\kappa,\bbP,\bmX)|$ for the operadic composition
  $\kappa\colon \bbO^2\Rightarrow \bbO$. Second, we have a map
  $\LF^\scO_\scP(\bmX)\to F^\scO_\scP(\bmX)$ of $\scO$-algebras, by noticing
  that $F^\scO_\scP(\bmX)$ is the reflexive coequaliser of
  $B_1(\bbO,\bbP,\bmX)\rightrightarrows B_0(\bbO,\bbP,\bmX)$,
  see \iref{Definition}{defi:baseChange}.
\end{constr}

Before stating our main theorem, let us fix once and for all the point-set
requirements we want to assume.
\enlargethispage{-\baselineskip}
\clearpage

\begin{setting}\label{set:set}
  Throughout this section, we consider the following:
  \begin{enumerate}
  \item Let $\scO$ be an $N$-coloured $\frS$-free operad with homological
    stability such that the inclusions
    $\set{\bEn_n}\hookrightarrow \scO\binom{n}{n}$ are cofibrations.
  \item Let $\bfI$ be a topologically enriched category with object set $N$ such
    that the inclusions
    $\set{\bEn_n}\hookrightarrow \bfI\binom{n}{n}$ are cofibrations.
    We assume that there is a map $\scB\otimes\bfI\to \scO$ of based $N$-coloured operads.
  \item Let $\bmX=(X_n)_{n\in N}$ be an $(\scB\otimes\bfI)$-algebra,
    or in other words, an enriched functor $X_\bullet\colon\bfI\to \Top_*$, and we assume that each $X_n$
    is well-based.
  \end{enumerate}
  Moreover, we assume that all involved spaces are Hausdorff.
\end{setting}

Of course, the example we have in mind is $\scO$ being the surface operad $\scM$
\new{and} $\bfI$ being the family $\bm{R}=(R_n)_{n\ge 1}$ of twisted tori.  We
want to show the following:

\begin{theo}[Splitting theorem]\label{theo:splitOHS}
  In the above \iref{Setting}{set:set}, we have, for each $n\in N$, a weak
  equivalence of loop spaces
  \[\Omega B\LF^\scO_{\scB\otimes\bfI}(\bmX)_n\simeq
    \Omega B\scO\binom{}{n}\times \Omega^\infty\Sigma^\infty
    \on{hocolim}_{\bfI}(X_\bullet).\]
\end{theo}

The proof of \iref{Theorem}{theo:splitOHS} will occupy the rest of this section.
Let us start by establishing a map that compares the two sides.

To do so, we start by constructing an $N$-coloured version of the
$E_\infty$-operad $\scD_\infty$ and show that we can, without loss of
generality, assume that there is a comparison map from $\scO$ to it:

\begin{constr}
  For each colour set $N$, we consider the \emph{chaotic category} $EN$ with
  object set $N$ and morphism spaces $(EN)\binom{k}{n}=*$ for all $k,n\in N$.
  We consider the category $\scD_\infty\otimes EN$.
\end{constr}

\begin{lem}
  For the proof of \iref{Theorem}{theo:splitOHS}, we can without loss of
  generality assume a map $\pi\colon \scO\to \scD_\infty\otimes EN$ such that
  the diagram
  \[\begin{tikzcd}[column sep=2em]
      \scB\otimes N\ar[r]\ar[d] & \scB\otimes \bfI\ar[d]\ar[ddr,bend left=15] & \\
      \scD_1\otimes N\ar[r]\ar[rrd,bend right=15] & \scO\ar[dr,"\pi"] & \\
      & & \scD_\infty\otimes EN
    \end{tikzcd}\]
  commutes, where all arrows apart from $\pi$ are either given or induced by the
  canonical maps $\scB\to \scD_1\to\scD_\infty$ and $N\to \bfI\to EN$.
\end{lem}
\begin{proof}
  The commutativity of the square is part of the general setting: recall that we
  assumed that $\scB\otimes \bfI\to \scO$ is a map of \emph{based} operads, and
  $\scO$ is ca\-no\-ni\-cally based as an operad under $\scD_1$.

  To establish the map $\pi$, we replace $\scO$ by a slightly larger operad: if
  we consider the product operad
  $\scO'\coloneqq \scO\times(\scD_\infty\otimes EN)$, together
  with:\vspace*{-2px}
  \begin{itemize}
  \item the diagonal inclusion $\scD_1\otimes N \to \scO'$,\vspace*{-4px}
  \item the diagonal inclusion $\scB\otimes \bfI\to \scO'$,\vspace*{-4px}
  \item the second projection $\scO'\to \scD_\infty\otimes EN$,\vspace*{-2px}
  \end{itemize}
  then the above diagram clearly commutes with $\scO$ instead of $\scO'$.
  Moreover, note that \emph{each} operation space of $\scD_\infty\otimes EN$ is
  contractible: hence $\scO'$ is again admissibly graded with
  $\pi_0(\scO'\binom{}{n})=\pi_0(\scO\binom{}{n})$ and $\scO'$ is again an
  operad with homological stability, satisfying
  $\scO'\binom{}{n}=\scO\binom{}{n}$.

  Finally, the first projection $\scO'\to \scO$ induces a map of monads
  $\bbO'\Rightarrow \bbO$ and hence a map
  $\LF^{\scO'}_{\scB\otimes \bfI}(\bmX)\to
  U^\scO_{\scO'}\LF^{\scO}_{\scB\otimes\bfI}(\bmX)$ of $\scO'$-algebras, which
  is in particular a map of $A_\infty$-algebras. If we denote by $\bbI$ the
  monad for $\scB\otimes \bfI$, then we can easily see that, since $\scO$ is
  $\frS$-free, each $X_n$ is well-based, and every space is Hausdorff, the
  simplicial map $B_\bullet(\bbO',\bbI,\bmX)\to B_\bullet(\bbO,\bbI,\bmX)$ is
  levelwise an equivalence. Second, both simplicial spaces are proper \new{in
    the sense of \cite[§\,11]{May}}, by using that the inclusions of the
  identities are cofibrations. Therefore, by \cite[Thm.\,A.4]{mayGroup}, the
  induced map on the geometric realisations
  $\LF^{\scO'}_{\scB\otimes \bfI}(\bmX)$ and $\LF^\scO_{\scB\otimes \bfI}(\bmX)$
  is again an equivalence, whence their group completions are equivalent as loop
  spaces.
\end{proof}

Using the lemma, we obtain, as in the monochromatic case, two maps:
\begin{enumerate}
\item the map $\bmX\to *$ to $* = (*)_{n\in N}$ induces a map
  $\LF^\scO_{\scB\otimes\bfI}(\bmX)\to \LF^\scO_{\scB\otimes\bfI}(*)$ of
  $\scO$-algebras, which is in particular a map of $A_\infty$-algebras.
\item the morphism $\scO\to \scD_\infty\otimes EN$ induces a map
  $\LF^\scO_{\scB\otimes\bfI}(\bmX)\to \LF^{\scD_\infty\otimes
    EN}_{\scB\otimes\bfI}(\bmX)$ of $A_\infty$-algebras.
\end{enumerate}

The two targets can be identified with the following spaces:

\begin{lem}
  For each $n\in N$, we have equivalences of $A_\infty$-algebras
  \begin{align*}
    \LF^\scO_{\scB\otimes \bfI}(*)_n &\simeq \scO\binom{}{n},\\[2px]
    \LF^{\scD_\infty\otimes EN}_{\scB\otimes \bfI}(\bmX)_n &\simeq F^{\scD_\infty}_\scB(\on{hocolim}_\bfI (X_\bullet)).
  \end{align*}
\end{lem}
\begin{proof}
  For the first equivalence, we note that since
  $(\scB\otimes\bfI)\binom{}{n}=*$, we have
  $F^{\scB\otimes\bfI}_{\scB\otimes N}(*)_n=*$.  Now consider the natural map
  \[\LF^\scO_{\scB\otimes\bfI}(*)
    = \LF^\scO_{\scB\otimes\bfI}(F^{\scB\otimes\bfI}_{\scB\otimes N}(*))
    \to F^\scO_{\scB\otimes\bfI}(F^{\scB\otimes\bfI}_{\scB\otimes N}(*))
    = F^\scO_{\scB\otimes N}(*)
    = \scO\binom{}{n}.\]
  This map arises from the augmentation
  $B_\bullet\coloneqq B_\bullet(\bbO,\bbI,\bbI(*))\to B_{-1}\coloneqq \bbO(*)$,
  which has a (colour-wise) extra degeneracy $s_{-1}\colon B_{p}\to B_{p+1}$
  induced by the unit of $\bbI$; and hence is an equivalence by
  \cite[Cor.\,4.5.2]{Riehl}.

  For the second equivalence, we start with the general observation that for a
  sequence $\scQ\to \scP\to \scO$ of $N$-coloured operads, we have a levelwise
  equivalence of $\scO$-algebras among the derived algebras
  $\LF^\scO_\scQ(\bmX)\simeq \LF^\scO_\scP(\LF^\scP_\scQ(\bmX))$: by
  construction, the left side is the realisation of the bisimplicial
  space with \mbox{$B_{p,q}=\bbO\bbP^{p+1}\bbQ^q\bmX$}. If we first realise each
  $B_{p,\bullet}$, then we obtain a simplicial space $\tB_\bullet$ with
  $\tB_p = |B_\bullet(\bbO\bbP^{p+1},\bbQ,\bmX)|$. Again, we have an
  aug\-men\-ta\-tion map
  $\tB_\bullet\to \tB_{-1}\coloneqq |B_\bullet(\bbO,\bbQ,\bmX)|$, that admits an
  extra degeneracy by the unit of $\bbP$, whence the induced map
  $|B_{\bullet,\bullet}|\simeq |\tB_\bullet|\to \tB_{-1}$ is an equivalence, as
  desired. In our case, we obtain for each $n\in N$ an equivalence of
  $\scD_\infty$-algebras
  \begin{align*}
    \LF^{\scD_\infty\otimes EN}_{\scB\otimes \bfI}(\bmX)_n
    \simeq \LF^{\scD_\infty\otimes EN}_{\scB\otimes EN}(\LF^{\scB\otimes EN}_{\scB\otimes \bfI}(\bmX))_n
    \simeq F^{\scD_\infty}_\scB(\on{hocolim}_\bfI(X_\bullet)),
  \end{align*}
  where for the last equivalence, we use that
  $\LF^{\scB\otimes EN}_{\scB\otimes \bfI}(\bmX)$ is, when regarded as a functor
  $EN\to \Top_*$, the constant diagram with value $\on{hocolim}_\bfI(\bmX)$, and
  $\LF^{\scD_\infty\otimes EN}_{\smash{\scB\otimes EN}}$ is equivalent to the
  postcomposition with $F^{\scD_\infty}_\scB$, using that the natural map
  \mbox{$\smash{\tilde F^{\scD_\infty}_{\smash\scB}(X)\to F^{\scD_\infty}_{\smash\scB}(X)}$}
  is an equivalence
  for each based space $X$, as the underlying simplicial space is constant.
\end{proof}

Putting everything together, we get, for each $n\in N$, a map of
$A_\infty$-algebras
\[\LF^\scO_{\scB\otimes\bfI}(\bmX)\to \scO\binom{}{n}\times
  F^{\scD_\infty}_\scB(\on{hocolim}_\bfI(X_\bullet)),\]
which, after group completion, gives us the map from
\iref{Theorem}{theo:splitOHS}, In the next subsection, we show that it is
an equivalence, and, by doing so, prove the theorem.

\subsection{Proof of the splitting theorem}

For the proof of \iref{Theorem}{theo:splitOHS}, we denote the monads associated
with $\scO$ and $\scD_\infty\otimes EN$ by $\bbO$ and $\bbD$, respectively, and
we write $\bbO(\bmX)_n$ and $\bbD(\bmX)_n$ for their respective
$n$\textsuperscript{th} levels.
Since permuting and blocking inputs preserve the degree of the operations, we
get, for each colour $n\in N$ and each degree $\delta\in\pi_0(\scO\binom{}{n})$,
a functor \mbox{$\smash{\scO\binom{-}{n}{}^\delta\colon (N\wr\Inj)^\op\to \mathbf{Top}}$}. This
gives rise to a decomposition
\[\bbO(\bm{X})_n =\coprod_{\delta}\bbO(\bm{X})^\delta_n\quad\text{with}\quad \bbO(\bm{X})^\delta_n\coloneqq
  \int^{K\in N\wr\Inj}\scO\binom{K}{n}^\delta\times \bm{X}(K).\]
We denote by $\tilde{x}_n=[\tilde{e}_n;()]\in \bbO(\bm{X})^{e_n}_n$ the image of
the propagator and define, in analogy with
\iref{Construction}{constr:stabopspace},
\[\bbO(\bm{X})^\infty_n \coloneqq \on{hocolim}\mathopen{}\left(\hspace*{-5px}
    \begin{tikzcd}[column sep=2.8em] \bbO(\bm{X})^0_n\ar[r,"-\Ydown \tilde{x}_n"] &
      \bbO(\bm{X})^{e_n}_n\ar[r,"-\Ydown\tilde{x}_n"] &
      \bbO(\bm{X})^{2e_n}_n\ar[r,"-\Ydown\tilde{x}_n"] & \dotsb
    \end{tikzcd}\hspace*{-5px}\right)\mathclose{}.\]

\noindent Again, we have two relevant maps:
\begin{enumerate}
\item The map $f\colon \bbO(\bm{X})\to \bbO(*)$ decomposes into maps
  $f^\delta_n\colon \bbO(\bm{X})^\delta_n\to \bbO(*)^\delta_n$ which are
  compatible with stabilisations. Therefore we obtain a map between the mapping
  telescopes $f_n^\infty\colon \bbO(\bm{X})^\infty_n \to \bbO(*)_n^\infty$.
\item Similarly, the map of $\scO$-algebras $\eta\colon \bbO(\bm{X})\to\bbD(\bm{X})$
  restricts to maps of spaces $\eta^\delta_n\colon
  \bbO(\bm{X})^\delta_n\to \bbD(\bm{X})_n$, and the triangle
  \[\begin{tikzcd}[column sep=2.1em,row sep=1.6em]
      \bbO(\bmX)^\delta_n\ar{dr}[swap]{\eta^\delta_n}\ar[rr,"-\Ydown\tilde{x}_n"] &&
      \bbO(\bmX)^{\delta\Ydown e_n}_n\ar[dl,"\eta^{\delta\Ydown e_n}_n"]\\ & \bbD(\bmX)_n
    \end{tikzcd}\]
  is H-commutative, whence we obtain a map from the mapping telescope
  $\eta_n^\infty\colon \bbO(\bm{X})^\infty_n\to \bbD(\bm{X})_n$.
\end{enumerate}

\begin{lem}[Key lemma]\label{lem:lasthequiv}
  For each colour $n\in N$, the product map
  \[(f_n^\infty,\eta_n^\infty)\colon \bbO(\bm{X})^\infty_n\to
    \bbO(*)_n^\infty\times \bbD(\bm{X})_n\]
  is a homology equivalence.
\end{lem}

Let us first prove \iref{Theorem}{theo:splitOHS} using the \iref{Key
  Lemma}{lem:lasthequiv}.

\begin{proof}[Proof of \mbox{Theorem \ref{theo:splitOHS}}]
  Let us abbreviate $\tbO\coloneqq \tF^\scO_{\scB\otimes \bfI}$ and
  $\tbD\coloneqq \tF^{\scD_\infty\otimes EN}_{\scB\otimes \bfI}$. Then we have
  to show that the map $(f,\eta)\colon\tbO(\bmX)\to \tbO(*)\times\tbD(\bmX)$
  from the previous section is levelwise an equivalence.

  To this aim, we study the stabilisations for $\tbO$: as before, restricting
  the operation spaces of $\scO$ gives rise to a grading of each level
  $B_\bullet(\bbO,\bbI,\bmX)_n$, and we denote the components by
  $B_\bullet(\bbO,\bbI,\bmX)^\delta_n$ and their realisation by
  $\tbO(\bmX)_n^\delta$.  Adding the propagator gives rise to maps
  $B_\bullet(\bbO,\bbI,\bmX)^\delta_n\to B_\bullet(\bbO,\bbI,\bmX)^{\delta\Ydown
    e_n}_n$ of simplicial spaces and thus also to maps
  $\tbO(\bmX)_n^\delta\to \tbO(\bmX)_n^{\delta\Ydown e_n}$.  We denote the
  colimit by $B_\bullet(\bbO,\bbI,\bmX)^\infty_n$ resp.\ $\tbO(\bmX)_n^\infty$;
  then clearly $\tbO(\bmX)_n^\infty\simeq |B_\bullet(\bbO,\bbI,\bmX)^\infty_n|$.

  Again, we obtain a product map
  $(f_n^\infty,\eta^\infty_n)\colon\tbO(\bmX)^\infty_n\to
  \tbO(*)^\infty_n\times\tbD(\bmX)_n$, and we claim that it is a homology
  equivalence: since we have already seen that the involved simplicial spaces
  are proper, we can invoke the spectral sequence for the geometric realisation
  from \cite[Prop.\,A1]{SegalCoh}, whence it is enough to see that for each
  dimension $p\ge 0$ and each colour $n\in N$, the map of $p$-simplices
  \mbox{$B_p(\bbO,\bbI,\bmX)^\infty_n\to B_p(\bbO,\bbI,*)^\infty_n\times
    B_p(\bbD,\bbI,\bmX)_n$} is a homology equivalence. As we have
  $B_p(\bbO,\bbI,*)^\infty_n=\bbO(*)^\infty_n$, the map in question is exactly
  the map from the \iref{Key Lemma}{lem:lasthequiv} for the sequence
  $\bbI^p\bmX$. This shows the subclaim.\looseness-1

  The rest of the proof is a combinatorially enhanced variation of the first
  part of the proof of \cite[Thm.\,5.4]{BasterraEtAl} that uses the classical
  group completion theorem: let us denote by
  $\bm{e}_n\in \pi_0(\tbO_n(*)\times \tbD_n(\bmX))$ the component of the
  0-simplex $(\tilde{e}_n,[\new{\frv};\emptyset])$, i.e.\ the propagator and the
  unit, and by \mbox{$\bm{x}_n\in \pi_0(\tbO_n(\bmX))$} the component of the
  0-simplex $\tilde{x}_n$.  By a classical telescope argument, the subclaim
  implies that the map
  \[H_\bullet(f_n,\eta_n)\colon  H_\bullet(\tbO(\bm{X})_n)[\bm{x}_n^{-1}]
    \to H_\bullet(\tbO(*)_n\times \tbD(\bm{X})_n)[\bm{e}_n^{-1}]\]
  that is induced by the map of Pontrjagin rings is an isomorphism.
  
  Now recall that all input blocking maps
  $\beta^K_n\colon \scO\binom{K}{n}\to \pi_0(\scO\binom{}{n})$ are assumed to be
  surjective. Then
  \mbox{$(f_n,\eta_n)_*\colon \pi_0(\tbO(\bmX)_n)\to \pi_0(\tbO(*)_n\times
    \tbD(\bmX)_n)$} is sur\-jective as well, so under the above map, the
  multiplicative submonoid $\pi_0(\tbO(\bmX)_n)$ is sent surjectively onto the
  the submonoid \mbox{$\pi_0(\tbO(*)_n\times \bbD(\bmX)_n)$.} Therefore, we can
  localise further, with respect to the multiplicative submonoids of \emph{all}
  path components on both sides, still obtaining an isomorphism. We get a
  diagram
  \[\begin{tikzcd}[row sep=1.5em,column sep=7em]
      H_\bullet(\tbO(\bmX)_n)[\bm{x}_n^{-1}]\ar{r}{H_\bullet(f_n,\eta_n)}[swap]{\cong}\ar[d]
      & H_\bullet(\tbO(*)_n\times \tbD(\bmX)_n)[\bm{e}_n^{-1}]\ar[d]\\
      H_\bullet(\tbO(\bmX)_n)[\pi_0^{-1}]\ar{r}{H_\bullet(f_n,\eta_n)}[swap]{\cong}\ar[d,equal]
      & H_\bullet(\tbO(*)_n\times \tbD(\bmX)_n)[\pi_0^{-1}]\ar[d,equal]\\
      H_\bullet(\Omega B \tbO(\bmX)_n)\ar{r}{H_\bullet(\Omega B(f_n,\eta_n))}[swap]{\cong}
      & H_\bullet(\Omega B\tbO(*)_n\times\Omega B \tbD(\bmX)_n),
    \end{tikzcd}\]
  where the vertical isomorphisms between the second and the third row follow
  from the group completion theorem \cite{SegalMcDuff} for $\scD_1$-algebras.
  This shows that $\Omega B(f_n,\eta_n)$ is a homology equivalence of loop
  spaces and thus a weak equivalence.
\end{proof}

The pending proof of \iref{Key Lemma}{lem:lasthequiv} requires some further
preparation. Recall \new{from \iref{Definition}{defi:equigroth}} that for each
tuple $K$, we denote by $N[K]\subseteq N\wr\Inj$ the full sub\-group\-oid
spanned by all objects of the form $\tau^*K$.  For an input profile $K$ and an
output $n\in N$, recall the stable operation spaces $\scO^\infty\binom{K}{n}$.
Since input permutation and precomposition commutes with stabilisation, these
spaces assemble, for each $n\in N$ and each tuple $K$, into a functor
$\scO\binom{-}{n}{}^\infty\colon N[K]^\op\to \mathbf{Top}$.

Now let $\bm{Q}\colon N[K]\to \mathbf{Top}$ be any functor.  Again, we have two
maps: first, for each tuple $L=\tau^*K$ and each $n\in N$, we have the stable
input block map
$\scO\binom{L}{n}{}^\infty\times {\bm{Q}}(L)\to \scO\binom{}{n}{}^\infty$ which
ignores the factor ${\bm{Q}}(L)$. These maps define a natural transformation of
functors from $\scO\binom{-}{n}{}^\infty\times \bm{Q}(-)$ to the constant
functor $N[K]^\op\times N[K]\to \mathbf{Top}$ with value
$\scO\binom{}{n}^\infty$, so we get
\[\alpha_1\colon \int^{L}\scO\binom{L}{n}^\infty\times{\bm{Q}}(L)\to \scO\binom{}{n}^\infty.\]

Second, the morphism $\pi\colon \scO\to \scD_\infty\otimes EN$ gives, for each
$n\in N$, rise to a natural transformation
$\scO\binom{-}{n}^\infty\Rightarrow (\scD_\infty\otimes EN)\binom{-}{n}$ of
functors $N[K]^\op\to\Top$, so we obtain a morphism, where $L$ ranges in $N[K]$,
\[\alpha_2\colon \int^{L}\scO\binom{L}{n}^\infty\times{\bm{Q}}(L)
  \to \int^{L}(\scD_\infty\otimes EN)\binom{L}{n}\times{\bm{Q}}(L)
  \cong \scD_\infty\times_{\frS_r}\coprod_{L=\tau^* K}\bm{Q}(L).\]
The following lemma is a coloured version of \cite[Lem.\,5.2]{BasterraEtAl}.

\begin{lem}\label{lem:prodheq}
  The product map $\alpha_{\bm{Q}}\coloneqq (\alpha_1,\alpha_2)$ is a homology equivalence:
  \[\alpha_{\bm{Q}}\colon \int^{L\in N[K]}\scO\binom{L}{n}^\infty\times {\bm{Q}}(L)
    \to \scO\binom{}{n}^\infty\times \scD_\infty(r)\times_{\frS_r}
    \coprod_{L\in N[K]} {\bm{Q}}(L).\]
\end{lem}
\begin{proof}
  Since $\scD_\infty(r)$ is contractible, the sequence
  \[\begin{tikzcd}
      \scO\binom{L}{n}^\infty\ar[r] &
      \scO\binom{L}{n}^\infty\times{\bm{Q}}(L)\ar{r}{\pi\times\on{id}\hspace*{1px}} &
      \scD_\infty(r)\times\bm{Q}(L)
    \end{tikzcd}\]
  induces a split long exact sequence of homotopy groups for each $L=\tau^*K$
  and each choice of basepoint.  If we take for the total space and the base
  space the disjoint union over all such $L$, the common fibre for each
  component is $\scO\binom{L}{n}{}^\infty\cong \scO\binom{K}{n}{}^\infty$. If we
  moreover quotient by the free and compatible $\frS_r$-actions on total space
  and base space, we finally obtain a long exact sequence of homotopy groups
  that is induced by
  \[\scO\binom{K}{n}^\infty
    \to \int^{L}\scO\binom{L}{n}^\infty\times{\bm{Q}}(L)
    \to \scD_\infty(r)\times_{\frS_r}\coprod_{L=\tau^*K}{\bm{Q}}(L).\]
  Now the product map $\alpha_{\bm{Q}}$ is the composition of the two middle
  vertical maps in the following $(3\times 3)$-diagram, abbreviating
  $\frS\coloneqq\frS_r$,
  \[\begin{tikzcd}[column sep=.88em,row sep=2.8em]
      \scO\binom{K}{n}^\infty\ar[d,equal]\ar[r] &
      \int^{L}\scO\binom{L}{n}^\infty\times{\bm{Q}}(L)\ar[r]
      \ar{d}{\int^{L}(\on{id},\pi)\times\on{id}} &
      \scD_\infty(r)\times_{\frS_r}\coprod_{L}{\bm{Q}}(L)\ar[d,equal]\\
      \scO\binom{K}{n}^\infty\ar[r]\ar[d,swap,"\beta^K_n"] &
      \int^{L}\pa{\scO\binom{L}{n}^\infty\times (\scD_\infty\otimes EN)\binom{L}{n}}\times{\bm{Q}}(L)\ar[r]
      \ar{d}{\int^{L}(\beta^{L}_n\times\on{id})\times\on{id}} &
      \scD_\infty(r)\times_{\frS_r}\coprod_{L}{\bm{Q}}(L)\ar[d,equal]\\
      \scO\binom{}{n}^\infty\ar[r] &
      \scO\binom{}{n}^\infty\times\scD_\infty(r)\times_{\frS_r}\coprod_{L}{\bm{Q}}(L)\ar[r] &
      \scD_\infty(r)\times_{\frS_r} \coprod_{L}{\bm{Q}}(L).
    \end{tikzcd}\]
  Here the top-left square commutes up to homotopy and all other squares commute
  strictly.  We have already seen that the top row induces a long exact sequence
  on homotopy groups, and the second row is clearly a fibration. By the 5-lemma,
  the first middle vertical map is a weak equivalence. Similarly, we know that
  both, the second and the third row, are fibrations, so we obtain a morphism
  between the associated Serre spectral sequences in homology. Since $\scO$ is
  an operad with homological stability, the map $\beta^K_n$ between the fibres
  is a homology equivalence, so by a standard comparison argument
  \cite[{}5.2.12]{weibel}, the second middle vertical map is a homology
  equivalence as well.
\end{proof}

Now we have everything together to prove the \iref{Key Lemma}{lem:lasthequiv}.

\begin{proof}[Proof of the \mbox{Key Lemma \ref{lem:lasthequiv}}]
  Recall that, after identifying the stable spaces $\new{\bbO(*)^\infty_n}$ with
  $\scO\binom{}{n}^\infty$, our aim is to show that the map
  \[q\coloneqq (f^\infty_n,\eta_n^\infty)\colon \bbO(\bmX)^\infty_n\
    to \scO\binom{}{n}^\infty\times \bbD(\bmX)_n\]
  induces isomorphisms on homology. If we denote by $(N\wr\Inj)_{\le r}$ the
  full subcategory of $N\wr\Inj$ whose objects are tuples of length at most $r$,
  then the two sides of $q$ are exhaustively filtered by
  $F^{\phantom\prime}_{-1}=F'_{-1}=\emptyset$ and
  \begin{align*}
    F_r &\coloneqq \int^{K\in (N\swr\Inj)_{\le r}}\scO\binom{K}{n}^\infty\times\bmX(K)\\
    \text{resp.}\quad F'_r&\coloneqq \scO\binom{}{n}^\infty\times \int^{K\in (N\swr\Inj)_{\le r}}(\scD_\infty\otimes EN)\binom{K}{n}\times \bm{X}(K),
  \end{align*}
  and the map $q$ is filtration-preserving. Let us fix a system
  $S_r\subseteq N^r$ of representatives for unordered tuples and set
  ${\bm{Q}}(K) \coloneqq X_{k_1}\wedge\dotsb\wedge X_{k_r}$.  Then the
  filtration quotients are of the form
  \begin{align*}
    F_r/F_{r-1} &\cong\bigvee_{K\in S_r}\int^{L\in N[K]}\scO\binom{L}{n}^\infty_+\wedge {\bm{Q}}(L)\\
    F'_r/F'_{r-1} &\cong \bigvee_{K\in S_r}\scO\binom{}{n}^\infty_+\wedge
                    \scD_\infty(r)_+\wedge_{\frS_r}\bigvee_{L=\tau^*K}{\bm{Q}}(L),
  \end{align*}
  and the map $q_r\colon F_r/F_{r-1}\to F'_r/F'_{r-1}$ between the filtration
  quotients splits as a bouquet $q_r=\bigvee_{K\in S_r}q_K$. We show that each
  $q_K$ is a homology equivalence; then it follows that also the map $q_r$
  between the filtration quotients is a homo\-logy equivalence, so by applying a
  comparison argument \new{\cite[{}5.2.12]{weibel}} to the morphism of spectral
  sequences assigned to the filtration-preserving map $q$, we get that $q$
  itself is a homology equivalence.

  In order to see that each $q_K$ is indeed a homology equivalence, we use that
  $\bmX$ is well-based and obtain that the induced maps
  \begin{align*}
    \int^{L\in N[K]}\scO\binom{L}{n}^\infty
    &\to \int^{L\in N[K]}\scO\binom{L}{n}^\infty\times{\bm{Q}}(L)\\[3px]
    \text{and}\quad\scD_\infty(r)\times_{\frS_r}[K]
    &\to \scD_\infty(r)\times_{\frS_r}\coprod_{L=\tau^*K}{\bm{Q}}(L)
  \end{align*}
  are cofibrations, where we write $[K]\coloneqq \{\tau^*K;\,\tau\in \frS_r\}$.
  If we write $\alpha_{\bm{Q}}$ for the product map from
  \iref{Lemma}{lem:prodheq} and $\alpha_0$ for the analogous one for the trivial
  family $*=(*)_{n\in N}$, then we obtain a morphism of cofibre sequences
  (written vertically for space reasons)
  \[\begin{tikzcd}
      \int^{L}\scO\binom{L}{n}^\infty\ar[d]\ar[r,"\alpha_0"]
      & \scO\binom{}{n}^\infty\times\scD_\infty(r)\times_{\frS_r}[K]\ar[d]\\
      \int^{L}\scO\binom{L}{n}^\infty\times{\bm{Q}}(L)\ar[r,"\alpha_{\bm{Q}}"]\ar[d]
      & \scO\binom{}{n}^\infty\times \scD_\infty(r)\times_{\frS_r}\coprod_{L}{\bm{Q}}(L)\ar[d]\\
      \int^{L}\scO\binom{L}{n}^\infty_+\wedge{\bm{Q}}(L)\ar[r,"q_K"]
      & \scO\binom{}{n}^\infty_+\wedge\scD_\infty(r)_+\wedge_{\frS_r}\bigvee\!\!_{L}{\bm{Q}}(L),
    \end{tikzcd}\]
  where $L$ ranges in $N[K]$. By \iref{Lemma}{lem:prodheq}, $\alpha_0$ and
  $\alpha_{\bm{Q}}$ induce isomorphisms in homology, so by the 5-lemma applied
  to the long exact sequence associated to the cofibre sequence, we obtain that
  also $q_K$ is a homology equivalence.
\end{proof}

\section{\texorpdfstring{$\Lambda\frM_{*,1}$}{ΛM⁎₁} as a relatively free algebra}
\label{sec:LM1asFreeAlg}
In this section we combine the insights from the previous sections: we translate
the results of \iref{Section}{sec:arcSystems} and
\iref{Section}{sec:centralisers}, which are expressed in terms of groups, into
the analogous results in terms of classifying spaces.

This will lead to \iref{Theorem}{theo:isfreealg}, expressing $\Lambda\frM_{*,1}$
as the colour-1 part of a relatively free $\scM$-algebra, with generators an
algebra over the family of twisted tori, which depends on $\del$-irreducible
mapping classes. Using the results of \iref{Section}{sec:colouredOHS} and
\iref{Section}{sec:infiniteLSFromOHS}, we deduce the identification from
\iref{Theorem}{eq:a}.

\subsection{Recollections from \iref{Section}{sec:arcSystems} and
  \iref{Section}{sec:colouredOHS}}

Recall \iref{Definition}{defi:fixedarccomplex} and
\iref{Definition}{defi:cutlocus}, let $\caS$ be a surface of type $\Sigma_{g,n}$
for some $g\ge0$ and $n\ge1$, and note that a mapping class
$\phi\in\Gamma(\caS,\del\caS)$ is $\del$-irreducible if and only if the cut
locus of $\phi$ is equal to $[\del\caS]$, i.e.\ it is the collection of oriented
isotopy classes of all boundary curves of $\caS$.

\noindent Note in particular the following special cases:
\begin{itemize}
\item if $\caS$ is a cylinder, $[\del\caS]$ contains two isotopy classes, as the
  two curves of $\del\caS$ are not isotopic as oriented curves; the mapping
  class $\one\in\Gamma(\caS,\del\caS)$ has empty cut locus, and every other
  mapping class is $\del$-irreducible and has cut locus equal to $[\del \caS]$;
\item if $\caS$ is a disc, $[\del\caS]$ contains one isotopy class (of a
  null-homotopic curve); the unique mapping class $\one\in\Gamma(\caS,\del\caS)$
  has empty cut locus and was declared not to be $\del$-irreducible.
\end{itemize}
Alternatively, we saw that $\phi$ is $\del$-irreducible if and only if the white
part $W\subset\caS$ deformation retracts onto $\del^{\text{out}}W=\del\caS$.
Being $\del$-irreducible is an invariant of conjugacy classes in
$\Gamma(\caS,\del\caS)$, and in fact even of conjugacy classes of $\Gamma(\caS)$
that are contained in the normal subgroup $\Gamma(\caS,\del\caS)$.

Recall \iref{Construction}{constr:Rnembedding}.  We have a similar action of
$R_n=T^n\rtimes \frS_n$ on the space $\scM\binom{}{n}=\bfM_\del\binom{0}{n}$ by
postcomposition, regarding $R_n\subset\bfM_\del\binom{n}{n}$.  By taking this
action pointwise over $S^1$, we obtain an action of $R_n$ on
$\Lambda\scM\binom{}{n}$.

\begin{nota}
  For $g\ge0$ and $n\ge1$, we denote by
  $\scM_{g,n}\subset \scM\binom{}{n}=\bfM_\del\binom{0}{n}$ the subspace
  corresponding to conformal classes $(W,\tilde\theta)$ with $W$ of type
  $\Sigma_{g,n}$.
\end{nota}
We have a homotopy equivalence $\scM_{g,n}\to\frM_{g,n}$ given by sending
$(W,\tilde\theta)$ to $(W,\theta)$, where $\theta$ is obtained by restricting
$\tilde\theta$ to (the preimage of) $\del W$.  We have therefore a homotopy
equivalence
\[\Lambda\scM_{g,n}\simeq \Lambda\frM_{g,n}\simeq
  \coprod_{[\phi]\in \on{Conj}(\Gamma_{g,n})}BZ(\phi,\Gamma_{g,n}).\]
In the following we will mostly replace the spaces $\Lambda\frM_{g,n}$ by the
spaces $\Lambda\scM_{g,n}$. In particular we denote
\[\Lambda\scM_{*,1}\coloneqq\coprod_{g\ge0}\Lambda\scM_{g,1}\simeq\Lambda\frM_{*,1},\]
where the last equivalence is an equivalence of $\scD_2$-algebras.

\subsection{Action of \texorpdfstring{$\frS_n$}{Sₙ} on
  \texorpdfstring{$\Lambda\scM_{g,n}$}{ΛM⁎ₙ}}

In this subsection, we analyse the action of $\frS_n$ on the set
of components of $\Lambda\scM_{g,n}$ and classify the orbits of this action.

\begin{nota}
  \label{nota:LMphi}
  For a mapping class $\phi\in\Gamma_{g,n}$ we denote by
  $\Lambda\frM_{g,n}(\phi)$ the connected component of $\Lambda\frM_{g,n}$
  corresponding to $BZ(\phi,\Gamma_{g,n})$.  It contains all free loops
  $\lambda\colon S^1\to\frM_{g,n}$ with the following property: if $\caS$ is a
  Riemann surface of type $\Sigma_{g,n}$ representing $\lambda(1)$ and if
  $\phi'\in\Gamma(\caS,\del\caS)$ is the monodromy of $\lambda$, then there is a
  diffeomorphism $\Xi\colon\caS\to\Sigma_{g,n}$ preserving the parametrisation
  \emph{and the ordering} of the boundary components, such that
  $(\phi')^\Xi=\phi\in\Gamma_{g,n}$.  We denote by $\Lambda\scM_{g,n}(\phi)$ the
  corresponding component of $\Lambda\scM_{g,n}$.
\end{nota}

Note that $\Lambda\scM_{g,n}(\phi)\subset\Lambda\scM_{g,n}$ is invariant under
the action of $T^n\subset R_n$, but not necessarily under the action of
$\frS_n$. We denote by $\fS_n\cdot\Lambda\scM_{g,n}(\phi)\subset\Lambda\scM_{g,n}$
the orbit of $\scM_{g,n}(\phi)$ under the action of $\fS_n$ or, equivalently,
under the action of $R_n$. We have
\[\fS_n\cdot\Lambda\scM_{g,n}(\phi)=\bigcup_{\phi'}\Lambda\scM_{g,n}(\phi'),\]
where $\phi'$ ranges over all conjugates of $\phi$ in the extended mapping
class group $\Gamma_{g,(n)}$. Note that these conjugates still lie in the
subgroup $\Gamma_{g,n}\subset\Gamma_{g,(n)}$, which is normal. Note also that
$\Lambda\scM_{g,n}(\phi')=\Lambda\scM_{g,n}(\phi)$ if and only if $\phi$ and
$\phi'$ are conjugate not only in $\Gamma_{g,(n)}$, but also in $\Gamma_{g,n}$;
see \tref{Figure}{fig:conj} for an example that shows the difference.

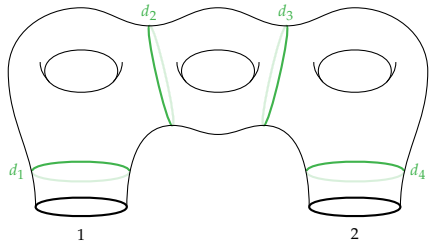
\begin{figure}[b]
  \centering
  \begin{tikzpicture}[scale=1.2]
    \draw[dgreen,thick,looseness=.35] (-.04,.39) to[out=90,in=90] (1.03,.39)
    to[out=-90,in=-90] (-.04,.39);
    \draw[dgreen,thick,looseness=.35] (2.97,.39) to[out=90,in=90] (4.04,.39)
    to[out=-90,in=-90] (2.97,.39);
    \draw[dgreen,thick,looseness=.15] (1.5,.9) to[out=190,in=190] (1.25,2)
    to[out=10,in=10] (1.5,.9);
    \draw[dgreen,thick,looseness=.15] (2.5,.9) to[out=170,in=170] (2.75,2)
    to[out=-10,in=-10] (2.5,.9);
    \draw[white,very thick,opacity=.8,looseness=.15] (1.25,2) to[out=10,in=10] (1.5,.9);
    \draw[white,very thick,opacity=.8,looseness=.15] (2.5,.9) to[out=170,in=170] (2.75,2);
    \fill[white,opacity=.8] (-.2,.39) rectangle (4.3,.2);
    \draw[looseness=.35,thick] (0,0) to[out=-90,in=-90] ++(1,0) to[out=90,in=90] (0,0);
    \draw[looseness=.35,thick] (3,0) to[out=-90,in=-90] ++(1,0) to[out=90,in=90] (3,0);
    \node at (.5,-.3) {\tiny $1$};
    \node at (3.5,-.3) {\tiny $2$};
    \draw (0,0) to[out=90,in=-90] (-.3,1.5) to[out=90,in=180] (.5,2.2) to[out=0,in=180]
    (1.25,2) to[out=0,in=180] (2,2.2) to[out=0,in=180] (2.75,2) to[out=0,in=180] (3.5,2.2)
    to[out=0,in=90] (4.3,1.5) to[out=-90,in=90] (4,0);
    \draw (1,0) to[out=90,in=180] (1.5,.9) to[out=0,in=180] (2,.8) to[out=0,in=180] (2.5,.9)
    to[out=0,in=90] (3,0);
    \draw[looseness=1] (.1,1.5) to[out=-90,in=-90] ++(.8,0) to[out=90,in=90] ++(-.8,0);
    \draw[thin] (.15,1.38) to[out=130,in=-90] (.05,1.6);
    \draw[thin] (.85,1.38) to[out=50 ,in=-90] (.95,1.6);
    \draw[looseness=1] (1.6,1.5) to[out=-90,in=-90] ++(.8,0) to[out=90,in=90] ++(-.8,0);
    \draw[thin] (1.65,1.38) to[out=130,in=-90] (1.55,1.6);
    \draw[thin] (2.35,1.38) to[out=50 ,in=-90] (2.45,1.6);
    \draw[looseness=1] (3.1,1.5) to[out=-90,in=-90] ++(.8,0) to[out=90,in=90] ++(-.8,0);
    \draw[thin] (3.15,1.38) to[out=130,in=-90] (3.05,1.6);
    \draw[thin] (3.85,1.38) to[out=50 ,in=-90] (3.95,1.6);
    \node[dgreen] at (-.2,.39) {\tiny $d_1$};
    \node[dgreen] at (4.2,.39) {\tiny $d_4$};
    \node[dgreen] at (1.25,2.15) {\tiny $d_2$};
    \node[dgreen] at (2.75,2.15) {\tiny $d_3$};
  \end{tikzpicture}
  \caption{If we denote by $D_i\coloneqq D_{d_i}$ the Dehn twist along $d_i$,
    then the mapping classes $D_1D_2D_4$ and $D_1D_3D_4$ are both
    $\partial$-irreducible in $\Gamma_{3,2}$ and not conjugate to each other,
    but they are conjugate in the extended mapping class group
    $\Gamma_{3,(2)}$.}\label{fig:conj}
\end{figure}

The subspace of $\frM_{g,n}$ corresponding to
$\fS_n\cdot \Lambda\frM_{g,n}(\phi)\subset\Lambda\frM_{g,n}$ can be described as
follows: it contains all free loops $\lambda\colon S^1\to\frM_{g,n}$ with the
same property as in \iref{Notation}{nota:LMphi}, but where $\Xi$ is only
required to preserve the parametrisation, and not necessarily the order, of the
boundary components.

\begin{lem}
  \label{lem:cosets}
  Let $\phi\in\Gamma_{g,n}$ and let $\frH\subset\frS_n$ be the image of
  $Z(\phi,\Gamma_{g,(n)})$ under the natural map $\Gamma_{g,(n)}\to\frS_n$.
  Then there is a bijection of $\frS_n$-sets
  \[\pi_0\pa{\frS_n\cdot \Lambda\scM_{g,n}(\phi)}\cong \frS_n/\frH.\]
\end{lem}
\begin{proof}
  The action of $\frS_n$ on $\pi_0(\Lambda\scM_{g,n})$ can be described as
  follows: given $\sigma\in\frS_n$, we choose a mapping class
  $\xi\in\Gamma_{g,(n)}$ which is sent to $\sigma$ under the natural map
  $\Gamma_{g,(n)}\to \frS_n$, and a representative $\Xi$ of $\xi$; then the
  component $\pi_0(\Lambda\scM_{g,n}(\phi'))$ is sent by $\sigma$ to the
  component $\pi_0(\Lambda\scM_{g,n}((\phi')^\Xi))$.

  The action of $\frS_n$ on $\pi_0(\frS_n\cdot \Lambda\scM_{g,n}(\phi))$ is
  transitive, so it suffices to check that the stabiliser of the component
  $\Lambda\scM_{g,n}(\phi)$ is the subgroup $\frH$. First, note that an element
  $\sigma\in\frH$ can be lifted to a class $\xi\in\Gamma_{g,(n)}$ that commutes
  with $\phi$; this implies that $\frH$ is contained in the stabiliser of
  $\Lambda\scM_{g,n}(\phi)$. Vice versa, if $\sigma\in\frS_n$ belongs to the
  stabiliser of $\Lambda\scM_{g,n}(\phi)$, then we can choose a lift
  $\xi\in\Gamma_{g,(n)}$ of $\sigma$ and a representative $\Xi$ such that
  $\Lambda \scM_{g,n}(\phi^\Xi)$ is equal to $\Lambda\scM_{g,n}(\phi))$; this
  implies that $\phi^\Xi$ is conjugate to $\phi$ in $\Gamma_{g,n}$, i.e.\ there
  is a mapping class $\bar\xi\in\Gamma_{g,n}$ and a representative $\bar\Xi$
  with $\phi^{\Xi}=\phi^{\bar\Xi}$. As a consequence
  $\xi^{-1}\bar\xi\in\Gamma_{g,(n)}$ commutes with $\phi$, and since
  $\xi^{-1}\bar\xi$ also projects to $\sigma$ along the natural map
  $\Gamma_{g,(n)}\to\frS_n$, we conclude that $\sigma\in\frH$.
\end{proof}

\subsection{Relative generators for the \texorpdfstring{$\scM|_1$}{M|₁}-algebra
  \texorpdfstring{$\Lambda\scM_{*,1}$}{ΛM⁎₁}}

\begin{defi}
  For all $n\ge1$ and $g\ge0$ we define the space
  \[\frC_{g,n} \coloneqq
    \coprod_{\substack{[\varphi]\in\on{Conj}(\Gamma_{g,n})\\\del\text{-irreducible}}}
    \Lambda\scM_{g,n}(\phi).\]
  The action of $R_n$ on $\Lambda\scM_{g,n}$ restricts to an action on those
  path components that constitute $\frC_{g,n}$. We furthermore set, for all
  $n\ge1$,
  \[\frC_n\coloneqq \coprod_{g\ge 0} \frC_{g,n}\]
  and obtain a $\bm{R}$-algebra $\bm{\frC}\coloneqq(\frC_n)_{n\ge 1}$, where
  $\bm{R}=(R_n)_{n\ge 1}$.
\end{defi}
To see that the action of $R_n$ on $\Lambda\frM_{g,n}$ indeed restricts to an
action on $\frC_{g,n}$, we note that, for every $\partial$-irreducible mapping
class $\phi\in\Gamma_{g,n}$, the orbit
$\fS_n\cdot\Lambda\scM_{g,n}(\phi)$ is con-\linebreak tained in $\frC_{g,n}$.  Note
that, though $\fS_n\cdot\Lambda\scM_{g,n}(\phi)$ may be disconnected, by
\iref{Lemma}{lem:cosets} the action of $R_n$ is transitive on the connected
components of $\fS_n\cdot\Lambda\scM_{g,n}(\phi)$.\looseness-1

We now want to look at the relatively free $\scM$-algebra
$F^\scM_{\bm{R}}(\frC)$. Here we see, levelwise, that for each $n\ge 1$, we have
\[F^\scM_{\bm{R}}(\frC)_n \cong \int^{(k_1,\dotsc,k_r)\in \N_{\ge 1}\kern.3px\wr\kern.3px\bSig}
  \scM\binom{k_1,\dotsc,k_r}{n}\times_{R_{k_1}\times\dotsb\times
    R_{k_r}}
  \mathopen{}\big(\frC_{k_1}\times\dotsb\times\frC_{k_r}\big)\mathclose{}.\]

\begin{theo}\label{theo:isfreealg}
  There is an equivalence of $\scM|_1$-algebras
  \[ F^{\scM}_{\scB\otimes\bm{R}}(\bm{\frC}_+)_1
    = F^\scM_{\bm{R}}(\frC)_1
    \simeq \Lambda\scM_{*,1}\]
\end{theo}
The proof of \iref{Theorem}{theo:isfreealg} occupies the remainder of this
section.

\bigskip

Recall the initial $\scM$-algebra $(\scM\binom{}{n})_{n\ge 1}$. Note that
$\Lambda\scM\binom{}{n}$ has several connected components, and the ones
corresponding to connected surfaces form precisely the subspace
$\Lambda\scM_{*,n}\coloneqq\coprod_{g\ge0}\Lambda\scM_{g,n}$. We therefore have
an inclusion of $\bm{R}$-algebras
$\bm{\frC}\subset(\Lambda\scM\binom{}{n})_{n\ge1}$.  This inclusion is adjoint
to a morphism of $\scM$-algebras
\mbox{$F^{\scM}_{\bm{R}}(\bm{\frC})\to(\Lambda\scM\binom{}{n})_{n\ge1}$}, which can be
restricted to a morphism of $\scM|_1$-algebras
$\kappa\colon F^{\scM}_{\bm{R}}(\bm{\frC})_1 \to\Lambda\scM\binom{}{1} =
\Lambda\scM_{*,1}$. It suffices to prove that $\kappa$ is a weak equivalence at
the level of spaces. The right-hand side can be decomposed into its connected
components as
\[\Lambda\scM_{*,1}=\coprod_{\substack{g\ge 0\\ [\phi]\in\on{Conj}(\Gamma_{g,1})}}
  \Lambda\scM_{g,1}(\phi).\]
Fix $g\ge0$ and a mapping class $\phi\in\Gamma_{g,1}$. Decompose $\Sigma_{g,1}$
along a system of curves $c_1,\dots,c_h$ representing the cut locus of $\phi$,
and let $W$ and $Y$ denote the white and the yellow regions.
We use \iref{Notation}{nota:Yi} and write
$\smash{Y=\coprod_{i=1}^r\coprod_{j=1}^{s_i}Y_{i,j}}$.  Each component $Y_{i,j}$ of $Y$
is of type $\Sigma_{g_i,n_i}$, and the restriction
$\smash{\phi_{i,j}\in\Gamma(Y_{i,j},\del Y_{i,j})}$ of $\phi$ is conjugated by a
suitable diffeomorphism $\Xi_{i,j}\colon Y_{i,j}\to\Sigma_{g_i,n_i}$ to
$\bar\phi_i\in\Gamma_{g_i,n_i}$.

Note that $W$ is a connected surface: each component of $W$ touches some
component of $\del^{\text{out}}W$, and there is precisely one outgoing boundary
component, since $\del^{\text{out}}W=\del\Sigma_{g,1}\cong S^1$.

We denote by $\scM(W)\subset\scM\tbinom{s_1\times n_1,\dotsc, s_r\times n_r}{1}$
the component of the surface type of $W$, with one outgoing boundary component
and $h=\sum_{i}s_i\cdot n_i$ incoming boundary curves partitioned as follows:
for $1\le i\le r$ and $1\le j\le s_i$ the curves
$\del Y_{i,j}\subset\del^{\text{in}}W$ form a piece of the partition,
corresponding to an input of colour $n_i$. The total number of inputs is thus
$s_1+\dotsb +s_r$.

Consider now the following subspace, which a priori is a union of connected
components of $F^{\scM}_{\bm{R}}(\bm{\frC})_1$,
\begin{align*}
  F^{\scM}_{\bm{R}}(\bm{\frC})_\phi
  &\coloneqq  \scM(W) \times_{\prod_i R_{n_i}^{s_i}\rtimes \fS_{s_i}}
    \prod_{i=1}^r\pa{\fS_{n_i}\cdot\Lambda\scM_{g_i,n_i}(\bar\phi_i)}^{s_i}\\
  &\>= \scM(W) \times_{T^{h} \rtimes \prod_i (\frS_{n_i}^{s_i}\rtimes\frS_{s_i})}
    \prod_{i=1}^r\pa{\fS_{n_i}\cdot\Lambda\scM_{g_i,n_i}(\bar\phi_i)}^{s_i}
    \subset F^{\scM}_{\bm{R}}(\bm{\frC})_1.
\end{align*}

\begin{lem}
  \label{lem:FscMphiconnected}
  The space $F^{\scM}_{\bm{R}}(\bm{\frC})_\phi$ is connected.
\end{lem}
\begin{proof}
  We observe that $\scM(W)$ is connected; therefore, in order to prove that
  $F^{\scM}_{\bm{R}}(\bm{\frC})_\phi$ is connected, it suffices to check that
  $T^{h} \rtimes \prod_{i=1}^r (\frS_{n_i}^{s_i}\rtimes\frS_{s_i})$ acts
  transitively on the connected components of
  $\prod_{i=1}^r\pa{\fS_{n_i}\cdot\Lambda\scM_{g_i,n_i}(\bar\phi_i)}^{s_i}$.
  
  This reduces to checking that for all $1\le i\le r$ the group
  $\frS_{n_i}^{s_i}\rtimes\frS_{s_i}$ acts transitively on the components of
  $\pa{\fS_{n_i}\cdot\Lambda\scM_{g_i,n_i}(\bar\phi_i)}^{s_i}$, and to this aim it
  suffices to check that $\frS_{n_i}$ acts transitively on the components of
  $\fS_{n_i}\cdot\Lambda\scM_{g_i,n_i}(\bar\phi_i)$, which is clear by definition.
\end{proof}
\enlargethispage{-\baselineskip}
Using \iref{Notation}{nota:Hi}, the proof of \iref{Lemma}{lem:FscMphiconnected}
shows that there is a homeomorphism
\[F^{\scM}_{\bm{R}}(\bm{\frC})_\phi\cong \scM(W) \times_{T^h \rtimes \prod_i
    (\frH_i^{s_i}\rtimes\frS_{s_i})}
  \prod_{i=1}^r\pa{\Lambda\scM_{g_i,n_i}(\bar\phi_i)}^{s_i},\]
using that $\frH_i\subseteq\frS_{n_i}$ is the stabiliser of the path component
$\Lambda\scM_{g,n}(\bar\varphi_i)\subseteq \frC_{g_i,n_i}$.  The following
proposition directly implies \iref{Theorem}{theo:isfreealg}, since
$F^{\scM}_{\bm{R}}(\bm{\frC})_\phi$ and $\Lambda\scM_{g,1}(\phi)$ are the
connected components of $F^{\scM}_{\bm{R}}(\bm{\frC})_1$ and
$\Lambda\scM_{*,1}$.

\begin{prop}
  The map $\kappa$ restricts to a homotopy equivalence
  \[\kappa\colon F^{\scM}_{\bm{R}}(\bm{\frC})_\phi\to \Lambda\scM_{g,1}(\phi).\]
\end{prop}
\begin{proof}
  The space $\scM(W)$ classifies $\Gamma(W,\del W)$, whereas the
  product $\prod_{i}\pa{\Lambda\scM_{g_i,n_i}(\bar\phi_i)}^{s_i}$ is a
  classifying space for the group
  $\prod_{i} (Z(\bar\phi_i,\Gamma_{g_i,n_i}))^{s_i}$. This implies that there is
  a homotopy equivalence
  \[\scM(W) \times \prod_{i=1}^r\pa{\Lambda\scM_{g_i,n_i}(\bar\phi_i)}^{s_i}
    \simeq B\pa{\Gamma(W,\del W)\times \prod_{i=1}^r (Z(\bar\phi_i,\Gamma_{g_i,n_i}))^{s_i}}.\]
  For $\frH\coloneqq \prod_i(\frH^{s_i}_i\rtimes\frS_{s_i})\subseteq\frS_h$, the
  compact Lie group $T^{h} \rtimes \frH \subseteq R_n$ acts freely on $\scM(W)$
  by \iref{Lemma}{lem:freeactionRn}. Since
  $\scM(W)\times \prod_i (\Lambda\scM_{g_i,n_i}(\bar\varphi_i))^{s_i}$ is a
  Tychonoff space, \cite[Thm.\,A8viii]{Orbispaces} ensures that we obtain a
  principal fibre bundle
  \[T^{h} \!\rtimes \frH
    \to \scM(W) \times \prod_{i=1}^r\pa{\Lambda\scM_{g_i,n_i}(\bar\phi_i)}^{s_i}
    \to \scM(W) \times\!_{T^{h} \rtimes \frH} \prod_{i=1}^r\pa{\Lambda\scM_{g_i,n_i}(\bar\phi_i)}^{s_i}\]
  We will split the fibre bundle in two stages, involving separately the factors
  $\frH$ and $T^{h}$ of the fibre. First, consider the principal fibre bundle
  \[\frH \to \scM(W) \times \prod_{i}\pa{\Lambda\scM_{g_i,n_i}(\bar\phi_i)}^{s_i}\
    \to \scM(W) \times\!_\frH \prod_{i}\pa{\Lambda\scM_{g_i,n_i}(\bar\phi_i)}^{s_i},\]
  from which we obtain that
  $\smash{\scM(W) \times_{\frH}
    \prod_{i}\pa{\Lambda\scM_{g_i,n_i}(\bar\phi_i)}^{s_i}}$ is connected and
  aspherical.

  In fact
  $\smash{\scM(W) \times_{\frH}
    \prod_{i}\pa{\Lambda\scM_{g_i,n_i}(\bar\phi_i)}^{s_i}}$ is a classifying
  space for the group $\smash{\tZ(\phi)}$ introduced in
  \iref{Definition}{defi:tZdelphi}: the space $\scM(W)\simeq B\Gamma(W,\del W)$
  admits a free action of the finite group $\frH$, and the quotient
  $\scM(W)/\frH$ is a classifying space for the extended mapping class group
  $\Gamma^{\frH}(W)$; see \iref{Definition}{defi:extendedmcg}.  Similarly, the
  homotopy quotient
  \[\pa{\prod_{i=1}^r\pa{\Lambda\scM_{g_i,n_i}(\bar\phi_i)}^{s_i}}\qq \frH
    =E\frH\times_\frH \pa{\prod_{i=1}^r\pa{\Lambda\scM_{g_i,n_i}(\bar\phi_i)}^{s_i}}\]
  is a classifying space for $Z(\phi_Y,\Gamma(Y))$, which by
  \iref{Lemma}{lem:ZphiYdecomposition} is isomorphic to
  $\prod_{i}(Z(\bar\phi_i,\Gamma_{g_i,(n_i)})^{s_i}\rtimes \fS_{s_i})$.  The
  balanced product
  $\scM(W) \times_{\frH}\prod_{i}\pa{\Lambda\scM_{g_i,n_i}(\bar\phi_i)}^{s_i}$
  is then homotopy equivalent to\looseness-1
  \[\scM(W)\times_{\frH}
    \pa{\prod_{i=1}^r\pa{\Lambda\scM_{g_i,n_i}(\bar\phi_i)}^{s_i}\times E\frH},\]
  which is a classifying space for $\tZ(\phi)$. The last step is to consider the
  fibre bundle (which is no longer a principal bundle)
  \[T^{h} \to \scM(W) \times\!_{\frH}\prod_{i=1}^r\pa{\Lambda\scM_{g_i,n_i}(\bar\phi_i)}^{s_i}
    \to \scM(W) \times_{T^{h} \rtimes \frH}\prod_{i=1}^r
    \pa{\Lambda\scM_{g_i,n_i}(\bar\phi_i)}^{s_i}.\]
  Since the fibre and the total space are aspherical, the base space is also
  aspherical: note that the base is precisely
  $F^{\scM}_{\bm{R}}(\bm{\frC})_\phi$.

  The above fibre bundle shows that the fundamental group of
  $F^{\scM}_{\bm{R}}(\bm{\frC})_\phi$ is a quotient of $\tZ(\phi)$ by a normal
  subgroup $\Z^h$. Consider the gluing map
  \[\tilde\kappa\colon\scM(W) \times_{\frH}\prod_{i=1}^r
    \pa{\Lambda\scM_{g_i,n_i}(\bar\phi_i)}^{s_i}
    \to \Lambda\scM_{g,1}(\phi).\]
  The map induced by $\tilde\kappa$ on fundamental groups is
  $\epsilon\colon \tZ_\del(\phi)\to Z_\del(\phi)$, so by
  \iref{Proposition}{prop:Zdelphistructure} we just have to identify the kernel
  of $\epsilon$ and the subgroup
  \[\pi_1(T^h)\subset\pi_1\pa{ \scM(W) \times\!_\frH
      \prod_{i=1}^r\pa{\Lambda\scM_{g_i,n_i}(\bar\phi_i)}^{s_i}}.\]
  The fibre inclusion
  $T^{h} \hookrightarrow \scM(W)
  \times_{\frH}\prod_i\pa{\Lambda\scM_{g_i,n_i}(\bar\phi_i)}^{s_i}$ lifts to the
  covering space
  $\scM(W) \times\prod_{i}\pa{\Lambda \scM_{g_i,n_i}(\bar\phi_i)}^{s_i}$ of the
  right-hand side and becomes part of the fibre inclusion
  $T^{h}\rtimes\frH\hookrightarrow\scM(W)\times
  \prod_{i}\pa{\Lambda\scM_{g_i,n_i}(\bar\phi_i)}^{s_i}$ of the first principal
  bundle. In particular, the inclusion
  \mbox{$\pi_1(T^h)\hookrightarrow \tZ(\phi) =\pi_1\pa{\scM(W) \times_{\frH}\prod_{i}
    \pa{\Lambda\scM_{g_i,n_i}(\bar\phi_i)}^{s_i}}$} has its image inside
  \[\Gamma(W,\del W)\times Z(\phi_Y,\Gamma(Y,\del Y))
    =\pi_1\pa{\scM(W) \times \prod_{i=1}^r
      \pa{\Lambda\scM_{g_i,n_i}(\bar\phi_i)}^{s_i}}.\]
  Recall $T^h$ is included into
  $\scM(W) \times \prod_{i}\pa{\Lambda\scM_{g_i,n_i}(\bar\phi_i)}^{s_i}$ as an
  orbit of the \emph{balanced} diagonal $T^h$-action on the two factors of
  the latter space. The action of $(z_1,\dots,z_h)\in T^h$ is on
  the first factor, by precomposition with
  \mbox{$(z_1,\dots,z_h)\in R_n\subset\bfM_\del\binom{h}{h}$}, and  on
  the second factor, by postcomposition with \emph{the inverse of}
  \mbox{$(z_1,\dots,z_h)\in R_n\subset\bfM_\del\binom{h}{h}$}.

  At the level of fundamental groups, the $i$\textsuperscript{th} generator of
  $\pi_1(T^h)\cong\Z^h$ is mapped to
  $(D_{c_i},D_{c_i}^{-1})\in\Gamma(W,\del W)\times Z(\phi_Y,\Gamma(Y,\del Y))$.
  We saw that these $h$ couples generate the kernel of $\epsilon$ as a free
  abelian group of rank $h$.
\end{proof}

Putting together \iref{Theorem}{theo:isfreealg} and
\iref{Theorem}{theo:splitOHS}, we obtain \iref{Theorem}{theo:main}:

\begin{proof}[Proof of \iref{Theorem}{theo:main}]
  The point-set requirements for our \iref{Setting}{set:set} are clearly
  satisfied in the case $\scO=\scM$, $\bfC=\bm{R}$, and $\bmX=\frC_+$.  Hence,
  we can apply \iref{Theorem}{theo:splitOHS} and obtain
  \[\Omega B \tF^\scM_{\scB\otimes \bm{R}}(\frC_+)_n
    \simeq \Omega B\scM\binom{}{n}\times \Omega^\infty\Sigma^\infty
    \on{hocolim}_{\bm{R}}(\frC_+).\]
  Consider the left side first: we claim that
  \mbox{$\phi\colon \tF^\scM_{\scB\otimes \bm{R}}(\frC_+) \to
    F^\scM_{\scB\otimes\bm{R}}(\frC_+)$,} the map of $\scM$-algebras from the
  derived relatively free algebra to the actual one is an equivalence: here we
  use that the basepoints are isolated, whence the map splits into
  $\coprod_{K\in S}\phi_K$ for a system $S\subseteq\coprod_r \N_{\ge 1}^r$ of
  representatives of tuples with respect to coordinate permutation.  If we
  denote again by $r(k)\ge 0$ the number of occurrences of $k$ in the sequence
  $K$, then the compact Lie group $G\coloneqq \prod_{k\ge 0}R_{k}\wr \frS_{r_k}$
  acts on $Y\coloneqq \scM\binom{K}{n}\times \prod_i\frC_{k_i}$, and $\phi_K$ is
  exactly the map that compares the \emph{homotopy} quotient of this action with
  the \emph{actual} quotient. However, $Y$ is a Hausdorff space and $G$ acts
  freely on $Y$, since $\scM$ is $\frS$-free and $R_k$ acts freely on $\scM$ by
  precomposition, see \iref{Lemma}{lem:freeactionRn}.  In this situation,
  \cite[Thm.\,A.7]{Orbispaces} tells us that the map comparing the homotopy
  quotient with the actual quotient is an equivalence.  In particular, for
  $n=1$, the left side is equivalent, as a loop space, to the group completion
  $\Omega B\Lambda\frM_{*,1}$ by \iref{Theorem}{theo:isfreealg}.
  Let us now look at the right side: here we see that
  \[\on{hocolim}_{\bm{R}}(\frC_+) \simeq \bigvee_{k\ge 1}(\frC_n)_+\sslash R_k
    = \set{*}\sqcup\coprod_{k\ge 1}\frC_k\sslash R_k,\]
  where $\sslash$ denotes the homotopy quotient. If we focus again on the case
  $n=1$, then we saw in \iref{Example}{ex:initialscM} that the first level of
  the initial $\scM$-algebra, $\scM\binom{}{1}$, coincides with the old
  $\scM|_1$-algebra $\coprod_{g}\frM_{g,1}$, whose group completion is
  accessible by the Madsen–Weiss theorem. Hence, we can replace the first factor
  $\Omega B\scM\binom{}{1}$ by $\Omega^\infty\mathbf{MTSO}(2)$. This proves the
  claim.
\end{proof}

\appendix
\section{General mapping spaces into \texorpdfstring{$\frM_{*,1}$}{M⁎₁}}
\label{apx:generalX}
In this first appendix, we briefly discuss a variation of
\iref{Theorem}{theo:main} for a generic parametrising topological space $X$,
highlighting the main enhancements that the proof requires.

We chose to restrict ourselves to $X=S^1$ throughout the main part of
the article because this special context already presents all the relevant
complexity of the problem, and we believe that it is more instructive to
consider the special case first.

\subsection{General parametrising spaces \texorpdfstring{$X$}{X}}

Our goal is to give a description of $\Omega B(\on{map}(X,\frM_{*,1}))$ as an
infinite loop space. We assume for simplicity that $X$ has the homotopy type of
a connected \acr{CW} complex.

Let $G\coloneqq \pi_1(X)$.  For all $g\ge0$ and $n\ge1$, we replace the space
$\frM_{g,n}$ with the homotopy equivalent space $\scM_{g,n}$, which admits an
action of the group $R_n$.  Since $\scM_{g,n}$ is aspherical, the space
$\on{map}(X,\scM_{g,n})$ is homotopy equivalent to $\on{map}(BG,\scM_{g,n})$.

More precisely, let $\on{Conj}(G\to\Gamma_{g,n})$ be the set of conjugacy
classes of homomorphisms $G \to \Gamma_{g,n}$, where two such homomorphisms
$\phi,\phi'$ are conjugate if there exists $\xi\in\Gamma_{g,n}$, represented by
$\Xi$, with $\phi'=\phi^\Xi$: then we can identify
\[\on{map}(X,\scM_{g,n})=\hspace*{-2mm}
  \coprod_{[\phi]\in\on{Conj}(G\to\Gamma_{g,n})}\on{map}(X,\scM_{g,n})(\phi)
  \simeq\hspace*{-2mm} \coprod_{[\phi]\in\on{Conj}(G\to\Gamma_{g,n})} BZ(\phi,\Gamma_{g,n}),\]
where $Z(\phi,\Gamma_{g,n})$ denotes the centraliser of the image of $\phi$ in
$\Gamma_{g,n}$.

\begin{defi}
 \label{defi:deltairreduciblegeneralised}
 A homomorphism $\phi\colon G\to\Gamma_{g,n}$ is called
 \emph{$\del$-irreducible} if it is not the trivial representation in
 $\Gamma_{0,1}$ and there is no isotopy class of an essential arc
 $\alpha\subset\Sigma_{g,n}$ which is fixed by the entire image of $\phi$. Being
 $\del$-irreducible is a conjugacy-invariant property of representations.  We
 denote by
 \[\frC_{g,n}(X) \coloneqq
   \coprod_{\substack{[\varphi]\in\on{Conj}(G\to\Gamma_{g,n})\\\del\text{-irreducible}}}
   \on{map}(X,\scM_{g,n})(\phi).\]
\end{defi}

For all $n\ge1$ there is a natural action of the group $R_n$ on $\frC_{g,n}(X)$,
and the analogue of \iref{Theorem}{theo:main} is the following:

\begin{theo}
 \label{theo:maingeneralised}
 In the above setting and with the above definitions, there is an equivalence of
 loop spaces
 \[\Omega B(\on{map}(X,\frM_{*,1}))\simeq \Omega^\infty\mathbf{MTSO}(2)
    \times \Omega^\infty\Sigma^\infty_+\coprod_{n\ge 1}
    \coprod_{g\ge 0} \frC_{g,n}(X) \qq R_n.\]
\end{theo}

In order to prove \iref{Theorem}{theo:maingeneralised}, the main obstacle is, at
the very beginning, to generalise \iref{Proposition}{prop:Alexandermethod} from
the context of a single diffeomorphism $\Phi$, representing a single mapping
class $\phi\in\Gamma(\caS,\del\caS)$, to the context of a representation
$\phi\colon G\to\Gamma(\caS,\del\caS)$. This is done as follows:

\begin{prop}
 \label{prop:Alexandermethodgeneralised}
 Let $\alpha_0,\dots,\alpha_k,\beta$ be a collection of essential arcs in a
 connected surface with non-empty boundary $\caS$ satisfying the two properties
 listed in \iref{Proposition}{prop:Alexandermethod}. Let $U$ be a small
 neighbourhood of
 $\alpha_0\cup\dots\cup\alpha_k\cup\beta\cup\del\caS\subset\caS$ in $\caS$, let
 \mbox{$\Diff(\caS,U)\subset\Diff(\caS,\del\caS)$} be the group of diffeomorphisms of
 $\caS$ fixing $U$ pointwise, and let $\Gamma(\caS,U)$ be the associated mapping
 class group.
  
 Then the canonical homomorphism of groups
 $\Gamma(\caS,U)\to\Gamma(\caS,\del\caS)$, induced by the inclusion
 $\Diff(\caS,U)\subset\Diff(\caS,\del\caS)$, is injective. Moreover, a
 homomorphism \mbox{$\phi\colon G\to \Gamma(\caS,\del\caS)$} lifts to $\Gamma(\caS,U)$
 if and only if for all $\gamma\in G$ the mapping class
 $\phi(\gamma)\in\Gamma(\caS,\del\caS)$ fixes up to isotopy each of the arcs
 $\alpha_0,\dots,\alpha_k$ and $\beta$.
\end{prop}
\begin{proof}
  Note that $U$ is a connected surface with at least three boundary components,
  and in particular $U$ is neither a disc nor an annulus. We can then apply
  \cite[Thm.\,3.18]{FarbMargalit} (see also the remark after the cited Theorem)
  and conclude that the homomorphism of mapping class groups
  $\Gamma(\caS\setminus U,\del(\caS\setminus U))\to\Gamma(\caS,\del\caS)$ is
  injective. Composing with the natural isomorphism
  $\Gamma(\caS\setminus U,\del(\caS\setminus U))\cong \Gamma(\caS,U)$ we obtain
  precisely the canonical homomorphism of groups
  $\Gamma(\caS,U)\to\Gamma(\caS,\del\caS)$, which hence is injective.

  For the second claim, the injectivity of
  $\Gamma(\caS\setminus U,\del(\caS\setminus U))\to\Gamma(\caS,U)$ implies that
  $\phi\colon G\to \Gamma(\caS,\del\caS)$ lifts to $\Gamma(\caS,U)$ if and only
  if, for all $\gamma\in G$, the mapping class
  $\phi(\gamma)\in\Gamma(\caS,\del\caS)$ lies in the image of the canonical
  homomorphism, and \iref{Proposition}{prop:Alexandermethod} ensures that this
  is equivalent to requiring, for all $\gamma\in G$, that the mapping class
  $\phi(\gamma)$ fixes up to isotopy each of the arcs $\alpha_0,\dots,\alpha_k$,
  and $\beta$.
\end{proof}

For a homomorphism $\phi\colon G\to\Gamma(\caS,\del\caS)$ we define the
\emph{fixed arc complex} of $\phi$ as in
\iref{Definition}{defi:fixedarccomplex}, but we require vertices to be isotopy
classes of arcs fixed by \emph{every} mapping class in the image of $\phi$: in
fact the fixed arc complex of $\phi$ is the intersection of the arc complexes of
$\phi(\gamma)$ for $\gamma$ ranging in $G$, where we consider all arc complexes
as simplicial subcomplexes of the arc complex of $\one\in\Gamma(\caS,\del\caS)$.

The bound on the dimension of the fixed arc complex of $\phi$ in terms of
$\chi(\caS)$ is proved in the same way. The white–yellow decomposition of $\caS$
along the cut locus is defined in the same way as in
\iref{Construction}{constr:cutlocus}: by choosing a maximal simplex in the
fixed-arc complex of the homomorphism $\phi$. The proof of
\iref{Lemma}{lem:cutlocus} can be extended to this generalised context as
follows:

\begin{lem}
 \label{lem:cutlocusgeneralised}
 Let $\phi\colon G\to\Gamma(\caS,\del\caS)$, let $\alpha_0,\dots,\alpha_k$ be
 arcs representing a maximal simplex in the cut locus of $\phi$, and let $\beta$
 be another arc fixed up to isotopy by $\phi$; then $\beta$ can be isotoped
 relative to endpoints to an arc lying in a small neighbourhood $U$ of
 $\alpha_0\cup\dots\cup\alpha_k\cup\del\caS$.
\end{lem}
\begin{proof}
  We assume without loss of generality that $\beta$ is in minimal position with
  $\alpha_0,\dots,\alpha_k$.  By
  Proposition~\ref{prop:Alexandermethodgeneralised} we can lift $\phi$ to a
  homomorphism \mbox{$\tilde\phi\colon G\to\Gamma(\caS,U')$}, where $U'$ is a small
  neighbourhood of $\alpha_0\cup\dotsb\cup\alpha_k\cup\beta\cup\del \caS$.  For
  each $\gamma\in G$ we can thus represent $\tilde\phi(\gamma)$ by a
  diffeomorphism $\Phi_{\gamma}\colon\caS\to\caS$ fixing $U$ pointwise.\looseness-1

  The rest of the proof is the same as for \iref{Lemma}{lem:cutlocus}: in
  particular, we use the representatives $\Phi_{\gamma}$ to check that $\beta'$
  and $\beta''$ are fixed up to isotopy by $\phi(\gamma)$ for all $\gamma\in G$.
\end{proof}
The proof of \iref{Proposition}{prop:cutlocus} can be repeated word by word in
the generalised context, and the analogue of \iref{Lemma}{lem:cutlocusequiv} is
the following:
\begin{lem}
 \label{lem:cutlocusequivgeneralised}
 Let $\psi\in\Gamma(\caS,\del\caS)$ be a mapping class, and let $\Psi$ be a
 diffeomorphism representing $\psi$. Moreover, let
 $\phi\colon G\to\Gamma(\caS,\del\caS)$ be a group homomorphism and let
 $[c_1,\dots,c_h]$ be the cut locus of $\phi$.

 Then $[\Psi(c_1),\dots,\Psi(c_h)]$ is the cut locus of the conjugate
 $\psi\phi\psi^{-1}$, which is defined by
 $(\psi\phi\psi^{-1})(\gamma) \coloneqq \psi\cdot\phi(\gamma)\cdot\psi^{-1}$.
 In particular, if the image of $\phi$ commutes with $\psi$, then $\psi$
 preserves the cut locus of $\phi$ as an unordered collection of isotopy classes
 of oriented simple closed curves.
\end{lem}
\begin{proof}
  The proof is the same as for \iref{Lemma}{lem:cutlocusequiv}, with the
  exception that we first lift $\phi$ to a homomorphism
  $\tilde\phi\colon G\to \Gamma(\caS,U)$, where $U$ is a small neighbourhood of
  $\alpha_0\cup\dots\cup\alpha_k\cup\del\caS$, and then we compare this morphism
  with the conjugate
  \mbox{$\psi\tilde\phi\psi^{-1}\colon G\to\Gamma(\caS,\Psi(U))$.}
\end{proof}

This shows that, given a homomorphism $\phi\colon G\to\Gamma(\caS,\del\caS)$, we
can associate with it a cut locus $c_1,\dots,c_h$ separating $\caS$ into a white
region $W$ and a yellow region $Y$. We fix a parametrisation by $S^1$ of each
curve in the cut locus as in \iref{Subsection}{subsec:similar}.

We can also lift uniquely $\phi$ to a morphism
$\tilde\phi\colon G\to\Gamma(\caS,W)$. For each component $P\subset Y$, we then
have a morphism $\phi_P\colon G\to \Gamma(P,\del P)$ obtained by composing
$\tilde\phi$ with the restriction $\Gamma(\caS,W)\to\Gamma(P,\del P)$.
Mimicking \iref{Definition}{defi:similar}, we say that two path components $P$
and $P'$ of $Y$ are \emph{similar for $\phi$} if there is a a diffeomorphism
$\Xi\colon P\to P'$ preserving the boundary parametrisations and conjugating the
homomorphism $\phi_P$ to $\phi_{P'}$.  We write
\mbox{$Y=\coprod_{i=1}^r\coprod_{j=1}^{s_i}Y_{i,j}$} as in
\iref{Notation}{nota:Yi}, and introduce morphisms
\mbox{$\bar\phi_{i,j}\colon G\to\Gamma_{g_i,n_i}$} conjugate to $\phi_{Y_{i,j}}$
in an analogue way.

We denote by $\phi_Y\colon G\to\Gamma(Y,\del Y)$ the composition of $\tilde\phi$
and the restriction map \mbox{$\Gamma(\caS,W)\to\Gamma(Y,\del Y)$,} and denote
by $Z(\phi_Y)\subset\Gamma(Y)$ the subgroup of elements that commute with the
entire image of $\phi_Y$. \iref{Definition}{defi:tZdelphi} and
\iref{Lemma}{lem:epsilon} then carry over to the generalised context.

In \iref{Section}{sec:centralisers}, one only has to replace
‘$\phi\in \Gamma(\caS,\partial\caS)$’ by ‘group homomorphism
\mbox{$\phi\colon G\to \Gamma(\caS,\partial\caS)$}’, as well as ‘representative $\Phi$
of $\phi$ fixing $U$’ by ‘the unique lift of $\phi$ to a homomorphism
$\tilde\phi\colon G\to \Gamma(\caS,U)$’.

Finally, the discussion of \iref{Section}{sec:LM1asFreeAlg} generalises
straightforwardly to the context of a generic parametrising space $X$, leading
to \iref{Theorem}{theo:maingeneralised}.

\subsection{Functoriality in \texorpdfstring{$X$}{X}}

Both sides of the equivalence in \iref{Theorem}{theo:maingeneralised} depend on
a space $X$, and the left-hand side can be regarded as an enriched,
contravariant functor from the category of (connected) topological spaces to the
category of loop spaces. In fact each space $\Omega B(\on{map}(X,\frM_{*,1}))$
admits a natural infinite loop space structure, and for a map of spaces
$f\colon X\to X'$, the corresponding map
\[\Omega B(\on{map}(f,\frM_{*,1}))\colon
  \Omega B(\on{map}(X',\frM_{*,1}))\to \Omega B(\on{map}(X,\frM_{*,1}))\]
is defined as the group completion of a map of $\scM|_1$-algebras, and hence a
map of infinite loop spaces. We can thus consider
$\Omega B(\on{map}(-,\frM_{*,1}))$ as a functor from connected topological
spaces to infinite loop spaces.

Now it would be interesting to describe this functor solely in terms of the
right-hand side, which means: to upgrade, in a non-tautological way, the assignment\linebreak
\mbox{$X\mapsto \Omega^\infty\mathbf{MTSO}(2) \times
\Omega^\infty\Sigma^\infty_+\coprod_{n\ge 1} \coprod_{g\ge 0} \frC_{g,n}(X) \qq
R_n$} to an enriched functor from topological spaces to infinite loop spaces,
such that the equivalences given by \iref{Theorem}{theo:maingeneralised}, for
varying $X$, assemble into a natural equivalence of functors. We will content
ourselves with a weaker, but explicit result about functoriality.

\begin{defi}
  \label{defi:deltafaithful}
  A map of connected topological spaces $f\colon X\to X'$ induces a map of
  fundamental groups $f_*\colon\pi_1(X)\to\pi_1(X')$, which is defined up to
  conjugation and which, in turn, allows us to transform any homomorphism
  $\phi\colon\pi_1(X')\to\Gamma_{g,n}$ into a homomorphism
  $f^*\phi=\phi\circ f_*\colon\pi_1(X)\to\Gamma_{g,n}$, for all $g\ge0$ and
  $n\ge1$.
  
  We say that $f$ is \emph{$\del$-faithful} if $f^*\phi$ is $\del$-irreducible
  whenever $\phi$ is $\del$-irreducible, for all $g\ge0$ and $n\ge1$.  We denote
  by $\mathbf{Top}^\del$ the (topologically enriched) category of connected
  topological spaces that are homotopy equivalent to a \acr{CW} complex, with morphisms
  being $\del$-faithful continuous maps.
\end{defi}
Note that being $\del$-faithful is a homotopy-invariant property of continuous
maps, i.e.\ morphism spaces in $\mathbf{Top}^\del$ are obtained by selecting
unions of con\-nec\-ted components from the morphism spaces in $\mathbf{Top}$.
Recall \iref{Definition}{defi:deltairreduciblegeneralised}: if $f\colon X\to X'$
is a $\del$-faithful map, then we obtain for each $g\ge0$ and $n\ge1$ a map
\[\frC_{g,n}(f)\colon\frC_{g,n}(X')\to\frC_{g,n}(X);\]
Thus we can consider $\frC_{g,n}$ as a contravariant functor from
$\mathbf{Top}^\del$ to $\mathbf{Top}$, for all $g\ge0$ and $n\ge1$.
Putting together all values of $g$ and $n$, and taking the actions of the groups
$R_n$ into account, we can consider $\frC$ as a contravariant functor from
$\mathbf{Top}^\del$ to $\Alg{\bm{R}}$. We can then repeat the proof of
\iref{Theorem}{theo:isfreealg} and obtain an equivalence of (contravariant,
enriched) functors
\[F^\scM_{\bm{R}}(\frC(-))_1
  \simeq \on{map}(-,\scM_{*,1})\colon \mathbf{Top}^\del\to \Alg{\scM|_1}.\]
Composing with the group completion functor, we then obtain the following
enriched version of \iref{Theorem}{theo:maingeneralised}.

\begin{theo}
 \label{theo:maingeneralisedfunctorial}
 Let $\Omega\text{-}\mathbf{Top}$ and $\Omega^\infty\text{-}\mathbf{Top}$ denote
 the (enriched) categories of loop spaces and infinite loop spaces respectively. 
 Then there is is a weak equivalence between the following two enriched
 contravariant functors from $\mathbf{Top}^\del$ to
 $\Omega\text{-}\mathbf{Top}$:
 \begin{itemize}
  \item the functor $\Omega B(\on{map}(-,\frM_{*,1}))$;
  \item the composition of the contravariant functor
    \[\Omega^\infty\mathbf{MTSO}(2)
      \times \Omega^\infty\Sigma^\infty_+\coprod_{n\ge 1}
      \coprod_{g\ge 0} \frC_{g,n}(-) \qq R_n\]
    from $\mathbf{Top}^\del$ to $\Omega^\infty\text{-}\mathbf{Top}$ and the
    covariant, forgetful functor from $\Omega^\infty\text{-}\mathbf{Top}$ to
    $\Omega\text{-}\mathbf{Top}$.
  \end{itemize}
  The weak equivalence of functors assigns to $X\in \mathbf{Top}^\del$ the weak
  equivalence of loop spaces given by \iref{Theorem}{theo:maingeneralised}.
\end{theo}

\begin{expl}
  We briefly discuss an example showing why
  \iref{Theorem}{theo:maingeneralisedfunctorial} cannot be generalised to an
  analogue statement concerning the entire category $\mathbf{Top}$ of connected
  topological spaces with the homotopy type of a \acr{CW} complex, together with
  \emph{all} continuous maps: let $X$ be a connected \acr{CW} complex and let
  $f\colon*\to X$ be the inclusion of a point. We note that
  $\coprod_{n\ge 1}\coprod_{g\ge 0} \frC_{g,n}(*)$ is the empty space, whence
  $\Omega^\infty\Sigma^\infty_+\coprod_{n\ge 1} \coprod_{g\ge 0} \frC_{g,n}(*)
  \qq R_n$ is just a point A straightforward generalisation of
  \iref{Theorem}{theo:maingeneralisedfunctorial} would predict that the
  following square commutes up to homotopy, where the horizontal isomorphisms
  are given by \iref{Theorem}{theo:maingeneralised}:
  \[\begin{tikzcd}
      \Omega B(\on{map}(X,\frM_{*,1}))\ar[r,"\simeq"]
      \ar[d,"{\Omega B(\on{map}(f,\frM_{*,1}))}"] &
      \Omega^\infty\mathbf{MTSO}(2)\times \Omega^\infty\Sigma^\infty_+\coprod_{n\ge 1}
      \coprod_{g\ge 0} \frC_{g,n}(X) \qq R_n
      \ar[d,"{\mathrm{pr}_{\Omega^\infty\mathbf{MTSO(2)}}}"]\\
      \Omega B(\on{map}(*,\frM_{*,1}))\ar[r,"\simeq"]&
      \Omega^\infty\mathbf{MTSO}(2).
    \end{tikzcd}\]
  However, we can apply the functor $\pi_0(-)$, taking values in groups, as all
  terms in the previous diagram are loop spaces.  Applying $\pi_0$ to the top
  row, we obtain the group
  $\bbZ\oplus(\bigoplus_{n\ge1}\bigoplus_{g\ge0}\bbZ^{\oplus\pi_0(\frC_{g,n}(X))})$,
  whereas applying $\pi_0$ to the bottom row, we obtain the group $\bbZ$.

  After applying $\pi_0$, of the left vertical map sends the first summand
  $\bbZ$ isomorphically onto $\bbZ$, and the standard generator of each further
  summand $\bbZ^{\oplus\pi_0(\frC_{g,n}(X))}$ to $g+n-1\in\bbZ$, while the right
  vertical map sends the first summand $\bbZ$ isomorphically onto $\bbZ$, and
  each further summand $\bbZ^{\oplus\pi_0(\frC_{g,n}(X))}$ constantly to $0$.
  
  Note that it suffices to take $X=S^1$ to have
  $\pi_0(\frC_{g,n}(X))\neq\emptyset$ for some choice of $g$ and $n$, and thus
  to ensure that the previous example is not void: see
  \iref{Example}{expl:boundarydehn}.
\end{expl}

As an application of \iref{Theorem}{theo:maingeneralisedfunctorial}, let
$k\in\bZ\setminus\set{0}$ and consider the power map $(-)^k\colon S^1\to S^1$,
which induces multiplication by $k$ on the fundamental group
$\pi_1(S^1)\cong \bZ$. The following lemma proves that $(-)^k$ is
$\del$-faithful; hence it induces a self map of the infinite loop space
$\Omega^\infty\mathbf{MTSO}(2) \times \Omega^\infty\Sigma^\infty_+\coprod_{n\ge
  1} \coprod_{g\ge 0} \mathfrak C_{g,n} \qq R_n$ that can be described using
\iref{Theorem}{theo:maingeneralisedfunctorial}.

\begin{lem}
 \label{lem:powerS1}
 Let $\caS$ be a surface of type $\Sigma_{g,n}$ for some $g\ge0$ and
 $n\ge1$. Let $\phi\in\Gamma(\caS,\del\caS)$ be a mapping class, let
 $k\in\bZ\setminus\set{0}$, and let $\bar\alpha$ be an essential arc in
 $\caS$. Suppose that the isotopy class $[\bar\alpha]$ of $\bar\alpha$ relative
 to its endpoints is in the fixed-arc complex of $\phi^k$. Then $[\bar\alpha]$
 is in the fixed-arc complex of $\phi$.
\end{lem}
\begin{proof}
  Since the fixed-arc complexes of $\phi$ and $\phi^{-1}$ are equal, we can
  assume without loss of generality that $k\ge1$. The case $k=1$ is
  tautological, so we henceforth assume $k\ge2$.
  Let $p,q\in\del\caS$ be the endpoints of $\bar\alpha$ and consider the set
  $\mathrm{Arc}(p,q)$ of isotopy classes $[\alpha]$ of essential arcs $\alpha$
  in $\caS$ from $p$ to $q$. An essential arc $\alpha$ is a map
  $\alpha\colon [0,1]\to\caS$, but we will abuse notation and denote by $\alpha$
  also the image of this map, which is a subset of $\caS$.
 
  Let $\alpha$ and $\beta$ be two essential arcs from $p$ to $q$. We say that
  $\alpha$ and $\alpha'$ are \emph{transverse} if they have linearly independent
  velocities at each intersection point, including the endpoints $p$ and
  $q$. Moreover, we say that $\alpha$ and $\alpha'$ are in \emph{minimal
    position away from the endpoints} if the number of intersection points of
  $\alpha$ and $\alpha'$, excluding $p$ and $q$, is minimal among all pairs of
  transverse arcs $\alpha'$ and $\beta'$ with $\alpha\sim\alpha'$ and
  $\beta\sim\beta'$. The bigon criterion applies: $\alpha$ and $\beta$ are in
  minimal position if and only if they do not form bigons.
  In particular, if $\alpha$ and $\beta$ are non-isotopic arcs from $p$ to $q$
  in minimal position away from the endpoints, then we can compare the
  velocities of $\alpha$ and $\beta$ at $p$: these are linearly independent
  vectors in $T_p\caS$ pointing inside $\caS$: we say that \emph{$\alpha$ is on
    right of $\beta$} if the sequence
  $(\frac{d}{dt}\alpha(0),\frac{d}{dt}\beta(0))$ is a positive basis of
  $T_p\caS$.
 
  The Alexander method guarantees the following: let
  $\alpha,\alpha',\beta,\beta'$ be four essential arcs in $\caS$ from $p$ to
  $q$, with $\alpha\sim\alpha'\not\sim\beta\sim\beta'$, and suppose that
  $\alpha$ and $\beta$, as well as $\alpha'$ and $\beta'$, are in minimal
  position away from their endpoints. Then there is an isotopy of $\caS$
  bringing $\alpha\cup\beta$ to $\alpha'\cup\beta'$. In particular $\alpha$ is
  on right of $\beta$ if and only if $\alpha'$ is on right of $\beta'$.
 
  We can now put a total order on $\mathrm{Arc}(p,q)$: for two distinct isotopy
  classes $[\alpha]$ and $[\beta]$ we say that $[\alpha]\prec[\beta]$ if, for
  any two representatives $\alpha$ and $\beta$ which are in minimal position
  away from the endpoints, $\alpha$ is on right of $\beta$.
  The action of $\Gamma(\caS,\del\caS)$ on $\mathrm{Arc}(p,q)$ preserves the
  total order $\prec$: for this, let $\Phi$ be a diffeomorphism representing
  $\phi\in \Gamma(\caS,\del\caS)$ and assume that $\Phi$ fixes a neighbourhood
  of $p\in\del\caS$ pointwise. Moreover, let $\alpha$ and $\beta$ represent
  classes $[\alpha],[\beta]\in\mathrm{Arc}(p,q)$, and assume that $\alpha$ and
  $\beta$ are in minimal position away from the endpoints and that $\alpha$ is
  on right of $\beta$, thus witnessing $[\alpha]\prec[\beta]$: then
  $\Phi(\alpha)$ and $\Phi(\beta)$ are also in minimal position away from the
  endpoints, and $\Phi(\alpha)$ is on right of $\Phi(\beta)$, witnessing that
  $\phi([\alpha])=[\Phi(\alpha)]\prec[\Phi(\beta)]=\phi([\beta])$.
 
  Let now $[\bar\alpha]$ be as in the hypothesis of the theorem, and assume for
  the sake of contradiction that
  $[\bar\alpha]\neq\phi([\bar\alpha])\in \mathrm{Arc}(p,q)$: then without loss
  of generality we can assume $[\bar\alpha]\prec\phi([\bar\alpha])$. We then
  have a chain of inequalities
  \[[\bar\alpha]\prec\phi([\bar\alpha])\prec\phi^2([\bar\alpha])\prec\dots
    \prec\phi^{k-1}([\bar\alpha])\prec\phi^k([\bar\alpha]),\]
  hence $[\bar\alpha]\prec\phi^k([\bar\alpha])$, contradicting the hypothesis
  $[\bar\alpha]=\phi^k([\bar\alpha])$. 
\end{proof}

\section{Braid groups, symmetric groups, and free loops}
\label{apx:braidSpaces}
In this second appendix, we sketch two results that are parallel to the
identification in \iref{Theorem}{theo:main}.

The first pertains to automorphisms of surfaces that have no genus but instead
punctures; we thus replace mapping class groups by braid groups. A particular
model of the corresponding classifying space is given by the collection of
unordered configuration spaces of the 2-dimensional disc.

In analogy to that, the collection of configuration spaces of the
$\infty$-dimensio\-nal disc is a model for the classifying space of symmetric
groups. We outline the corresponding result for that setting.

\subsection{Braid groups}

First recall that the $r$\textsuperscript{th} braid group $\BrG_r$ is
 the fundamental group of the unordered configuration space
$C_r(\mathring{D}^2) \coloneqq \tilde{C}_r(\mathring{D}^2)/\frS_r$ of the
2-dimensional disc, where
\[\tilde{C}_r(\mathring{D}^2) \coloneqq \set{(p_1, \dots, p_r)\in
    \mathring{D}^2;\, p_i \neq p_j \text{ for all } i \neq j}.\]
The contractible topological group $\Diff_{\partial}(D^2)$ of diffeomorphisms of
$D^2$ that are the identity in a neighbourhood of the boundary acts on
$C_r(\mathring{D}^2)$; the stabiliser of a point
$P \coloneqq \{p_1, \dots, p_r\} \in \UConf_r(D^2)$ with respect to this this
action is the subgroup $\Diff_{\partial}(D^2,P)$ of diffeomorphisms that
preserve $P$ as a set. The induced long exact sequence on homotopy groups of the
resulting fibration
\[\Diff_{\partial}(D^2,P) \to \Diff_{\partial}(D^2)
  \to C_r(\mathring{D}^2) \]
thus implies that $\BrG_r \cong \pi_0 \Diff_{\partial}(D^2,P)$, the mapping
class group of a genus $0$ surface with $r$ (unordered) punctures. A classical
argument by Fadell and Neuwirth \cite{FadellNeuwirth} shows that the
configuration spaces $\tilde{C}_r(\mathring{D}^2)$ and $C_r(\mathring{D}^2)$ are
aspherical, which implies the well-known result
$C_r(\mathring{D}^2) \simeq B \BrG_r$.

We remark that Alexander's theory of arc systems developed in detail for
surfaces in \iref{Section}{sec:arcSystems}, \iref{Definitions}{defn:arc},
\ref{defi:fixedarccomplex}, \ref{defi:cutlocus},
\iref{Propositions}{prop:Alexandermethod} and \ref{prop:cutlocus},
\iref{Construction}{constr:cutlocus}, and \iref{Lemma}{lem:cutlocus}, has a
canonical analogue if the surfaces with boundary considered before are replaced
by a punctured disc: now the arcs are required to start and end in
$\partial D^2 = S^1$ and their interior must lie in
$\mathring{D}^2 \backslash \{p_1, \dots, p_r\}$. Most importantly, we can also
define the notion of irreducibility of elements of braid groups as follows.

\begin{defi}
  \label{defi:Br_irreducible}
  \begin{enumerate}
  \item A braid $\gamma \in \BrG_r$ is \emph{reducible} if there exists an
    essential arc in $\mathring{D}^2 \backslash \{p_1, \dots, p_r\}$ that is
    fixed by $\gamma$; otherwise it is called \emph{irreducible}. In other
    words, a braid $\gamma$ is reducible if it is conjugate to a block sum of
    braids; see \iref{Figure}{fig:braidred}.
  \item Let
    $\mathfrak I_r \subseteq \Lambda C_r(\mathring{D}^2)=\Lambda B\mathrm{Br}_r$
    denote the subspace of free loops whose corresponding conjugacy classes in
    $\BrG_r$ consist of irreducible elements.
  \end{enumerate}
\end{defi}

\begin{figure}
  \centering
  \begin{tikzpicture}[yscale=.9]  
    \draw (1,3) to[out=-90,in=90] (2,1.5) to[out=-90,in=90] (1,0);
    \draw[white,line width=3mm](2,3) to[out=-90,in=90] (0,0);  
    \draw (2,3) to[out=-90,in=90] (0,0);
    \draw[white,line width=3mm] (0,3) to[out=-90,in=90] (2,0);
    \draw (0,3) to[out=-90,in=90] (2,0);
    \draw (7,3) -- (7,0);
    \draw (6,0) -- (6,3);
    \draw[white,line width=3mm] (5,3) to[out=-90,in=90] (7.25,1.5) to[out=-90,in=90] (5,0);
    \draw (5,3) to[out=-90,in=90] (7.25,1.5) to[out=-90,in=90] (5,0);
    \draw[white,line width=3mm] (7,1.5) -- (7,0);
    \draw[white,line width=3mm] (6,1.5) -- (6,0);
    \draw (7,1.5) -- (7,0);
    \draw (6,1.5) -- (6,0);
    \node at (0,3) {\tiny $\bullet$};
    \node at (0,0) {\tiny $\bullet$};
    \node at (1,3) {\tiny $\bullet$};
    \node at (1,0) {\tiny $\bullet$};
    \node at (2,3) {\tiny $\bullet$};
    \node at (2,0) {\tiny $\bullet$};
    \node at (7,3) {\tiny $\bullet$};
    \node at (7,0) {\tiny $\bullet$};
    \node at (5,3) {\tiny $\bullet$};
    \node at (5,0) {\tiny $\bullet$};
    \node at (6,3) {\tiny $\bullet$};
    \node at (6,0) {\tiny $\bullet$};
    \node at (0,3.3) {\tiny $1$};
    \node at (1,3.3) {\tiny $2$};
    \node at (2,3.3) {\tiny $3$};
    \node at (5,3.3) {\tiny $1$};
    \node at (6,3.3) {\tiny $2$};
    \node at (7,3.3) {\tiny $3$};
    \node at (0,-.3) {\tiny $1$};
    \node at (1,-.3) {\tiny $2$};
    \node at (2,-.3) {\tiny $3$};
    \node at (5,-.3) {\tiny $1$};
    \node at (6,-.3) {\tiny $2$};
    \node at (7,-.3) {\tiny $3$};  
  \end{tikzpicture}
  \caption{The left braid is reducible as it is conjugate to a block sum
    of two braids, whereas the second one is not.}\label{fig:braidred}
\end{figure}
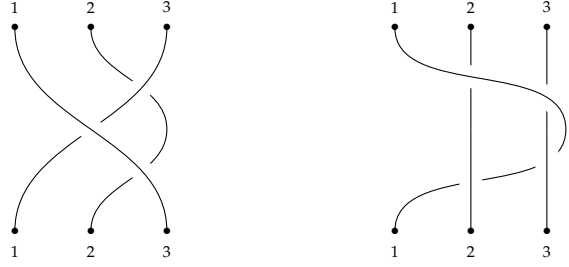

The genus-0 and colour-1 part of $\scM$, which coincides with the \emph{framed}
little 2-discs operad $\smash{\scD_2^{\text{fr}}}$ and acts on the union
$\Lambda \coprod_{r\ge 0}C_r(\mathring{D}^2)$. The Lie group $R_1\cong S^1$ acts
on $\frI\coloneqq \coprod_{r\ge 1} \frI_r$, and analogous to
\iref{Theorem}{theo:isfreealg} one obtains that
\[\Lambda \coprod_{r\ge 0}C_r(\mathring{D}^2) \simeq
  F_{S^1}^{\smash{\scD_2^{\mathrm{fr}}}}(\frI).\]

The free $\smash{\scD_2^{\mathrm{fr}}}$-algebra over a given $S^1$-space is, as
a $\scD_2$-algebra, equivalent to the free $\scD_2$-algebra over the underlying
space without the $S^1$-action as the operation spaces $\scD^{\mathrm{fr}}_2(r)$
and $\scD_2(r)\times (S^1)^r$ are equivalent as free right $(S^1)^r$-spaces,
whence the coend construction just cancels the toric factor. We therefore
obtain, after group completion, the following identification.

\begin{theo}\label{theo:app1} There is a homotopy equivalence
  \[\Omega B\Lambda \coprod_{n \ge 1} B \BrG_r \simeq
    \Omega^2 \Sigma_+^2 \coprod_{r\ge 1}
    \coprod_{\substack{[\gamma] \in \on{Conj}(\BrG_r)\\ \text{irreducible}}}
    B Z(\gamma,\BrG_r).\]
\end{theo}  

\subsection{Symmetric groups}

We conclude by considering free loop spaces of configuration spaces of an
infinite-dimensional disc $\mathring{D}^{\infty}$. These are classifying spaces
of symmetric groups, i.e. $\smash{C_r(\mathring{D}^{\infty}) \simeq
  B\fS_r}$. The analogous notion of irreducibility is even simpler: an element
of $\fS_r$, i.e.\ a permutation, is irreducible if and only if it comprises a
single cycle. In this case its centraliser is cyclic of order $r$, and as above
one obtains that
\[\Lambda \coprod_{r\ge 0} B\fS_r \simeq F^{\scD_{\infty}} \coprod_{k\ge 1} B(\mathbb Z/k).\]
After group completion we obtain the analogue of
\iref{Theorem}{theo:app1} for symmetric groups \cite[Cor.\,4.32]{JR_thesis}.

\begin{theo}\label{theo:app2} There is a homotopy equivalence
  \[\Omega B\Lambda \coprod_{r\ge 0}  B\fS_r \simeq
    \Omega^{\infty} \Sigma_+^{\infty} \coprod_{k\ge 1} B(\mathbb Z/k).\]
\end{theo}

\printbibliography[heading=bibintoc]

\begin{addr}
  \auth
  {4.7cm}
  {Andrea Bianchi}
  {https://andreabianchi.sites.ku.dk/}
  {Department of Mathematical Sciences\\
    University of Copenhagen\\
    Universitetsparken 5\\
    DK-2100 Copenhagen, Denmark}
  {anbi@math.ku.dk}
\hspace*{1.5mm}
  \auth
  {3.7cm}
  {Florian Kranhold}
  {https://www.math.uni-bonn.de/people/kranhold/}
  {Max Planck Institute\\
    \hspace*{2mm}for Mathematics, Bonn\\
    Vivatsgasse 7\\
    53111 Bonn, Germany}
  {fkranhold@mpim-bonn.mpg.de}
\hspace*{2.1mm}
  \auth
  {4cm}
  {Jens Reinhold}
  {https://sites.google.com/view/jens-reinhold/home}
  {Mathematics Münster\\
  University of Münster\\
    Orléans-Ring 10\\
    48149 Münster, Germany}
  {jens.reinhold@uni-muenster.de jens.reinhold@posteo.de}
  \hspace*{2mm}
\end{addr}

\end{document}